\documentclass[a4paper]{amsart}
\usepackage{url}
\usepackage[all]{xy}
\usepackage[mathscr]{eucal}
\usepackage{amsmath,amssymb,amsfonts}
\newtheorem{theorem}{Theorem}[section]
\newtheorem{lemma}[theorem]{Lemma}
\newtheorem{corollary}[theorem]{Corollary}
\newtheorem{example}[theorem]{Example}

\newtheorem{proposition}[theorem]{Proposition}
\newtheorem{remark}[theorem]{Remark}

\usepackage{stackengine} 
\usepackage{xcolor}
\usepackage[all]{xy}

\title{Higher dimensional generalizations of the Thompson groups}

\author{Mark V. Lawson}
\address{Mark V. Lawson, Department of Mathematics
and the
Maxwell Institute for Mathematical Sciences, 
Heriot-Watt University,
Riccarton,
Edinburgh EH14 4AS, 
UNITED KINGDOM}
\email{m.v.lawson@hw.ac.uk}

\author{Alina Vdovina}
\address{Alina Vdovina, 
School of Mathematics and Statistics,
Herschel Building,
Newcastle University,
Newcastle-upon-Tyne NE1 7RU,
UNITED KINGDOM}
\email{alina.vdovina@ncl.ac.uk}

\begin{document} 

\begin{abstract} 
Motivated by the goal of generalizing the Cuntz-Krieger algebras,
and making heavy use of the pioneering work of Robertson and Steger, 
Kumjian and Pask generalized the notion of a directed graph to what they termed a `higher rank graph'.
The $C^{\ast}$-algebras constructed from such higher rank graphs have proved to be highly interesting.
Higher rank graphs are, in fact, a class of cancellative categories
and so the one-vertex higher rank graphs are therefore a class of cancellative monoids
--- in this paper we call them $k$-monoids.
Finite direct products of free monoids belong to this family but there are many examples not of this form.
It is well-known that the classical Thompson groups arise from the free monoids,
and it is easy to show that Brin's higher dimensional Thompson groups arise from finite direct products of free monoids.
In this paper, we generalize these constructions and show how to construct a family of new groups with simple commutator subgroups from 
arbitrary aperiodic $k$-monoids thereby subsuming all the above examples.
Although we give simple direct constructions of our groups,
a key feature of our paper is that we also show that they arise as topological full groups of the usual \'etale groupoids associated with with $k$-monoids.
In this way, our groups are naturally associated with $C^{\ast}$-algebras.
The question of the mutual isomorphisms amongst this family of groups appears to be a delicate one and is not handled here.
\end{abstract}
\maketitle

\tableofcontents

\section{Introduction}

In this paper, we shall construct a class of new groups that includes not only the classical Thompson groups $G_{n,1}$ 
but also Brin's higher dimensional analogues $nV$ as special cases \cite{Brin}.
The building blocks of our groups are a class of 1-vertex higher rank graphs that we call `$k$-monoids'.
Examples of such monoids are the finite direct products of free monoids; such direct products are not themselves free
but have important properties in common with free monoids.

We state as a theorem what we achieve in this paper (and reassure the reader that all undefined terms will be defined).

\begin{theorem}\label{them:one} Let $S$ be an aperiodic $k$-monoid whose associated $k$ alphabets are all finite
and each have cardinality at least 2.
Then we construct a countable group $\mathscr{G}(S)$ as a subgroup of the group of self-homeomorphisms of the Cantor space
with simple commutator subgroup that arises as the topological full group of a Boolean  \'etale groupoid.
\end{theorem}

We generalize an approach that goes back to the work of Scott \cite{Scott} 
as reinterpreted by Birget \cite{Birget}.\footnote{The connection between our work and the most recent paper by Birget \cite{Birget2019} is described in Section~13.}
There, free monoids and the theory of prefix codes \cite{BP} are used 
to construct the groups $G_{n,1}$ (where $G_{2,1}$ is usually denoted by $V$).
This same approach was exploited in \cite{Lawson2007, Lawson2007b} but combined with inverse semigroup theory to define the same groups.
The unspoken hope of the latter two papers was that by generalizing free monoids in a suitable way, generalizations of the Thompson groups would follow.
John Fountain, in some unpublished notes \cite{Fountain} and a number of lectures,
did take a first step in realizing this aspiration but in our paper, we develop this idea much further.
A crucial step was the realization that 1-vertex higher rank graphs \cite{RS, KP, FMY, JS2014, JS2018} --- what we call `$k$-monoids' ---
can be viewed as natural generalizations of free monoids.
This paper is also closely allied with the categorical approach adopted by Spielberg \cite{JS2014, JS2018}
though, of course, we work with monoids rather than categories.

The paper consists of a further thirteen sections.
Section 2 reviews the very basic semigroup theory needed to read this paper,
whereas Section 3 reviews the properties of `$1$-vertex higher rank graphs' --- what we call `$k$-monoids' --- needed to read this paper.
These two sections are entirely preliminary.
Section 4 describes how to construct a Boolean space $S^{\infty}$ from a $k$-monoid $S$ and describes the properties of the action $S \times S^{\infty} \rightarrow S^{\infty}$.
This is the basic construction on which the rest of the paper depends.
In Section 5, we describe the inverse monoid $\mathsf{R}(S)$ that can be constructed from $S$.
The group $\mathscr{G}(S)$ that can be constructed from the $k$-monoid $S$ is defined in Section~6
but the definition reveals nothing about its properties.
In Section 7, we show that the group $\mathscr{G}(S)$ can also be defined in terms of what might be called `maximal generalized prefix codes'.
But it is in Section 8, that we are able to show that our group is a subgroup of the group of self-homeomorphisms
of the Cantor space. 
In Section 9, we construct from $\mathsf{R}(S)$ by means of a congruence, a Boolean inverse monoid $\mathsf{C}(S)$ and prove
that the group $\mathscr{G}(S)$ is its group of units.
It is in Section 10, that we are finally able to prove that $\mathscr{G}(S)$ is, in fact, a topological full group.
This section contains what we believe is a new approach to the standard \'etale groupoid associated with a $k$-monoid
since we observe that it is a groupoid of fractions.
In Section 11, we study the properties of what we term `rigid' $k$-monoids, 
the most natural generalizations of direct products of free monoids.
In Section~12, we study the structure of `generalized maximal prefix codes'.
Section~13 concentrates on constructing concrete examples of the groups introduced 
in this paper and making connections with the papers \cite{Brin, KMN, MN}. 
Section~14 wraps the paper up with some general, contextual results.\\

\noindent
{\bf Acknowledgements }The authors would like to thank Aidan Sims for his comments and corrections on a first draft of this paper. They would also like to thank an anonymous referee for their detailed comments and suggestions.
The second author would like to thank the Mathematical Picture Language Project at Harvard University.

\section{Some terminology and notation}

This section should be used as a reference when needed.

If $(P,\leq)$ is a poset and $X \subseteq P$ we write $X^{\downarrow} = \{y \in P \colon (\exists x \in X) \, y \leq x\}$.
If $X = X^{\downarrow}$ we say that $X$ is an {\em order ideal}.
If $X = \{x\}$ we write $x^{\downarrow}$ instead of $\{x\}^{\downarrow}$.

This paper introduces a class of groups but our approach is semigroup-theoretic.
We recall some key notions now.

A {\em zero} $0$ in a semigroup $S$ is an element $0$ such that $s0 = 0 = 0s$ for all $s \in S$.
A {\em congruence} $\rho$ defined on a semigroup $S$ is an equivalence relation
with the property that $a \, \rho \, b$ and $c \, \rho \, d$ imply that $ac \, \rho \, bd$.
If we denote the $\rho$-equivalence class containing $s$ by $[s]$ and the set of $\rho$-equivalence classes by $S/\rho$
then $S/\rho$ is a semigroup when we define $[s][t] = [st]$.
The natural map $S \rightarrow S/\rho$, given by $s \mapsto [s]$, is then a semigroup homomorphism.
If $\theta \colon S \rightarrow T$ is a homomorphism of semigroups then its {\em kernel}
is the congruence $\rho$ defined by $a \, \rho \, b$ if and only if $\theta (a) = \theta (b)$.
A congruence $\rho$ on a semigroup with zero is said to be {\em $0$-restricted} if 
$a \, \rho \, 0$ implies that $a = 0$.
A congruence $\rho$ is said to be {\em idempotent-pure} if $a\, \rho \, e$, where $e$ is an idempotent, implies that $a$ is an idempotent.

Most of the time we work with inverse semigroups which are abstractions of semigroups of partial bijections just as
groups are abstractions of semigroups of bijections.
Our reference to inverse semigroup theory is \cite{Lawson1998} but we recall some basic definitions here.
An {\em inverse semigroup} is a semigroup $S$ in which for each element $a$ there is a unique element $a^{-1}$, called the {\em inverse} of $a$, such that
$a = aa^{-1}a$ and $a^{-1} = a^{-1}aa^{-1}$.
It is usual to write $\mathbf{d}(a) = a^{-1}a$ and $\mathbf{r}(a) = aa^{-1}$, both of which are idempotents.
The set of idempotents $\mathsf{E}(S)$ of an inverse semigroup is a commutative subsemigroup.
For this reason, it is usually referred to as the {\em semilattice of idempotents} of $S$.
Our inverse semigroups will always have a zero.
An {\em inverse subsemigroup} of an inverse semigroup is simply a subsemigroup closed under the taking of inverses.
A {\em wide inverse subsemigroup} of $S$ is an inverse subsemigroup that contains $\mathsf{E}(S)$.
Inverse semigroups usually arise as inverse semigroups of partial bijections of a set where the inverse
of a partial bijection is simply its inverse partial bijection.
In an inverse semigroup $S$, we can define a partial order $\leq$, called the {\em natural partial order},
by $a \leq b$ if and only if $a = be$ for some idempotent $e$.
Despite appearances, this definition is self-dual in the sense that $a \leq b$ if and only if $a = fb$ for some idempotent $f$.
For inverse semigroups of partial bijections, the natural partial order is the usual restriction order of partial bijections.
If $a \leq b$ then $a^{-1} \leq b^{-1}$ and if $c \leq d$ then $ac \leq bd$.
The set of idempotents is an order ideal.
Observe that if $a,b \leq c$ then $a^{-1}b$ and $ab^{-1}$ are both idempotents.
Define the {\em compatibility relation} $\sim$ by $a \sim b$ if and only if $a^{-1}b$ and  $ab^{-1}$ are both idempotents.
It follows that a necessary condition for two elements to have a join is that they be compatible.
If $a$ and $b$ have a join, we denote it by $a \vee b$.
Elements $a$ and $b$ are said to be {\em orthogonal}, denoted by $a \perp b$, if $ab^{-1} = 0$ and $a^{-1}b = 0$.
If $a$ and $b$ are orthogonal and have a join, we denote it by $a \oplus b$.
An inverse monoid is said to be {\em distributive} if each compatible pair of elements has a join
and multiplication distributes over such joins on the left and the right.
A {\em morphism} of distributive inverse semigroups is required to preserve compatible joins.
A distributive inverse monoid is {\em Boolean} if its semilattice of idempotents is a Boolean algebra.

An inverse semigroup is said to be {\em fundamental} if the only elements that commute with all idempotents are
themselves idempotents.
An inverse semigroup with zero $S$ is said to be {\em $0$-simple} if for any two non-zero idempotents
$e$ and $f$ there is an element $x$ such that $\mathbf{d}(x) = e$ and $\mathbf{r}(x) \leq f$;
it is said to be {\em $0$-bisimple} if for each ordered pair of non-zero idempotents $e$ and $f$
there exists an element $a$ such that $e = a^{-1}a$ and $f = aa^{-1}$;
it is said to be {\em $0$-disjunctive} if $e < f$, where $e$ and $f$ are non-zero idempotents, 
implies that there is a non-zero idempotent $e' \leq f$ such that $ee' = 0$;
it is said to be {\em congruence-free} if the only congruences are equality and the universal congruence.

An inverse semigroup is a {\em $\wedge$-semigroup} if each pair of elements has a meet under the natural partial order.
These semigroups were first introduced in \cite{Leech}.
It can be proved that an inverse semigroup has binary meets if it possesses a fixed point operator $\phi$
meaning a map $\phi \colon S \rightarrow \mathsf{E}(S)$ such that $\phi (s)$ is the largest idempotent less than or equal to $s$ \cite[Definition 1.7]{Leech}.
If $\phi$ is such an operator, then 
$$a \wedge b = \phi (ab^{-1})b = a \phi (a^{-1}b)$$
by \cite[Theorem 1.7]{Leech}.
Having binary meets is the algebraic version of being Hausdorff which explains the importance of this notion in our paper.

An inverse semigroup $S$ is said to be {\em $E$-unitary} if $e \leq a$, where $e$ is an idempotent, implies that $a$ is an idempotent. Observe that $e$ could be zero.
The case where we require $e$ to be non-zero is the following:
an inverse semigroup with zero is said to be {\em $E^{\ast}$-unitary} if $e \leq a$, where $e$ is a non-zero idempotent, implies that $a$ is an idempotent.
It is an unfortunate feature of semigroup theory that definitions need to be changed according
to whether the semigroup has a zero or not.
However, this will cause us no problems in this paper.
The notion of being $E$-unitary is an important one in the development of inverse semigroup theory
and abstracts certain features of complex analytic functions \cite{Lawson1998}.

There is a congruence $\sigma$ defined on any inverse semigroup $S$ such that $S/\sigma$ is a group
and if $\rho$ is any congruence on $S$ such that $S/\rho$ is a group then $\sigma \subseteq \rho$.
Thus $\sigma$ is the {\em minimum group congruence}.
In fact, $s \, \sigma \, t$ if and only if there exists $u \leq s,t$.
Observe that $u$ could be zero in which case the group reduces to the trivial group.
However, we shall only apply this congruence when the inverse semigroup does not have a zero.
We sometimes write $\sigma_{S}$ if $\sigma$ is defined on $S$.
See \cite[Section 2.4]{Lawson1998}.
The following was proved as \cite[Theorem 2.4.6]{Lawson1998};
it makes checking whether two elements are $\sigma$-related much easier.

\begin{proposition}\label{prop:E} Let $S$ be an inverse semigroup.
Then $S$ is $E$-unitary if and only if $\sigma \, = \, \sim$.
\end{proposition}

Let $S$ be a monoid and $X$ a set.
Then a {\em monoid action} $S \times X \rightarrow X$ is a mapping $(a,x) \mapsto ax$ such that
$1x = x$ and $(ab)x = a(bx)$.
There is an associated homomomorphism $\lambda \colon S \rightarrow \mathcal{T}(X)$, where $\mathcal{T}(X)$ is the monoid of all functions from $X$ to itself,
given by $\lambda (s)(x) = sx$.
If this homomorphism is injective we say that the action is {\em effective}. 
This is equivalent to requiring that if $ax = bx$ for all $x \in X$ then $a = b$.

Free monoids are special examples of the $k$-monoids that form the basis of this paper
so we outline their basic properties here to help motivate what we do.
Let $A$ be a countable set.
The free monoid generated by $A$ is denoted $A^{\ast}$
and consists of all finite strings over the alphabet $A$ with the operation of concatenation as the semigroup
multiplication.
The set of all right-infinite strings over $A$ is denoted by $A^{\omega}$.
We denote by $A_{n}$ any alphabet with $n$ elements.
If $S$ is a free monoid and $a,b \in S$ then $aS \cap bS$ is non-empty precisely when
either $aS \subseteq bS$ or $bS \subseteq aS$.
If $a = bu$ for some $u \in S$ we say that $b$ is a {\em prefix} of $a$.
The elements $a$ and $b$ are said to be {\em prefix comparable} precisely when $a$ is a prefix of $b$ or $b$ is a prefix of $a$;
this is equivalent to $aS \cap bS \neq \varnothing$.
We shall generalize this definition to a more general setting later and change the terminology to avoid confusion.
If a pair of elements are not comparable we say that they are {\em incomparable}.
A {\em prefix code} in a free monoid is a (for us, finite) set of incomparable elements.
It is a {\em maximal prefix code} if it is a prefix code in which every element of the free monoid is prefix comparable with
an element of the prefix code.
The book by Lallement \cite{Lallement} is a good source for all these definitions.

\begin{center}
{\bf Notation}
\end{center}

\begin{itemize}

\item $S^{\infty}$ the topological space of all $k$-tilings of a $k$-monoid $S$ (Section~4).

\item $\mathsf{R}(S)$ the inverse monoid of all bijective morphisms between the finitely generated right ideals of the $k$-monoid $S$ (Section~5).

\item If $T$ is any inverse semigroup with zero then $T^{e}$ is the inverse subsemigroup of $T$ consisting of elements $a$ where $a^{-1}a$ and $aa^{-1}$ 
are both essential idempotents (Section~6).

\item $\mathscr{G}(S) = \mathsf{R}(S)^{e}/\sigma$, the group associated with a $k$-monoid $S$ (Section~6).

\item $\mathsf{P}(S)$ the inverse monoid of all bijective morphisms between finitely generated projective right ideals of the $k$-monoid $S$ (Section~7).

\item $\mathcal{I}(\mathscr{X})$, where $\mathscr{X}$ is a topological space, is the inverse monoid of all partial homeomorphisms between the open subsets of $\mathscr{X}$ (Section~4).

\item $\mathcal{I}^{\scriptstyle cl}(X)$, where $X$ is a Boolean space, is the Boolean inverse monoid of partial homeomorphisms between the clopen
subsets of $X$ (Section~8 onwards).

\item $\mathsf{C}(S) = \mathsf{R}(S)/\equiv$, the Boolean inverse monoid constructed from the $k$-monoid $S$.
Its group of units is isomorphic with $\mathscr{G}(S)$ (Section~9).

\item $\mathcal{G}(S)$ the Boolean groupoid associated with the $k$-monoid $S$ (Section~10).

\item $\mathsf{KB}(G)$ the Boolean inverse monoid of all compact-open local bisections of the Boolean groupoid $G$ (Section~10).

\item $\mathsf{B}(S)$ is the inverse monoid of all bijective morphisms between principal right ideals of $S$, together with the empty function, in the case where 
the $k$-monoid $S$ is singly aligned (Section~11).

\end{itemize}

\section{$k$-monoids} 

In this section, we define the class of monoids that will underlie all our constructions.

In what follows, $\mathbb{N}^{k}$ should be viewed as the positive cone of the lattice-ordered abelian group $\mathbb{Z}^{k}$.
If $\mathbf{m} \in \mathbb{N}^{k}$ then 
$$\mathbf{m} = (m_{1}, \ldots, m_{i}, \ldots, m_{k})$$ 
and we define $\mathbf{m}_{i} = m_{i}$.
Accordingly, we shall use superscripts $\mathbf{m}^{j}$ to label sequences of elements of $\mathbb{N}^{k}$. 
The order in $\mathbb{Z}^{k}$ is defined componentwise: $\mathbf{m} \leq \mathbf{n}$ if and only if $\mathbf{m}_{i} \leq \mathbf{n}_{i}$ for $1 \leq i \leq k$.
The join operation is $(\mathbf{m} \vee \mathbf{n})_{i} = \mbox{max}(m_{i},n_{i})$ and the meet operation is $(\mathbf{m} \wedge \mathbf{n})_{i} = \mbox{min}(m_{i},n_{i})$.
Put $\mathbf{0} = (0, \ldots, 0)$ and $\mathbf{1} = (1, \ldots, 1)$ both elements of $\mathbb{N}^{k}$.
Define $\mathbf{e}_{i}$, where $1 \leq i \leq k$, to be that element of $\mathbb{N}^{k}$ which is zero everywhere except at $i$ where it takes the value $1$.\\

\noindent
{\bf Definition. }A countable monoid $S$ is said to be a {\em $k$-monoid} if there is a homomorphism
$d \colon S \rightarrow \mathbb{N}^{k}$ satisfying the {\em unique factorization property (UFP)}:
if $d(x) = \mathbf{m} + \mathbf{n}$ then there exist unique elements $x_{1}$ and $x_{2}$ of $S$ such that $x = x_{1}x_{2}$ where $d(x_{1}) = \mathbf{m}$ and $d(x_{2}) = \mathbf{n}$.
We sometimes call $d(s)$ the {\em size} of $s$.\\

\noindent
{\bf Notation.} Let $S$ be a $k$-monoid and let $s \in S$ and $\mathbf{0} \leq \mathbf{m} \leq d(s)$.
Then we write $s = s(\mathbf{0},\mathbf{m})s(\mathbf{m},d(s))$ 
where 
$d(s(\mathbf{0},\mathbf{m})) = \mathbf{m}$ 
and
$d(s(\mathbf{m},d(s))) = d(s) - \mathbf{m}$
are the uniquely defined elements of their respective sizes.\\

The definition of a $k$-monoid is simply the monoid case of the more general notion of a higher rank graph defined in \cite{RS, KP, FMY, JS2014, JS2018}.\\

\noindent
{\bf Terminology. }Let $a,b \in S$, a $k$-monoid.
We shall say that $a$ is {\em bigger than} $b$ if $d(a) \geq d(b)$. 
This definition needs to be used with care since if $a$ and $b$ are arbitrary there need be no order
relation between $d(a)$ and $d(b)$.\\

For each $1 \leq i \leq k$, define $X_{i} = d^{-1}(\mathbf{e}_{i})$.
We call $(X_{1}, \ldots, X_{k})$ the {\em $k$ alphabets} associated with the $k$-monoid.

\begin{remark}{\em We can assume, without loss of generality, that the map $d$ used in the definition of a $k$-monoid
is actually surjective.
If $d$ were not surjective then we would have that $X_{i} = \varnothing$ for some $i$.
We could then replace $d$ by a map $d' \colon S \rightarrow \mathbb{N}^{k-1}$ by dropping the $i$-component
(which is always $0$).}
\end{remark}

\begin{lemma}\label{lem:bigger} Let $S$ be a $k$-monoid in which $d$ is surjective.
Then given any element $u \in S$ and $\mathbf{m} \in \mathbb{N}^{k}$ there is an element $v$ such that $d(uv) \geq \mathbf{m}$.
\end{lemma}
\begin{proof} If $d(u) \geq \mathbf{m}$ then put $v = 1$.
Otherwise, choose $v$ such that $d(v) \geq \mathbf{m}$.
This can be done since we are assuming that $d$ is surjective.
It follows that $d(uv) \geq \mathbf{m}$.
\end{proof}

A monoid is said to be {\em conical} if its group of units is trivial.
An {\em atom} in a monoid is an element $a$ such that if $a = bc$ then at least one of $b$ or $c$ is invertible.
The following are all well-known but can easily be proved from the definitions.

\begin{lemma}\label{lem:coffee} Let $S$ be a $k$-monoid.
\begin{enumerate}
\item The identity of $S$ is the only element $a \in S$ such that $d(a) = {\bf 0}$.
\item The monoid $S$ is conical.
\item The monoid $S$ is cancellative.
\item The atoms are the elements $a$ such that $d(a) = \mathbf{e}_{i}$ for some $1 \leq i \leq k$.
\item Each non-identity element is a product of atoms.
\end{enumerate}
\end{lemma}  

\begin{remark}\label{rem:free}
{\em Countable free monoids are precisely the $1$-monoids. This entitles us to view $k$-monoids as generalizations of free monoids which is
our perspective throughout this paper.}
\end{remark}

Let $S$ be a $k$-monoid and $T$ an $l$-monoid.
We denote their respective homomorphisms by 
$d_{S} \colon S \rightarrow \mathbb{N}^{k}$
and
$d_{T} \colon T \rightarrow \mathbb{N}^{l}$.
There is a natural isomorphism $\mathbb{N}^{k} \times \mathbb{N}^{l} \cong \mathbb{N}^{k+l}$
which takes $(\mathbf{m},\mathbf{n})$ to the element $\mathbf{m} \cdot \mathbf{n}$
where  $\mathbf{m} \cdot \mathbf{n}$ is the unique element of $\mathbb{N}^{k+l}$
whose components are $(m_{1}, \ldots, m_{k}, n_{1}, \ldots, n_{l})$.
Define $d \colon S \times T \rightarrow \mathbb{N}^{k+l}$ by
$d(s,t) = d_{S}(s) \cdot d_{T}(t)$.
In this way, it is easy to see that $S \times T$ is a $k+l$-monoid.
We have proved the following; see also \cite[Proposition 1.8]{KP}.

\begin{lemma}\label{lem:prod} 
If $S$ is a $k$-monoid and $T$ is an $l$-monoid then $S \times T$ is a $k+l$-monoid.
\end{lemma}

\begin{example}\label{ex:bbc}
{\em The direct product of $k$ free monoids is therefore a $k$-monoid by Lemma~\ref{lem:prod}.
However, the product of free monoids is not free.
For example, the free monoid on one generator is $\mathbb{N}$.
But $\mathbb{N} \times \mathbb{N}$ is not isomorphic to the free monoid on one generator but is abelian.
Thus it cannot be free}.
\end{example}

If $X \subseteq S$, where $S$ is a monoid, we denote by $X^{\ast}$ the submonoid of $S$ generated by $X$.
The proof of the following lemma is immediate by the UFP.

\begin{lemma}\label{lem:relations} Let $S$ be a $k$-monoid.
\begin{enumerate}
\item $X_{i}X_{j} \subseteq X_{j}X_{i}$ for all $i,j$.
\item Each submonoid $X_{i}^{\ast}$ is free.
\item Each element of $S$ can be written uniquely as a product of the form $x_{1} \ldots x_{k}$
where $x_{i} \in X_{i}^{\ast}$.
\end{enumerate}
\end{lemma}

We now develop part (1) of Lemma~\ref{lem:relations}.
Let $S$ be a $k$-monoid with $k$ alphabets $(X_{1}, \ldots, X_{k})$.
Let $x_{i_{1}} \in X_{i}$ and $x_{j_{1}} \in X_{j}$ where $i \neq j$.
Then by the UFP, we may write $x_{i_{1}}x_{j_{1}} = x_{j_{1}'}x_{i_{1}'}$ 
where $x_{i_{1}'} \in X_{i}$ and $x_{j_{1}'} \in X_{j}$ are uniquely defined.
We have therefore defined a function $\theta_{ij} \colon X_{i} \times X_{j} \rightarrow X_{i} \times X_{j}$
given by $\theta (i_{1},j_{1}) = (i_{1}',j_{1}')$.
By the UFP, this function is a bijection.
For the case $k = 2$, any bijection $X_{1} \times X_{2} \rightarrow X_{1} \times X_{2}$
determines a unique $2$-monoid and conversely \cite{Power2007}.
For $k > 2$, we need to add an extra condition
which arises by associativity as follows \cite{FS2002, DY2009b}.
Let $x_{i_{1}} \in X_{i}$, $x_{j_{1}} \in X_{j}$ and $x_{l_{1}} \in X_{l}$
where $i,j,l$ are distinct.
We now make two calculations.
First,
$$(x_{i_{1}}x_{j_{1}})x_{l_{1}} = x_{j_{1}'}(x_{i_{1}'}x_{l_{1}}) = (x_{j_{1}'}x_{l_{1}'})x_{i_{1}''} = x_{l_{1}''}x_{j_{1}''}x_{i_{1}''}.$$
Second,
$$x_{i_{1}}(x_{j_{1}}x_{l_{1}}) = (x_{i_{1}}x_{l_{2}})x_{j_{2}} = x_{l_{3}}(x_{i_{2}}x_{j_{2}}) = x_{l_{3}}x_{j_{3}}x_{i_{3}}.$$
By associativity (since we are working in a monoid), we therefore have that
$x_{l_{1}''} = x_{l_{3}}$, $x_{j_{1}''} = x_{j_{3}}$ and $x_{x_{1}''} = x_{i_{3}}$.
This leads to an important connection between $k$-monoids (where $k \geq 3$) and set-theoretical solutions to the
Yang-Baxter equation \cite{KP, DY2016}.

The following describes the `degenerate' $k$-monoids.

\begin{proposition}\label{prop:abelian} Let $S$ be a $k$-monoid with associated $k$-alphabets
$(X_{1}, \ldots,X_{k})$.
Then $S \cong \mathbb{N}^{k}$ if and only if $|X_{1}| = |X_{2}| = \ldots = |X_{k}| = 1$.
\end{proposition}
\begin{proof} Only the `if-direction' needs proving.
Suppose that $|X_{1}| = |X_{2}| = \ldots = |X_{k}| = 1$.
Then $X_{i}^{\ast} \cong \mathbb{N}$.
The homomorphism $d \colon S \rightarrow \mathbb{N}^{k}$, which we always assume surjective,
is injective by the UFP.
\end{proof}

By Remark~\ref{rem:free}, countable free monoids are $1$-monoids and by Lemma~\ref{lem:prod}, the product of
$k$ free monoids is a $k$-monoid and so $A_{n_{1}}^{\ast} \times \ldots \times A_{n_{k}}^{\ast}$
is a $k$-monoid.

\begin{proposition}\label{prop:brexit} The $k$-monoid $S$ is isomorphic to a finite direct product of free monoids if and only if
$xy = yx$ for all $x \in X_{i}$, $y \in X_{j}$ and $i \neq j$ where $1 \leq i,j \leq k$.
\end{proposition}
\begin{proof} Only the `if-direction' needs proving.
We show that $S \cong X_{1}^{\ast} \times \ldots \times X_{k}^{\ast}$.
Let $s \in S$.
Then it can be writtten uniquely as a product $s = x_{1} \ldots x_{k}$ where $x_{i} \in X_{i}^{\ast}$.
This enables us to define a bijection between $S$ and $X_{1}^{\ast} \times \ldots \times X_{k}^{\ast}$.
It only remains to prove that it is a homomorphism.
Let $t = y_{1} \ldots y_{k}$ where $y_{i} \in X_{i}^{\ast}$.
Then $st = (x_{1} \ldots x_{k})(y_{1} \dots y_{k})$.
We now use commutativity to move $y_{1}$.
We get $st = ((x_{1}y_{1})x_{2} \ldots x_{k})(y_{2} \ldots y_{k})$.
We now repeat this process to get
$st = (x_{1}y_{1})\ldots (x_{k}y_{k})$.
This now proves that our bijection is in fact a homomorphism and so an isomorphism.
\end{proof}

The following is the analogue for $k$-monoids of what is termed `Levi's theorem' for free monoids \cite[Corollary 5.1.6]{Lallement}.
It will prove useful.

\begin{lemma}\label{lem:levi} Let $S$ be a $k$-monoid.
Let $xy = uv$, where $x,y,u,v \in S$ and $d(x) \geq d(u)$.
Then there exists $t \in S$ such that such that $x = ut$ and $v = ty$.
In particular, if $d(x) = d(u)$ then $x = u$.
\end{lemma}
\begin{proof} First, we deal with the case where $d(x) = d(u)$.
Put $z = xy = uv$.
Since $d(x) = d(u)$, we must have that $d(y) = d(v)$.
It now follows by the UFP, that $x = u$ and $y = v$.
Now we suppose that $d(x) > d(u)$.
Thus there exists $\mathbf{r} \in \mathbb{N}^{k}$ such that
$d(x) = d(u) + \mathbf{r}$.
By the UFP, there are unique elements $u'$ and $t$ such that $d(u') = d(u)$, $d(t) = \mathbf{r}$ and $x = u't$.
It follows that $u'(ty) = uv$.
But $d(u') = d(u)$.
By the first part of the proof, we deduce that $u' = u$ and so $v = ty$
from which we get that $x = ut$. 
\end{proof}

We now introduce concepts that will play a major role in our construction of a group from a $k$-monoid.
Let $S$ be any monoid and let $x,y \in S$.
If $xS \cap yS \neq \varnothing$, we say that $x$ and $y$ are {\em comparable}
whereas if  $xS \cap yS = \varnothing$, we say that $x$ and $y$ are {\em incomparable}.
A finite, non-empty subset $X$ of $S$ is said to be a {\em generalized prefix code} if each distinct pair of elements of $X$ is incomparable.
A generalized prefix code $X$ is said to be {\em maximal} if for each $y \in S$ there exists $x \in X$ such that
$yS \cap xS \neq \varnothing$.

\begin{example}{\em In a free monoid, generalized prefix codes sets are precisely the {\em prefix codes} and maximal generalized prefix codes are precisely the
{\em maximal prefix codes}.}
\end{example}

\begin{remark}{\em If $C \neq \{1\}$ is a maximal generalized prefix code then $1 \notin C$.
To see why, suppose that $1 \in C$ and $c \in C$ where $c \neq 1$.
Then $c1 = 1c$ and so $c$ and $1$ are comparable.}
\end{remark}

Let $S$ be a $k$-monoid, let $\mathbf{m} \in \mathbb{N}^{k}$ and
let $C_{\mathbf{m}}$ be the set of all elements $x$ of $S$ such that $d(x) = \mathbf{m}$.

\begin{lemma}\label{lem:homogeneous} Let $S$ be a $k$-monoid and let $\mathbf{m} \in \mathbb{N}^{k}$. 
Then $C_{\mathbf{m}}$ is a maximal generalized prefix code.
\end{lemma}
\begin{proof} Let $x,y \in C_{\mathbf{m}}$ and suppose that $xu = yv$ for some $u,v \in S$.
Then $d(x) = d(y) = \mathbf{m}$, by assumption.
It follows by Lemma~\ref{lem:levi} that $x = y$.
Thus the elements of $C_{\mathbf{m}}$ are incomparable.
Let $u \in S$.
Then there is an element $v \in S$ such that $d(uv) \geq \mathbf{m}$ (since we always assume that $d$ is surjective).
By the UFP, there is an element $x \in C_{\mathbf{m}}$ and $y \in S$ such that $uv = xy$.
Thus $u$ is comparable with some element of $C_{\mathbf{m}}$.
\end{proof}

We can actually strengthen the above result.

\begin{lemma}\label{lem:redwine} Let $S$ be a $k$-monoid.
Let $C$ be a maximal generalized prefix code in $S$ in which every element has the same size $\mathbf{m}$.
Then, in fact, $C = C_{\mathbf{m}}$.
\end{lemma}
\begin{proof} Clearly, $C \subseteq C_{\mathbf{m}}$.
Let $y$ be any element of $S$ such that $d(y) = \mathbf{m}$.
Then, by assumption, there exists $x \in C$ and $u,v \in S$ such that
$yu = xv$.
But, by assumption, $d(y) = d(x)$.
It follows by Lemma~\ref{lem:levi} that $x = y$.
Thus $y \in C_{\mathbf{m}}$, as required.
\end{proof}

We develop more of the theory of maximal generalized prefix codes in Section~12.

We say that a monoid $S$ is {\em finitely aligned} if for each $x,y \in S$ either $x$ and $y$ are incomparable
or the right ideal $xS \cap yS$ is finitely generated.
This definition first appeared in \cite{RSY}.
We say that $S$ is {\em strongly finitely aligned}\footnote{This term is introduced in this paper; it is not current.} if for each $x,y \in S$ either $x$ and $y$ are incomparable
or the right ideal $xS \cap yS$ is finitely generated by incomparable elements; that is, it is generated by a generalized prefix code.
The result below is proved in the literature.

\begin{lemma}\label{lem:nelson} Let $d \colon S \rightarrow \mathbb{N}^{k}$ be a $k$-monoid
and let $cS \subseteq aS,bS$.
Then there exists an element $e \in S$ such that $cS \subseteq eS \subseteq aS \cap bS$ where $d(e) = d(a) \vee d(b)$.
\end{lemma}

Let $a,b \in S$.
We define the subset $a \vee b$ of $S$ as follows:
if $aS \cap bS = \varnothing$,  define $a \vee b = \varnothing$,
and if $aS \cap bS \neq \varnothing$ define $a \vee b$ to be all elements $e$ such that $e \in aS \cap bS$ and
$d(e) = d(a) \vee d(b)$.
By Lemma~\ref{lem:nelson}, it is immediate that $aS \cap bS = \bigcup_{e \in a \vee b} eS$.
Define also
$$\mathcal{C}(a,b) = \{(x,y) \in S \times S \colon ax = by \in a \vee b\}.$$
The fact that $a \vee b$ is always finite means precisely that $S$ is {\em finitely aligned}.

\begin{lemma}\label{lem:indep} 
If $a \vee b$ is non-empty then it is a set of incomparable elements. 
\end{lemma}
\begin{proof} Let $u,v \in a \vee b$.
Suppose that $uS \cap vS \neq \varnothing$. 
Then $us = vt$ for some $s,t \in S$.
But $d(u) = d(v)$.
It follows by Lemma~\ref{lem:levi} that $u = v$.
\end{proof}

By Lemma~\ref{lem:nelson} and Lemma~\ref{lem:indep}, it follows that a finitely aligned $k$-monoid is strongly finitely aligned.

\begin{remark}{\em In this paper, we are only interested in finitely aligned $k$-monoids. We shall also assume that each of the alphabets $X_{i}$ is finite and has cardinality at least 2.
Later, we shall require that the $k$-monoid is aperiodic. }\end{remark}

A $k$-monoid $S$ is said to be {\em singly aligned} if $aS \cap bS \neq \varnothing$ implies that $aS \cap bS = cS$ for some $c \in S$.
All free monoids are singly aligned as, so too, are the rigid $k$-monoids defined in Section~11.
The proof of the following is straightforward.

\begin{lemma}\label{lem:scott} Let $S$ and $T$ be $k$-monoid and $l$-monoid, respectively.
If both are singly aligned then $S \times T$ is singly aligned.
\end{lemma}

\section{The associated Boolean space}

Let $S$ be a $k$-monoid.
In this section, we shall recall the construction of a topological space, denoted by $S^{\infty}$, on which $S$ acts on the left.
We are particularly interested in when this left action is effective.
A topological space is said to be {\em Boolean} if it is compact Hausdorff with a basis of clopen subsets;
the space $S^{\infty}$ will be Boolean.

For each integer $k \geq 1$, define $\Omega_{k}$ to be the category with objects the elements of $\mathbb{N}^{k}$
and morphisms those ordered pairs of elements of $\mathbb{N}^{k}$, $(\mathbf{m},\mathbf{n})$, where $\mathbf{m} \leq \mathbf{n}$;
in other words, the category associated with the partial order $\leq$ on $\mathbb{N}^{k}$ where the arrow goes from the larger to the smaller element.
Define $d(\mathbf{m},\mathbf{n}) = \mathbf{n} - \mathbf{m}$.
By a {\em $k$-tiling} in $S$, we mean a function $w \colon \Omega_{k} \rightarrow S$ satisfying the following three properties:
\begin{enumerate}

\item $w(\mathbf{m}, \mathbf{m}) = 1$.

\item $w (\mathbf{m}, \mathbf{n}) w(\mathbf{n},\mathbf{p}) = w (\mathbf{m},\mathbf{p})$.

\item $d(w (\mathbf{m}, \mathbf{n})) = \mathbf{n} - \mathbf{m}$.

\end{enumerate}
An element of $S$ of the form $w(\mathbf{0},\mathbf{m})$ is called a {\em corner} of the $k$-tiling $w$.\\

\noindent
{\bf Definition.} Define $S^{\infty}$ to be the set of all $k$-tilings in $S$.\\

For each $\mathbf{p} \in \mathbb{N}^{k}$, define the {\em shift (erase) operator} $\sigma^{\mathbf{p}} \colon S^{\infty} \rightarrow S^{\infty}$
by 
$$\sigma^{\mathbf{p}}(w)(\mathbf{m}, \mathbf{n}) = w(\mathbf{m} + \mathbf{p}, \mathbf{n} + \mathbf{p})$$ 
where $w \in S^{\infty}$.
It is easy to check that this operation is well-defined.

A {\em comparable set of corners} is a subset of $S$: 
$$\mathscr{C} = \{w_{\mathbf{m}} \colon \mathbf{m} \in \mathbb{N}^{k},
d(w_{\mathbf{m}}) = \mathbf{m}, \text{ any two elements are comparable}\}.$$
The following is a special case of the first part of \cite[Remarks~2.2]{KP}.

\begin{proposition}\label{prop:moolah} Let $S$ be a $k$-monoid.
Then $k$-tilings and comparable sets of corners are equivalent notions.
\end{proposition}

We explain what we mean by the term `equivalent' used above.
Let  $w \colon \Omega_{k} \rightarrow S$ be a $k$-tiling.
Then $\mathscr{C}_{w} = \{w(\mathbf{0}, \mathbf{m}) \colon \mathbf{m} \in \mathbb{N}^{k}\}$
is a comparable set of corners.
Now, let $\mathscr{C}$ be a comparable set of corners.
Then an associated $k$-tiling $w_{\mathscr{C}}$ can be constructed as follows:
let $\mathbf{m}, \mathbf{n} \in \mathbb{N}^{k}$ be such that $\mathbf{m} \leq \mathbf{n}$.
Then $w_{\mathbf{m}}$ and $w_{\mathbf{n}}$ are comparable
and
$w_{\mathbf{n}} = w_{\mathbf{m}}w(\mathbf{m}, \mathbf{n})$ for a unique $w (\mathbf{m}, \mathbf{n}) \in S$
by Lemma~\ref{lem:levi}.
This defines the $k$-tiling.
The operations: $w \mapsto \mathscr{C}_{w}$ and
$\mathscr{C} \mapsto w_{\mathscr{C}}$ are mutually inverse.

\begin{remark}\label{rem:ribena}
{\em The above result is important because it tells us, in particular, that each $k$-tiling has corners of every possible size
and that any two corners are comparable.}
\end{remark}

The following lemma will be needed in Section~9.

\begin{lemma}\label{lem:size-matters} Let $S$ be a $k$-monoid, let $w$ be a $k$-tiling, let $a_{1}, \ldots, a_{n} \in S$ be arbitrary
and let $\mathbf{m}^{1}, \ldots, \mathbf{m}^{n} \in \mathbb{N}^{k}$.
Then there is a corner $x$ of $w$ such that $d(a_{j}x) \geq \mathbf{m}^{j}$ 
for each $1 \leq j \leq n$.
\end{lemma}
\begin{proof} The $k$-tiling $w$ has corners of every size.
Let $x$ be any corner of $w$ such that $d(x) = \bigvee_{j=1}^{n} \mathbf{m}^{j}$.
Then $d(a_{j}x) = d(a_{j}) + d(x) \geq \mathbf{m}^{j}$ for each $1 \leq j \leq n$.
\end{proof}

A comparable set of corners requires there to be an element of size $\mathbf{m}$ for each $\mathbf{m} \in \mathbb{N}^{k}$.
In fact, this condition can be weakened.
The following is proved as the second part of \cite[Remark 2.2]{KP}.

\begin{proposition}\label{prop:rain} Let $S$ be a $k$-monoid.
Let $\mathscr{D}$ be a comparable subset of $S$ 
such that for each $\mathbf{n} \in \mathbb{N}^{k}$ there is at least one element $b \in \mathscr{D}$ such that
$d(b) \geq \mathbf{n}$. 
Then $\mathscr{D}$ gives rise to a unique comparable set of corners $\mathscr{D}'$.
The set $\mathscr{D}'$ contains the element $x$, where $d(x) = \mathbf{n}$,
precisely when $a \in \mathscr{D}$ is any element such that $d(a) \geq \mathbf{n}$ and $a = xy$ for some $y$.
\end{proposition}

A subset $\mathscr{D} \subseteq S$ is called {\em expanding} if it satisfies the conditions of Proposition~\ref{prop:rain}.
If $\mathscr{D}_{1}$ and $\mathscr{D}_{2}$ are two expanding sets, we say that they are {\em equivalent}, denoted by $\mathscr{D}_{1} \equiv \mathscr{D}_{2}$,
if they determine the same $k$-tiling.
This means that $\mathscr{D}_{1}' = \mathscr{D}_{2}'$, using the notation of Proposition~\ref{prop:rain}.
The proof of the following is routine.
 
\begin{lemma}\label{lem:brexit} Let $\mathscr{D}_{1}$ and $\mathscr{D}_{2}$ be two expanding sets.
Suppose that for any $x \in \mathscr{D}_{1}$ and $y \in \mathscr{D}_{2}$, we have that $x$ and $y$ are comparable.
Then $\mathscr{D}_{1}$ and $\mathscr{D}_{2}$ are equivalent.
\end{lemma}

\begin{example}\label{ex:tomb}
{\em In the light of our above result, it is now easy to show that $(A_{n}^{\ast})^{\infty} = A_{n}^{\omega}$.}
\end{example}


The following is a special case of \cite[Proposition 2.3]{KP}.

\begin{lemma}\label{lem:pepsi} Let $a \in S$ and $w \in S^{\infty}$.
Then there is a unique element denoted $aw \in S^{\infty}$ such that $(aw)(\mathbf{0}, d(a)) = a$ and 
$\sigma^{d(a)}(aw) = w$. 
\end{lemma}

It follows by the above lemma, 
that there is therefore a natural action  $S \times S^{\infty} \rightarrow S^{\infty}$
meaning that $1w = w$ and $(ab)w = a(bw)$ for all $a,b \in S$.
We write
$$aS^{\infty} = \{aw \colon w \in S^{\infty}\},$$
the set of all $k$-tilings with $a$ as a corner.
Put $\tau = \{aS^{\infty} \colon a \in S\}$.
Then $\tau$ forms the basis for a topology on $S^{\infty}$.
Part of the following was first proved in \cite{KP}.
The proof we give is due to Aidan Sims.

\begin{proposition}\label{prop:top}
Let $S$ be a $k$-monoid with finite $k$-alphabets $(X_{1}, \ldots, X_{n})$.
Then $S^{\infty}$, equipped with the topology above, is a compact Boolean space.
If at least one of the alphabets $X_{i}$ has cardinality at least 2 then $S^{\infty}$ is homeomorphic
to the Cantor space.
\end{proposition}
\begin{proof} 
Recall that $C_{\mathbf{m}}$ is the set of all elements of $S$ of size $\mathbf{m}$.
If $x \in C_{(n, \ldots, n)}$ then there are unique elements $x_{1}$ and $x_{2}$ such that
$x = x_{1}x_{2}$, $d(x_{1}) = (n-1, \ldots, n-1)$ and $d(x_{2}) = \mathbf{1}$.
Define the map $p_{\mathbf{n}} \colon C_{(n,\ldots,n)} \rightarrow C_{(n-1, \ldots, n-1)}$ by $p_{n}(x) = x_{1}$.
Then $S^{\infty}$ is isomorphic to the projective limit of the system $(C_{(n,\ldots,n)}, p_{n})$.
Standard results from topology now tell us that $S^{\infty}$ is a Boolean space.
To show that $S^{\infty}$ is the Cantor space, we have to show that there are no isolated points.
But for any $a \in S$ the set $aS^{\infty}$ will contain at least two elements if at 
at least one of the alphabets $X_{i}$ has cardinality at least 2. 
\end{proof}

We say that the action $S \times S^{\infty} \rightarrow S^{\infty}$ is {\em effective} if $aw = bw$ for all $w \in S^{\infty}$ implies that $a = b$.
The following was motivated by \cite{LS2010}.

\begin{lemma}\label{lem:next} Let $S$ be a $k$-monoid.
\begin{enumerate}

\item Let $a,b \in S$ such that $aw = bw$ for some $w \in S^{\infty}$.
Then $a$ and $b$ are comparable.

\item Let $a,b \in S$ be such that $aw = bw$ for all $w \in S^{\infty}$.
Then $au$ and $bu$ are comparable for all $u \in S$ 

\end{enumerate}
\end{lemma}
\begin{proof} 
(1) Let $\mathbf{m} \geq d(a), d(b)$.
Then $(aw)(\mathbf{0}, \mathbf{m}) = (bw)(\mathbf{0}, \mathbf{m})$.
Thus by the UFP, we have that $au = bv$ for some $u,v \in S$.
It follows that $a$ and $b$ are comparable.

(2) Let $u \in S$ be arbitrary.
If $w \in S^{\infty}$ then $uw \in S^{\infty}$.
By assumption, $a(uw) = b(uw)$ for all $w \in S^{\infty}$.
Thus $(au)w = (bu)w$ for all $w \in S^{\infty}$.
By part (1), it follows that $au$ and $bu$ are comparable.
\end{proof}

We now prove the converse to part (2) above.

\begin{lemma}\label{lem:feeble} Let $a,b \in S$ be such that $au$ and $bu$ are comparable for all $u \in S$.
Then $aw = bw$ for all $w \in S^{\infty}$. 
\end{lemma}
\begin{proof} We may assume that $a \neq b$.
We prove the result by contradiction.
Suppose that there exists  $w \in S^{\infty}$ such that $aw \neq bw$.
Then for some $\mathbf{m}$ we have that $(aw)(\mathbf{0},\mathbf{m}) \neq (bw)(\mathbf{0},\mathbf{m})$.
Without loss of generality, we can choose  $\mathbf{m} \geq d(a), d(b)$.
We can therefore write
$au' = (aw)(\mathbf{0},\mathbf{m})$ 
and
$bv' = (bw)(\mathbf{0},\mathbf{m})$
where $u'$ and $v'$ are both corners of $w$ and
$d(au') = d(bv') = \mathbf{m}$.
Since $u'$ and $v'$ are both corners of $w$, it follows by Proposition~\ref{prop:moolah} that $u'$ and $v'$ are comparable and so there exist $x,y \in S$ such that $u = u'x = v'y$.
By assumption, $au$ and $bu$ are comparable and so
$aue = buf$ for some $e,f \in S$.
It follows that $au'xe = bv'yf$.
But $au'$ and $bv'$ have the same size, and so by Lemma~\ref{lem:levi}, we have that $au' = bv'$
which is a contradiction.
\end{proof}

Our next definition is nothing but a reformulation of the definition of aperiodicity given in \cite{LS2010}
with a new name.
We say that a $k$-monoid $S$ is {\em effective} (so, no reference to the action) if for each pair of distinct elements $a$ and $b$ of $S$ there exists an element $u \in S$
such that $au$ and $bu$ are not comparable.
By Lemma~\ref{lem:next} and Lemma~\ref{lem:feeble}, we have therefore proved the following result.

\begin{proposition}\label{prop:need}
The $k$-monoid $S$ is effective if and only if the action of $S$ on $S^{\infty}$ is effective.
\end{proposition}

We shall now connect effectiveness with another property.
Let $w \in S^{\infty}$.
We say that $w$ has {\em period} $\mathbf{p}$, where $\mathbf{p} \in \mathbb{Z}^{k}$, if for all
$(\mathbf{m},\mathbf{n}) \in \Omega_{k}$ with $\mathbf{m} + \mathbf{p} \geq \mathbf{0}$ we have that
$w(\mathbf{m} + \mathbf{p}, \mathbf{n} + \mathbf{p}) = w(\mathbf{m}, \mathbf{n})$.
We say that $w$ is {\em periodic} if it has some non-zero period.
We say that $w$ is {\em eventually periodic} if for some $\mathbf{q}$ we have that $\sigma^{\mathbf{q}}(w)$ is periodic.
An element of $S^{\infty}$ that is not eventually periodic is said to be {\em aperiodic}.

The next example will motivate the result that follows it.

\begin{example}\label{ex:blow}
{\em We work in the 1-dimensional case.
Consider the finite strings $ab$ and $abcd$.
Can we find a right-infinite string $w$ such that $abw = abcdw$?
A little thought will show that there is exactly one solution: $w = cdcdcd \ldots$.
In particular, $w$ is periodic.}
\end{example}

The following can also be found as \cite[Remark 4.4]{RSY2003}.

\begin{lemma}\label{lem:hard} Let $S$ be a $k$-monoid and let $a,b \in S$ be distinct elements.
If $w \in S^{\infty}$ is such that $aw = bw$ then $w$ is eventually periodic.
\end{lemma}

The proof of the following is now immediate on the basis of the above lemma.

\begin{corollary}\label{cor:good} Let $S$ be a $k$-monoid with at least one aperiodic $k$-tiling.
Then the action of $S$ on $S^{\infty}$ is effective.
\end{corollary}

\noindent
{\bf Definition. }A $k$-monoid with at least one aperiodic $k$-tiling will be called {\em aperiodic}.\\

The converse to Corollary~\ref{cor:good} is proved below
---
we thank Aidan Sims for the direct argument
---
whereas
we shall prove in Section~10 that in the context of \'etale groupoids 
being aperiodic and being effective are equivalent.

\begin{lemma}\label{lem:as} 
In a $k$-monoid, being effective implies the existence of an aperiodic $k$-tiling.
It follows that being effective and having an aperiodic $k$-tiling are equivalent notions.
\end{lemma}
\begin{proof} Suppose that $S$ is effective.
We prove that it must have an aperiodic $k$-tiling.
List the elements of  
$(\mathbb{N}^k \times \mathbb{N}^k) \setminus \{(\mathbf{n},\mathbf{n}) \colon \mathbf{n} \in \mathbb{N}^k \}$ as $(\mathbf{m}^{i}, \mathbf{n}^{i})^\infty_{i=1}$.
For each $i$, choose an element $s_i$ of length $\mathbf{m}^i + \mathbf{n}^i$. 
Since $S$ is effective, there exists an element $t_i$ such that 
$s_{i}(\mathbf{m}^i, \mathbf{m}^i + \mathbf{n}^i)t_i$ 
and  
$s_{i}(\mathbf{n}^i, \mathbf{m}^i + \mathbf{n}^i)t_i$  
are incomparable.
Observe that it is here we use the fact that $\mathbf{m}^{i} \neq \mathbf{n}^{i}$.
Put $l_i = s_it_i$.
The set of elements of $S$ given by $\{l_{1}, l_{1}l_{2}, l_{1}l_{2}l_{3}, \ldots \}$
is clearly an expanding set and so determines a unique $k$-tiling $x$ by Proposition~\ref{prop:moolah} and Proposition~\ref{prop:rain}.
We prove that $x$ is aperiodic.
Suppose that 
$\sigma^{\mathbf{m}^i}(x) = \sigma^{\mathbf{n}^i}(x)$.
Then $\sigma^{\sum^{i-1}_{j=0} d(l_i) + \mathbf{m}^i}(x) = \sigma^{\sum^{i-1}_{j=0} d(l_i) + \mathbf{n}^i}(x)$.
Thus $\sigma^{\mathbf{m}^i}(l_il_{i+1} \dots) = \sigma^{\mathbf{n}^i}(l_il_{i+1} \dots)$. 
But 
$$\sigma^{\mathbf{m}^i}(l_il_{i+1} \dots) = s_i(\mathbf{m}^i, \mathbf{m}^i+\mathbf{n}^i)t_il_{i+1} \ldots,$$ 
and 
$$\sigma^{\mathbf{n}^i}(l_il_{i+1} \dots) = s_i(\mathbf{n}^i, \mathbf{m}^i+\mathbf{n}^i)t_il_{i+1} \ldots,$$ 
and this contradicts the fact that
$s_i(\mathbf{m}^i, \mathbf{m}^i+\mathbf{n}^i)t_i$ and $s_i(\mathbf{n}^i, \mathbf{m}^i+\mathbf{n}^i)t_i$
are incomparable.
\end{proof}

We now describe the properties of the action $S \times S^{\infty} \rightarrow S^{\infty}$ that will be needed later.

\begin{proposition}\label{prop:bojo} Let $S$ be an aperiodic $k$-monoid.
\begin{enumerate}

\item If $aw = aw'$ where $a \in S$ and $w,w' \in S^{\infty}$ then $w = w'$.

\item Suppose that $w = aw_{1} = bw_{2}$ where $a,b \in S$ and $w_{1},w_{2} \in S^{\infty}$.
Then $au_{1} = bu_{2}$ where $u_{1},u_{2} \in S$
and
$w_{1} = u_{1} w'$ and $w_{2} = u_{2} w'$ for some $w' \in S^{\infty}$.

\item If $aw = bw$ for all $w \in S^{\infty}$ then $a = b$.

\end{enumerate}
\end{proposition}
\begin{proof}
(1) We have that $w = \sigma^{d(a)}(aw) = \sigma^{d(a)}(aw') = w'$.

(2) Put $\mathbf{m} = d(a) \vee d(b)$.
Then 
$z = w(\mathbf{0}, \mathbf{m})$ is a corner of $w$.
We can write 
$z 
= w(\mathbf{0}, d(a))w(d(a), \mathbf{m})
= w(\mathbf{0}, d(b))w(d(b), \mathbf{m})$.
Put $u_{1} = w(d(a), \mathbf{m})$ and $u_{2} = w(d(b), \mathbf{m})$.
We therefore have that $z = au_{1} = bu_{2}$.
Now $w_{1} = \sigma^{d(a)}(w)$ and $w_{2} = \sigma^{d(b)}(w)$.
It follows that 
$w_{1} = u_{1} w_{1}'$ and $w_{2} = u_{2} w_{2}'$
for suitable elements $w_{1}', w_{2}' \in S^{\infty}$.
But then $au_{1}w_{1}' = bu_{2}w_{2}'$.
By Part (1), we deduce that $w_{1}' = w_{2}'$.
Therefore we put $w' = w_{1}' = w_{2}'$.

(3) This is Corollary~\ref{cor:good}.
\end{proof}

We take the properties listed in Proposition~\ref{prop:bojo} as the basis for the following axiomatization.
Let $S$ be a $k$-monoid and let $\mathcal{X}$ be a set with a left $S$-action.
Thus, $1x = x$ for all $x \in \mathcal{X}$ and $(ab)x = a(bx)$ for all $a,b \in S$.
In addition, we assume the following three properties:
\begin{description}     

\item[{\rm (A1)}] If $ax = ax'$ then $x = x'$ where $a \in S$ and $x,x' \in \mathcal{X}$.

\item[{\rm (A2)}] If $ax_{1} = bx_{2}$ then there exist $u_{1}, u_{2} \in S$ such that $au_{1} = bu_{2}$ and
$x_{1} = u_{1}x'$ and $x_{2} = u_{2}x'$ for some $x' \in \mathcal{X}$.

\item[{\rm (A3)}] If $ax = bx$ for all $x \in \mathcal{X}$ then $a = b$.

\end{description}

We shall call the pair $(S,\mathcal{X})$ satisfying the above axioms a {\em system}.
We shall now investigate those consequences of the above axioms that will be useful to us later.
The following lemma actually proves that $a\mathcal{X} \cap b\mathcal{X} = (a \vee b)\mathcal{X}$ in all cases but we have divided
it into two results for the sake of clarity.

\begin{lemma}\label{lem:pepper} Let $(S,\mathcal{X})$ be a system.
\begin{enumerate}
\item For all $a,b \in S$, we have that $a\mathcal{X} \cap b\mathcal{X} = \varnothing$ if and only if $aS \cap bS = \varnothing$.
\item If $aS \cap bS \neq \varnothing$ then $a\mathcal{X} \cap b\mathcal{X} = (a \vee b)\mathcal{X}$.
\end{enumerate}
\end{lemma}
\begin{proof} (1) Suppose that  $aS \cap bS = \varnothing$.
If $a\mathcal{X} \cap b\mathcal{X} \neq \varnothing$ then by (A2) there exist $u_{1}, u_{2} \in S$ such that $au_{1} = bu_{2}$.
But this contradicts the assumption that  $aS \cap bS = \varnothing$.
It follows that $a\mathcal{X} \cap b\mathcal{X} = \varnothing$.
Now suppose that  $a\mathcal{X} \cap b\mathcal{X} = \varnothing$.
If  $aS \cap bS \neq \varnothing$ then we can find $u,v \in S$ such that $au = bv$.
Let $x \in \mathcal{X}$ be arbitrary.
Then $a(ux) = b(vx)$ from which it follows that $a\mathcal{X} \cap b\mathcal{X} \neq \varnothing$.
But this contradicts the assumption that $a\mathcal{X} \cap b\mathcal{X} = \varnothing$.

(2) Let $aS \cap bS = \{c_{1}, \ldots, c_{m}\}S$ where $a \vee b = \{c_{1}, \ldots, c_{m}\}$.
We prove that $a\mathcal{X} \cap b\mathcal{X} = \{c_{1}, \ldots, c_{m}\}\mathcal{X}$.
Let $x \in a\mathcal{X} \cap b\mathcal{X}$.
Then $x = ax_{1} = bx_{2}$.
Thus by (A2), there exist $u_{1}, u_{2} \in S$ such that $au_{1} = bu_{2}$ and
$x_{1} = u_{1}x'$ and $x_{2} = u_{2}x'$ for some $x' \in \mathcal{X}$.
Thus $au_{1} = bu_{2} = c_{i}p$ for some $i$ and $p \in S$.
It follows that $x = ax_{1} = au_{1}x' = c_{i}px'$.
Thus $x \in \{c_{1}, \ldots, c_{m}\}\mathcal{X}$.
We now prove the reverse inclusion.
Let $c_{i}x \in \{c_{1}, \ldots, c_{m}\}\mathcal{X}$.
We can write $c_{i} = au = bv$ for some $u,v \in S$.
It follows that $c_{i}x = aux = bvx$.
Thus $c_{i}x \in a\mathcal{X} \cap b\mathcal{X}$.
\end{proof}

The following is immediate by part (1) of Lemma~\ref{lem:pepper}.

\begin{corollary}\label{cor:insect} Let $(S,\mathcal{X})$ be a system.
Then $a\mathcal{X} \cap b\mathcal{X} = \varnothing$ if and only if $a$ and $b$ are incomparable.
\end{corollary}

Let $(S,\mathcal{X})$ be a system.
Then the set $\beta = \{a\mathcal{X} \colon a \in S\}$ forms the basis for a topology on $\mathcal{X}$,
since if $x \in a\mathcal{X} \cap b\mathcal{X}$, then it follows from the axioms for a system that there is $c \in S$ such that
$x \in c\mathcal{X} \subseteq a\mathcal{X} \cap b\mathcal{X}$.
We endow $\mathcal{X}$ with the topology with $\beta$ as basis.
At this stage, we require nothing else than it is a topology.

\begin{lemma}\label{lem:peking} Let $(S,\mathcal{X})$ be a system.
For each $a \in S$, define $\lambda_{a} \colon \mathcal{X} \rightarrow a\mathcal{X}$ by $x \mapsto ax$.
This is a homeomorphism.
\end{lemma}
\begin{proof} By (A1), $\lambda_{a}$ is a bijection.
It clearly maps basis elements of the topology to basis elements.
We now look at inverse images.
Suppose that $b\mathcal{X} \subseteq a\mathcal{X}$.
Then, certainly, $a\mathcal{X} \cap b\mathcal{X} \neq \varnothing$.
Let $bx_{1} \in b\mathcal{X}$ be arbitrary.
Then $bx_{1} = ax_{2}$.
Thus by (A2), there exist $u_{1}, u_{2} \in S$ such that
$bu_{1} = au_{2}$ and $x_{1} = u_{1}x'$ and $x_{2} = u_{2}x'$
for some $x' \in \mathcal{X}$.
Observe that $\lambda_{a}^{-1}(bx_{1}) = \lambda_{a}^{-1}(au_{2}x') = u_{2}x'$.
On the other hand, $\lambda_{a}(u_{2}\mathcal{X}) = au_{2}\mathcal{X} = bu_{1}\mathcal{X} \subseteq b\mathcal{X}$.
Thus, the full inverse image of $b\mathcal{X}$ under $\lambda_{a}^{-1}$ is a union of sets of the form $u_{2}\mathcal{X}$.
It follows that $\lambda_{a}$ is a homeomorphism.
\end{proof}  

Let $\mathscr{X}$ be a topological space.
We denote by $\mathcal{I}(\mathscr{X})$ the inverse monoid of all
partial homeomorphisms between the open subsets of $\mathscr{X}$.

\begin{lemma}\label{lem:fruit} Let $(S,\mathcal{X})$ be a system.
Then $\lambda \colon S \rightarrow \mathcal{I}(\mathcal{X})$ is an injective monoid homomorphism.
\end{lemma}
\begin{proof} By Lemma~\ref{lem:peking} and the observation that $\lambda_{ab} = \lambda_{a} \lambda_{b}$,
it is clear that $\lambda$ defines a homomorphism and since $\lambda_{1}$ is the identity on $\mathcal{X}$ it is a monoid
homomorphism. It is injective by (A3).
\end{proof}

We can construct examples where $\lambda$ is not injective.

\begin{example}{\em This example is essentially due to Aidan Sims.
Let $S$ be a $k$-monoid in which the alphabet $X_{1} = \{e\}$ is a singleton
and suppose that $e$ commutes with every element of $X_{i}$ where $1 \leq i \leq k$.
Then for any element $a \in S$, we have that $ea = ae$ by the UFP.
Let $\mathscr{D}$ be any expanding subset of $S$.
We claim that $\mathscr{D}$ and $e\mathscr{D}$ are equivalent.
Let $x \in \mathscr{D}$ and $ey \in e\mathscr{D}$, where $y \in \mathscr{D}$.
By assumption, $xu = yv$ for some $u,v \in S$.
Thus $exu = eyv$ and so $x(eu) = (ey)v$. 
Let $w$ be the $k$-tiling determined by $\mathscr{D}$.
Then $1w = ew$ by Lemma~\ref{lem:brexit}.
But $w$ was arbitrary.
Thus $\lambda_{1} = \lambda_{e}$ in this case.}
\end{example}

There is an isomorphism of categories $\Omega_{k+l} \cong \Omega_{k} \times \Omega_{l}$.
It follows that the space of $(k+l)$-tilings in $S \times T$ is homeomorphism
to the space of functors $F$ from $\Omega_{k} \times \Omega_{l}$ to $S \times T$
such that $d(F((\mathbf{m}_{1}, \mathbf{n}_{1}),(\mathbf{m}_{2},\mathbf{n}_{2}))) = (\mathbf{m}_{2} - \mathbf{m}_{1}) \cdot (\mathbf{n}_{2} - \mathbf{n}_{1})$ where we use the notation introduced after Lemma~\ref{rem:free}.
There is a bijective correspondence between such functors and ordered pairs $(w_{1},w_{2})$
where $w_{1} \colon \Omega_{k} \rightarrow S$ is a $k$-tiling and $w_{2} \colon \Omega_{l} \rightarrow T$ is an $l$-tiling.
The proof of the following is now immediate by the above argument.

\begin{proposition}\label{prop:topology} Let $S$ be a $k$-monoid and let $T$ be an $l$-monoid.
Then $(S \times T)^{\infty}$ is homeomorphic with $S^{\infty} \times T^{\infty}$.
\end{proposition}

Denote the set of right-infinite strings over the alphabet $A_{n}$ is denoted by $A_{n}^{\omega}$.
The next result follows by Example~\ref{ex:tomb}, Proposition~\ref{prop:topology} and induction.

\begin{proposition}\label{prop:doodah} The space
$(A_{n_{1}}^{\ast} \times \ldots \times A_{n_{k}}^{\ast})^{\infty}$
is homeomorphic to the space
$A_{n_{1}}^{\omega} \times \ldots \times A_{n_{k}}^{\omega}$.
\end{proposition}

The following result was reported to us by Aidan Sims.

\begin{lemma}\label{lem:aperiodic} Let $S$ be a $k$-monoid and let $T$ be an $l$-monoid.
If $S$ and $T$ are both effective then $S \times T$ is effective.
\end{lemma}
\begin{proof} Let $(s_{1},t_{1}), (s_{2},t_{2}) \in S \times T$ be distinct elements.
Suppose first that $s_{1}$ and $s_{2}$ are distinct but $t_{1} = t_{2} = t$.
Now, $S$ is effective.
Thus there exists $s \in S$ such that $s_{1}s$ and $s_{2}s$ are not comparable.
It follows that $(s_{1},t)(s,1)$ and $(s_{2},t)(s,1)$ are not comparable.
A similar argument applies if $s_{1} = s_{2}$ but $t_{1}$ and $t_{2}$ are distinct.
Suppose now that $s_{1} \neq s_{2}$ and $t_{1} \neq t_{2}$.
Choose $s \in S$ such that $s_{1}s$ and $s_{2}s$ are not comparable and $t \in T$ such that $t_{1}t$ and $t_{2}t$ are not comparable.
Then $(s_{1},t_{1})(s,t)$ and $(s_{2},t_{2})(s,t)$ are not comparable.
\end{proof}

\section{The inverse monoid $\mathsf{R}(S)$}

In this section, we define the inverse monoid $\mathsf{R}(S)$ associated with an aperiodic $k$-monoid that will be used to build our group.

Let $S$ be any monoid.
A subset $R \subseteq S$ is said to be a {\em right ideal} if $RS \subseteq R$.
If $X$ is any subset of $S$ then $XS$ is a right ideal.
If $X$ consists of a single element $x$ then we write $xS$; this is called a {\em principal right ideal}.
Let $R_{1}$ and $R_{2}$ be right ideals of the monoid $S$.
A function $\alpha \colon R_{1} \rightarrow R_{2}$ is called a {\em morphism}
if $\alpha (as) = \alpha (a)s$ for all $a \in R_{1}$ and $s \in S$.\footnote{Morphisms are the analogues of (right) module homomorphisms.}
Given a $k$-monoid $S$, define
$\mathsf{R}(S)$ to consist of all the bijective morphisms between 
the finitely generated right ideals of $S$.

\begin{lemma}\label{lem:nato} Let $S$ be a $k$-monoid and let $\alpha \colon XS \rightarrow YS$ be a bijective morphism between two right ideals of $S$.
Then there is a subset $Z \subseteq YS$ such that $ZS = YS$ and
$\alpha$ induces a bijection between $X$ and $Z$.
\end{lemma}
\begin{proof} Put $Z = \alpha (X)$.
Clearly, $ZS \subseteq YS$.
Let $ys \in YS$. 
Then because $\alpha$ is a bijection, there exists $xt \in XS$ such that $\alpha (xt) = ys$.
Thus $\alpha (x)t = ys$.
But $\alpha(x) \in Z$.
We have therefore proved that $ys \in ZS$.
Thus $ZS = YS$.
The fact that $\alpha$ induces a bijection between $X$ and $Z$ is immediate.
\end{proof}

\begin{lemma}\label{lem:usa} Let $S$ be a $k$-monoid.
Let $\alpha \colon XS \rightarrow YS$ be a bijective morphism between two right ideals of $S$.
Let $ZS \subseteq XS$ be a finitely generated right ideal.
Then $\alpha (ZS)$ is a finitely generated right ideal.
\end{lemma}
\begin{proof} Since $\alpha$ is a morphism, it maps right ideals to right ideals.
Thus $\alpha (ZS)$ is a right ideal.
By Lemma~\ref{lem:nato}, we have that $\alpha (ZS) = \alpha (Z)S$.
It follows that $\alpha (ZS)$ is also finitely generated.
\end{proof}

Our next result simply establishes what we would expect.

\begin{proposition}\label{prop:truss} Let $S$ be a $k$-monoid.
Then  $\mathsf{R}(S)$ is an inverse monoid.
\end{proposition}
\begin{proof} In order that the intersection of two finitely generated right ideals be finitely generated, it is necessary
that the intersection of any two principal right ideals be either empty or finitely generated but this condition also implies that
the intersection of any two finitely generated ideals is finitely generated.
From this result and Lemma~\ref{lem:usa}, it follows that $\mathsf{R}(S)$ is an inverse monoid.
\end{proof}

We shall now show that the representation $\lambda \colon S \rightarrow \mathcal{I}(\mathcal{X})$
lifts to a representation of the inverse monoid $\mathsf{R}(S)$.
To do this, we shall need some properties of this monoid.
Let $S$ be a $k$-monoid and let $x,y \in S$.
Denote by $xy^{-1}$ the bijective morphism from $yS$ to $xS$ defined by $ya \mapsto xa$.
We shall call bijective morphisms of the form $xy^{-1}$ {\em basic}.

\begin{lemma}\label{lem:tom} Let $S$  be a $k$-monoid.
\begin{enumerate}
\item $xy^{-1} \leq uv^{-1}$ if and only if $(x,y) = (u,v)s$ for some $s \in S$.
It follows that if $xy^{-1}$ is an idempotent so too is $uv^{-1}$.
\item  $xx^{-1} \perp yy^{-1}$ if and only if $x$ and $y$ are incomparable in $S$.
\item $xx^{-1} yy^{-1} = \bigvee_{u \in x \vee y} uu^{-1}$.
\item Suppose that $bS \cap cS = \{x_{1}, \ldots, x_{m}\}S$.
Then
$$(ab^{-1})(cd^{-1}) = \bigvee_{i=1}^{n} ap_{i}(dq_{i})^{-1}$$
where $x_{i} = cq_{i} = bp_{i}$.
\end{enumerate}
\end{lemma}
\begin{proof} (1) By the definition of the order on partial functions, we have that $yS \subseteq vS$ and $xS \subseteq uS$.
In addition, $xy^{-1}$ and $uv^{-1}$ agree on elements of $yS$.
We have that $y = va$ and $x = ub$.
Now, $(xy^{-1})(y) = x$.
But $(uv^{-1})(y) = ua$.
It follows that $x = ua$ and so $ub = ua$ and so $a = b$.
The result now follows with $s = a = b$.
In order that $xy^{-1}$ be an idempotent, we must have that $x = y$.
It is therefore immediate that if $xy^{-1}$ is an idempotent then so too is $uv^{-1}$.

(2) The idempotents $xx^{-1}$ and  $yy^{-1}$ are orthogonal if and only if $xS \cap yS = \varnothing$.
But this is equivalent to saying that $x$ and $y$ are incomparable.

(3) The product of $xx^{-1}$ and $yy^{-1}$ is the identity function on $xS \cap yS$ which is the identity function on $(x \vee y)S$.

(4) We simply compute the composite of the partial functions $ab^{-1} \colon bS \rightarrow aS$ and $cd^{-1} \colon dS \rightarrow cS$.
\end{proof}

Basic morphisms are the building blocks of all elements of $\mathsf{R}(S)$.

\begin{lemma}\label{lem:trump} Let $S$ be a $k$-monoid.
Then each element of $\mathsf{R}(S)$ is a join of basic bijective morphisms.
\end{lemma}
\begin{proof} Let  $\theta \colon XS \rightarrow YS$ be a bijective morphism where $X$ and $Y$ are finite sets.
By Lemma~\ref{lem:nato}, we may assume that $\theta$ induces a bijection between $X$ and $Y$.
Let $x \in X$. Put $y_{x} = \theta (x)$.
Then $\theta (xs) = \theta (x)s = y_{x}s$.
It follows that the restriction of $\theta$ to $xS$ is a map of the form $y_{x}x^{-1}$.
Thus $\theta = \bigcup_{x \in X} y_{x}x^{-1}$.
\end{proof}

The following property of $\mathsf{R}(S)$ is simple, but important.

\begin{lemma}\label{lem:line} Let $S$ be a $k$-monoid.
Let $xy^{-1} \leq \bigvee_{j=1}^{n} u_{j}v_{j}^{-1}$ in $\mathsf{R}(S)$.
Then $xy^{-1} \leq u_{j}v_{j}^{-1}$ for some $j$.
\end{lemma}
\begin{proof} Clearly, $y \in v_{j}S$ for some $j$.
It follows that $yS \subseteq v_{j}S$.
Let $y = v_{j}m$ for some $m \in S$.
Then $(xy^{-1})(y) = x$ and $(u_{j}v_{j}^{-1})(v_{j}m) = u_{j}m$.
Thus $x = u_{j}m$.
It follows that $xS \subseteq u_{j}S$.
It is now clear that $xy^{-1} \leq u_{j}v_{j}^{-1}$. 
\end{proof}

\begin{proposition}\label{prop:cheese} Let $(S,\mathcal{X})$ be a system.
Then there is a monoid homomorphism $\chi \colon \mathsf{R}(S) \rightarrow \mathcal{I}(\mathcal{X})$
such that $\chi (ab^{-1}) = \lambda_{a}\lambda_{b}^{-1}$.
\end{proposition}
\begin{proof} Define $\chi (ab^{-1}) = \lambda_{a}\lambda_{b}^{-1}$.
By Lemma~\ref{lem:peking}, this is a partial homeomorphism of $\mathcal{X}$.
It has domain $b\mathcal{X}$, codomain $a\mathcal{X}$ and $bx \mapsto ax$.
We prove first that if $ab^{-1} \leq cd^{-1}$ then $\chi (ab^{-1}) \leq \chi (cd^{-1})$.
By part (1) of Lemma~\ref{lem:tom}, there is an element $s \in S$ such that
$a = cs$ and $b = ds$.
It is now routine to verify that
$\chi (ab^{-1}) \leq \chi (cd^{-1})$
using the fact that $\lambda_{s}\lambda_{s}^{-1}$ is an idempotent.
Now, let $ab^{-1}$ and $cd^{-1}$ be compatible elements.
We prove that $\chi (ab^{-1})$ and $\chi (cd^{-1})$ are compatible.
If $b\mathcal{X} \cap d\mathcal{X} = \varnothing$ then $ab^{-1}$ and $cd^{-1}$ are orthogonal and so
$b\mathcal{X} \cap d\mathcal{X} = \varnothing$ and $a\mathcal{X} \cap c\mathcal{X} = \varnothing$;
here we use the result that a pair of compatible elements with orthogonal domains have orthogonal ranges.
It follows that $\chi (ab^{-1})$ and $\chi (cd^{-1})$ are both orthogonal and so are compatible.
In the sequel, we may therefore assume that $b\mathcal{X} \cap d\mathcal{X} \neq \varnothing$.
Let $x \in b\mathcal{X} \cap d\mathcal{X}$.
Then $x = bx_{1} = dx_{2}$.
Thus by axiom (A2), there are elements $u_{1},u_{2} \in S$  and $x' \in \mathcal{X}$ such that
$bu_{1} = du_{2}$ and $x_{1} = u_{1}x'$ and $x_{2} = u_{2}x'$.
By definition,
$\chi (ab^{-1})(x) = ax_{1} = au_{1}x'$ and $\chi (cd^{-1})(x) = cx_{2} = cu_{2}x'$.
But $ab^{-1}$ and $cd^{-1}$ are compatible and so $au_{1} = cu_{2}$.
It follows that  
$\chi (ab^{-1})(x) = \chi (cd^{-1})(x)$.
By symmetry, we deduce that $\chi (ab^{-1})$ and $\chi (cd^{-1})$ are compatible.
We can now extend $\chi$ to the whole of $\mathsf{R}(S)$.
By Lemma~\ref{lem:trump}, a typical element of $\mathsf{R}(S)$
has the form $\bigvee_{i=1}^{n} a_{i}b_{i}^{-1}$.
Define
$$\chi \left( \bigvee_{i=1}^{n} a_{i}b_{i}^{-1} \right)     
=
\bigcup_{i=1}^{n} \chi (a_{i}b_{i}^{-1}).$$
The right-hand side is certainly a well-defined element.
It remains to show that the map $\chi$ as defined is actually well-defined.
Suppose that  
$$\bigvee_{i=1}^{n} a_{i}b_{i}^{-1} 
=
\bigvee_{j=1}^{m} c_{j}d_{j}^{-1}.$$ 
By Lemma~\ref{lem:line}, for each $i$ there exists a $j$ such that
$a_{i}b_{i}^{-1} \leq c_{j}d_{j}^{-1}$
and conversely.
From this it quickly follows that our definition of $\chi$ is indeed well-defined.
Finally, we prove that $\chi$ is a monoid homomorphism.
To do this, we need only prove that
$\chi (ab^{-1})\chi (cd^{-1}) = \chi (ab^{-1} \cdot cd^{-1})$.
We use part (4) of Lemma~\ref{lem:tom}.
We therefore need to prove that
$$\chi (ab^{-1}) \chi (cd^{-1}) = \chi \left( \bigvee_{i=1}^{n} ap_{i}(dq_{i})^{-1}\right)$$
where $bS \cap cS = \{x_{1}, \ldots, x_{m}\}S$
and $x_{i} = cq_{i} = bp_{i}$.
We check first that the map on the left-hand side
has the same domain as the map on the right-hand side.
By Lemma~\ref{lem:pepper}, we have that
$b\mathcal{X} \cap c\mathcal{X} = \{x_{1},\ldots, x_{m}\}\mathcal{X}$.
Thus the domain of the left-hand side is 
$\{dq_{1}, \ldots, dq_{m}\}\mathcal{X}$ which is the same as the domain of the right-hand side.
The result now follows by a routine calculation.
\end{proof}

\section{Definition of the group $\mathscr{G}(S)$.}

In this section, we define the group associated with a $k$-monoid.
The properties of this group will be investigated in more detail in subsequent sections.

We use the inverse monoid $\mathsf{R}(S)$ defined in the previous section.
We shall construct our group, not from $\mathsf{R}(S)$ itself, but from a certain inverse subsemigroup.
Let $T$ be an inverse semigroup.
A non-zero idempotent $e$ of $T$ is said to be {\em essential} if $ef \neq 0$ for all non-zero
idempotents $f$ of $S$.
Denote by $T^{e}$ the set of all $s \in T$ such that $s^{-1}s$ and $ss^{-1}$ are essential.
It follows by \cite[Lemma~4.2]{Lawson2007}, that $T^{e}$ is an inverse semigroup (without zero).

\begin{lemma}\label{lem:flower} Let $S$ be $k$-monoid.
Then the idempotent in $\mathsf{R}(S)$ associated with the finitely generated right ideal $XS$, that is the identity function defined on $XS$, 
is essential if and only
if for every $s \in S$ there exists an $x \in X$ such that $sa = xb$ for some $a,b \in S$.
\end{lemma}
\begin{proof} A right ideal $XS$ is essential if and only if it has a non-empty intersection with every principal
right ideal $sS$.
The result is now immediate.
\end{proof}

We now come to the key definition of this paper.
\vspace{5mm}
\begin{center}
\fbox{\begin{minipage}{15em}
{\bf Definition.} Let $S$ be a $k$-monoid. 
Then 
$$\mathscr{G}(S) = \mathsf{R}(S)^{e}/\sigma$$
is the  {\em group associated with $S$}.
\end{minipage}}
\end{center}
\vspace{5mm}

\begin{remark}
{\em The above process for constructing a group from an inverse monoid
of partial bijections is in fact identical to the one used in \cite{YC}.
However, we stress that the group defined in our paper is quite different from the one defined in \cite{YC}.
Observe that we can avoid the explicit use of the congruence $\sigma$ by simply defining
two elements of $\mathsf{R}(S)^{e}$ to be equal if they agree on an essential finitely generated right ideal.
}
\end{remark}

\section{The group $\mathscr{G}(S)$ in terms of maximal generalized prefix codes.}

In Section~6, we defined the group $\mathscr{G}(S)$ associated with any $k$-monoid.
In this section, we shall show that this group can be defined in a different way
that will make it much easier to represent it in geometric terms later.
As a first step, we shall construct an inverse submonoid of $\mathsf{R}(S)$.
Let $S$ be a $k$-monoid.
We say that the finitely generated right ideal $XS$ is {\em projective} if $X$ is a generalized prefix code.

\begin{lemma}\label{lem:beer}  Let $S$ be a $k$-monoid.
Suppose that $XS = YS$ where $X$ and $Y$ are generalized prefix codes.
Then $X = Y$.
\end{lemma}
\begin{proof} Let $x \in X$.
Then $x = ym$ for some $y \in Y$ and $m \in S$.
Also, $y = x'm'$ for some $x' \in X$ and $m' \in S$.
Therefore $x = x'm'm$.
But $x$ and $x'$ are incomparable if distinct.
Thus $x = x'$ and so $m'm = 1$ by cancellation.
However, $S$ is conical and so $m = m' = 1$.
It follows that $x = y$.
We have proved that $X \subseteq Y$.
By symmetry, $Y \subseteq X$ and so $X = Y$.
\end{proof}

Lemma~\ref{lem:beer} shows that there is a bijection between projective right ideals and generalized prefix codes.

\begin{lemma}\label{lem:Jones} Let $S$ be a $k$-monoid.
Then the intersection of any two finitely generated projective right ideals of $S$
is a finitely generated projective right ideal of $S$.
\end{lemma} 
\begin{proof} Let $XS$ and $YS$ be two finitely generated projective right ideals
where $X$ and $Y$ are generalized prefix codes.
We shall assume that $XS \cap YS \neq \varnothing$, otherwise there is nothing to prove.
By assumption, for each $x \in X$ and $y \in Y$ either $xS \cap yS = \varnothing$ or 
$xS \cap yS = Z_{(x,y)}S$ where $Z_{(x,y)}$ is a finite set of incomparable elements.
It follows that $XS \cap YS = \left( \bigcup_{(x,y) \in X \times Y} Z_{(x,y)} \right) S$.\footnote{Be aware that some of the sets $Z_{(x,y)}$
could be empty.}
We prove that $\bigcup_{(x,y) \in X \times Y} Z_{(x,y)}$ is a generalized prefix code.
Let $a \in Z_{(x_{i}, y_{k})}$ and $b \in Z_{(x_{j}, y_{l})}$
where $x_{i}, x_{j} \in X$ and $y_{k}, y_{l} \in Y$.
Suppose that $aS \cap bS \neq \varnothing$.
Then $am_{1} = bm_{2}$ for some $m_{1},m_{2} \in S$.
But $a = up$, where $u \in x_{i}S \cap y_{k}S$, and $b = vq$, where $v \in x_{j}S \cap y_{l}S$.
It follows that $x_{i} = x_{j}$ and $y_{k} = y_{l}$ by virtue of the fact that the elements are incomparable.
\end{proof}

\begin{lemma}\label{lem:comet} Let $S$ be a $k$-monoid.
Each bijective morphism from $XS$ to $YS$, where $X$ and $Y$ are generalized prefix codes,
is determined by a bijection from $X$ to $Y$, and vice versa.
\end{lemma}
\begin{proof} Suppose first that $\alpha \colon X \rightarrow Y$ is a bijection.
Define $\alpha' \colon XS \rightarrow YS$ by $\alpha' (xm) = \alpha (x)m$.
We need first to show that $\alpha'$ is a well-defined function.
Suppose that $xm = x'm'$ where $x,x' \in X$.
Then, since $X$ is a generalized prefix code, we must have that $x = x'$ and so, by cancellation, $m = m'$.
It follows that $\alpha'$ is well-defined.
It is a morphism by construction.
We show that it is a bijection.
Suppose that $\alpha'(xm) = \alpha' (x'm')$.
Then $\alpha (x)m = \alpha (x')m'$.
By assumption, $\alpha (x), \alpha (x') \in Y$ which is a generalized prefix code.
It follows that $\alpha (x) = \alpha (x')$ and so, since $\alpha$ is a bijection, we have that $x = x'$.
By cancellation, we then get that $m = m'$ and so $xm = x'm'$.
It follows that $\alpha'$ is injective.
The fact that it is a surjection is immediate.

Suppose now that $\alpha \colon XS \rightarrow YS$ is a bijective morphism.
We claim that $\alpha$ induces a bijection between $X$ and $Y$.
Let $x \in X$.
Then $\alpha (x) = ym$.
Take $\alpha^{-1}$ of both sides and we get that
$x = \alpha^{-1}(y)m$.
Let $\alpha^{-1}(y) = x'n$.
Then $x = x'nm$.
But $x$ and $x'$ are incomparable unless $x = x'$ in which case $1 = nm$.
We now apply the fact that $S$ is conical to get that $m = n = 1$.
\end{proof}

\begin{lemma}\label{lem:wine}
Let $S$ be a $k$-monoid.
Let $\alpha \colon XS \rightarrow YS$ be a bijective morphism between two finitely generated right ideals
and let $Z \subseteq XS$ be a generalized prefix code.
Then $\alpha (Z)$ is a generalized prefix code.
\end{lemma}
\begin{proof}
Suppose that $z,z' \in Z$ are such that $\alpha (z)S \cap \alpha (z')S \neq \varnothing$.
Then $\alpha (z)m = \alpha (z')m'$ for some $m,m' \in S$.
But $\alpha$ is a morphism and so $\alpha (zm) = \alpha (z'm')$.
By injectivity, we have that $zm = z'm'$ and, by assumption, $z$ and $z'$ are supposed to be incomparable.
Thus $z = z'$.
It follows that $\alpha (z) = \alpha (z')$.
This proves that $\alpha (Z)$ is also a generalized prefix code.
\end{proof}

Denote by $\mathsf{P}(S)$ the set of all bijective morphisms between finitely generated projective right ideals of $S$.
By Lemma~\ref{lem:Jones} and Lemma~\ref{lem:wine}, it follows that
$\mathsf{P}(S) \subseteq \mathsf{R}(S)$
and that, in fact, we have proved the following.

\begin{proposition}\label{prop:ophelia}
Let $S$ be a $k$-monoid.
The set of all bijective morphisms between finitely generated projective right ideals of $S$
is an inverse submonoid $\mathsf{P}(S)$ of $\mathsf{R}(S)$.
\end{proposition}

The following lemma is simple, but important.

\begin{lemma}\label{lem:birget} Let $S$ be a $k$-monoid.
Then $\mathsf{P}(S)^{e} \subseteq \mathsf{R}(S)^{e}$.
\end{lemma}
\begin{proof} The result follows from the simple observation that principal right ideals $aS$
are projective.
\end{proof}

\begin{lemma}\label{lem:otan} Let $S$ be a $k$-monoid.
Let $\alpha \colon XS \rightarrow YS$ be a bijective morphism between two essential, finitely generated right ideals of $S$
such that $\alpha$ induces a bijection between $X$ and $Y$.
Suppose that $ZS \subseteq XS$ is such that $Z$ is a maximal generalized prefix code.
Then $\alpha (Z)$ is a maximal generalized prefix code.
\end{lemma} 
\begin{proof} Let $Z = \{z_{1}, \ldots, z_{m}\}$.
Suppose that $\alpha (z_{i})$ and $\alpha (z_{j})$ are comparable.
Then $\alpha (z_{i})a = \alpha (z_{j})b$.
But $\alpha$ is a morphism, so $\alpha (z_{i}a) = \alpha (z_{j}b)$.
Thus $z_{i}a = z_{j}b$.
But $Z$ is a generalized prefix code and so $i = j$ and $z_{i} = z_{j}$.
We have therefore proved that $\alpha (Z)$ is a generalized prefix code.
We now prove that $\alpha (Z)$ is maximal.
Let $a \in S$ be any element.
Then $aS \cap YS \neq \varnothing$.
Thus $ab = ys$ for some $y \in Y$ and $b,s \in S$.
It follows that $\alpha^{-1}(a)b = \alpha^{-1}(y)s$.
Since $ZS$ is an essential ideal we have that $\alpha^{-1}(y)sd = zf$.
Thus $ysd = \alpha (z)f$
and so $abd = \alpha (z)f$. 
It follows that $aS \cap \alpha (z)S \neq \varnothing$, as claimed.
\end{proof}

\begin{lemma}\label{lem:bliss} Let $XS$ be an essential finitely generated right ideal in a $k$-monoid $S$.
Then there is a maximal generalized prefix code $Z \subseteq XS$.
\end{lemma}
\begin{proof} Put $\mathbf{n} = \bigvee_{x \in X}d(x)$.
Recall that $C_{\mathbf{n}}$ is the set of all elements $s \in S$ such that $d(s) = \mathbf{n}$.
This is a maximal generalized prefix code by Lemma~\ref{lem:homogeneous}.
Put $Z = C_{\mathbf{n}}$.
Let $z \in Z$.
Then $zS \cap XS \neq \varnothing$.
Thus $za = xb$ for some $x \in X$ and $a,b \in S$.
By definition, $d(z) \geq d(x)$.
It follows that $z = xt$ for some $t \in S$ by Lemma~\ref{lem:levi}.
Thus $Z \subseteq XS$.
\end{proof}

\begin{lemma}\label{lem:tony} 
Each element of $\mathsf{R}(S)^{e}$ extends an element of  $\mathsf{P}(S)^{e}$.
\end{lemma}
\begin{proof} Let $\alpha \colon XS \rightarrow YS$ be a bijective morphism between two essential, finitely generated right ideals of $S$.
By Lemma~\ref{lem:bliss}, there is a maximal generalized prefix code $Z \subseteq XS$.
By lemma~\ref{lem:otan}, the bijective morphism $\alpha | ZS$, the restriction of $\alpha$ to the set $ZS$, 
maps a finitely generated projective right ideal to a finitely generated projective right ideal.
It is therefore an element of  $\mathsf{P}(S)^{e}$.
\end{proof}

\begin{lemma}\label{lem:benn} Let $U$ be an inverse subsemigroup of the inverse semigroup $V$.
Suppose that each element of $V$ lies above an element of $U$ in the natural partial order.
Then $V/\sigma_{V} \cong U/\sigma_{U}$.
\end{lemma}
\begin{proof} We denote the $\sigma_{V}$-congruence class containing the element $s$ by $[s]_{V}$
and the $\sigma_{U}$-congruence class containing the element $s$ by $[s]_{U}$.
There is a homomorphism $U \rightarrow V/\sigma_{V}$ given by $s \mapsto [s]_{V}$.
This map is surjective since each element of $V$ is above an element in $U$.
Suppose that $s$ and $s'$ map to the same element under this map.
Then, by definition, there is an element $t \in V$ such that $t \leq s,s'$.
By assumption, there exists $s'' \leq t$ where $s'' \in U$.
It follows that $[s]_{U} = [s']_{U}$.
The result now follows.
\end{proof}

By Lemma~\ref{lem:tony} and Lemma~\ref{lem:benn}, we have therefore proved the following.
\vspace{5mm}
\begin{center}
\fbox{\begin{minipage}{15em}
\begin{theorem}\label{them:tea} Let $S$ be a $k$-monoid.
Then  
$$\mathsf{P}(S)^{e}/\sigma \cong \mathsf{R}(S)^{e}/\sigma.$$
\end{theorem}
\end{minipage}}
\end{center}
\vspace{5mm}
The above theorem tells us that the two ways in which we might have defined our group $\mathscr{G}(S)$
give the same answer.

\begin{remark}
{\em The above theorem implies that the structure of the group $\mathscr{G}(S)$
is intimately connected with the structure of the maximal generalized prefix codes in $S$.}
\end{remark}

\section{A geometric representation of the group $\mathscr{G}(S)$.}

Let $S$ be an aperiodic $k$-monoid.
By proposition~\ref{prop:cheese} and Proposition~\ref{prop:bojo},
there is a monoid homomorphism $\chi \colon \mathsf{R}(S) \rightarrow \mathcal{I}^{\scriptsize cl}(S^{\infty})$,
the inverse monoid of all partial homeomorphisms between the clopen subsets of $S^{\infty}$.
We are first of all interested in which elements of $\mathsf{R}(S)$ map to bijections in $\mathcal{I}^{\scriptsize cl}(S^{\infty})$.

\begin{lemma}\label{lem:boris} Let $S$ be an aperiodic $k$-monoid.
Let $\iota$ be the identity function on the finitely generated right ideal $AS$.
Then $\chi (\iota)$ is the identity on $S^{\infty}$ if and only if $AS$ is an essential ideal of $S$.
\end{lemma}
\begin{proof} We prove that $AS^{\infty} = S^{\infty}$ if and only if $AS$ is an essential right ideal of $S$.
Suppose first that $AS$ is an essential right ideal.
Let $x \in S^{\infty}$.
Choose a corner $b$ of $x$ which is larger than any element of $A$.
Then $bS \cap AS \neq \emptyset$ and so $bu = av$ for some $a \in A$.
But, by assumption, $d(b) \geq d(a)$.
It follows by Lemma~\ref{lem:levi}, that $b = at$ for some $t \in S$.
It follows that $a$ is a corner of $x$ and so $x \in aS^{\infty}$.
Thus $AS^{\infty} = S^{\infty}$.
To prove the converse, suppose that  $AS^{\infty} = S^{\infty}$.  
We prove that $AS$ is an essential ideal.
Suppose not.
Then there is an element $b \in S$ such that $AS \cap bS = \varnothing$.
However, $b$ must be a corner of some element $w \in aS^{\infty}$ for some $a \in A$.
But this implies that $b$ is comparable with $a$
and so $bS \cap aS \neq \varnothing$.
This is a contradiction.
\end{proof}

By the above lemma, it follows that the elements of $\mathsf{R}(S)$ that are mapped to bijections of $S^{\infty}$ under $\chi$
are precisely the elements of $\mathsf{R}(S)^{e}$.

\begin{proposition}\label{prop:tizer} Let $S$ be a $k$-monoid.
Then  $\mathsf{R}(S)^{e}$ is $E$-unitary.
\end{proposition}
\begin{proof} Let $\alpha \colon XS \rightarrow YS$ be a bijective morphism between
two finitely generated essential right ideals.
Suppose that $\alpha$ is the identity when restricted to the finitely generated essential right ideal $ZS$ where $ZS \subseteq XS$.
Let $xa \in XS$.
Then, since $ZS$ is an essential right ideal, we have that $xaS \cap ZS \ne \varnothing$.
It follows that $xab = zc$ for some $b,c \in S$ and $z \in Z$.
But $\alpha (zc) = zc$ and so $\alpha (xab) = xab$.
But $\alpha$ is a morphism and so $\alpha (xab) = \alpha (x)ab$.
By cancellation, it follows that $\alpha (x) = x$.
We have therefore proved that $\alpha$ is the identity on $X$ and so $\alpha$ is also an idempotent. 
\end{proof}

The above result makes the proof of the following much easier.

\begin{proposition}\label{prop:tresbon}
Let $S$ be an aperiodic $k$-monoid.
Let
$$\alpha = x_{1}y_{1}^{-1} \vee \ldots \vee x_{m}y_{m}^{-1}
\text{ and }
\beta = u_{1}v_{1}^{-1} \vee \ldots \vee u_{n}v_{n}^{-1}$$
be two elements of  $\mathsf{P}(S)^{e}$ such that $\chi (\alpha) = \chi (\beta)$.
Then $\alpha \, \sigma \, \beta$.
\end{proposition}
\begin{proof} By Proposition~\ref{prop:E} and Proposition~\ref{prop:tizer},
we need only prove that $\alpha \, \sim \, \beta$.
By symmetry, it is enough to prove that elements of the form $y_{i}x_{i}^{-1}u_{j}v_{j}^{-1}$ are idempotents.
Let $s \in S$ be such that $(y_{i}x_{i}^{-1}u_{j}v_{j}^{-1})(s)$ is defined.
Then $s = v_{j}s_{1}$ for some $j$ and $u_{j}s_{1} = x_{i}s_{2}$ for some $i$.
In particular, 
$(y_{i}x_{i}^{-1}u_{j}v_{j}^{-1})(s) = y_{i}s_{2}$.
By assumption, $\chi (\alpha) = \chi (\beta)$ and so
$\chi (\alpha)^{-1} = \chi (\beta)^{-1}$. 
Thus for all $w \in S^{\infty}$, and using the fact that $u_{j}s_{1} = x_{i}s_{2}$, we have that
$$\chi (\beta)^{-1}(u_{j}s_{1}w) = v_{j}s_{1}w = \chi (\alpha)^{-1}(x_{i}s_{2}w) = y_{i}s_{2}w.$$
By Corollary~\ref{cor:good}, we deduce that $v_{j}s_{1} = y_{i}s_{2}$.
Thus $(y_{i}x_{i}^{-1}u_{j}v_{j}^{-1})(s) = s$ when defined.
It follows that $y_{i}x_{i}^{-1}u_{j}v_{j}^{-1}$ is an idempotent.
\end{proof}

We have therefore proved the following.
Recall that we usually make assumptions so that $S^{\infty}$ is the Cantor space.

\begin{theorem}\label{them:rice} Let $S$ be an aperiodic $k$-monoid.
The the group $\mathscr{G}(S)$ is isomorphic to a subgroup of the group of all self-homeomorphisms $S^{\infty}$.
\end{theorem}

\begin{remark}
{\em We can think of the elements of the group $\mathscr{G}(S)$ in the following concrete terms.
We use the results of Section 7.
Let $A = (a_{1}, \ldots, a_{m})$ and $B = (b_{1}, \ldots, b_{m})$ be two ordered maximal generalized prefix codes.
Then the sets $\{a_{i}S^{\infty} \colon 1 \leq i \leq m\}$ and $\{b_{i}S^{\infty} \colon 1 \leq i \leq m\}$
are both partitions of $S^{\infty}$ by Corollary~\ref{cor:insect} and Lemma~\ref{lem:boris}.
We define a function $f_{(B,A)} \colon S^{\infty} \rightarrow S^{\infty}$ by $a_{i}w \mapsto b_{i}w$ where $w \in S^{\infty}$.
Then $f_{(B,A)}$ is a self-homeomorphism of $S^{\infty}$.
The group $\mathscr{G}(S)$ is isomorphic to the totality of all such maps.}
\end{remark}

\section{The group $\mathscr{G}(S)$ as a group of units}

In this section, we shall define a congruence $\equiv$ on the inverse monoid $\mathsf{R}(S)$
and show that the quotient is a Boolean inverse monoid.
We then show that our group $\mathscr{G}(S)$ is exactly the group of units of $\mathsf{R}(S)/\equiv$.

Let $T$ be an inverse $\wedge$-semigroup with zero.
Define the relation $\equiv$ on $T$ as follows:
$s \equiv t$ if and only if for each $0 < x \leq s$ we have that $x \wedge t \neq 0$
and for each $0 < y \leq t$ we have that $y \wedge s \neq 0$.
By \cite{Lenz, JL}, we have the following.

\begin{lemma}\label{lem:toy} Let $T$ be an arbitrary inverse $\wedge$-monoid with zero.
Then $\equiv$ is a $0$-restricted congruence on $T$.
\end{lemma}

This congruence is defined out of the blue, but
we now describe some of its properties
which will make it much more natural.
Let $a \leq b$.
We write $a \leq_{e} b$ and say that {\em $a$ is essential in $b$} if
$0 < x \leq b$ implies that $a \wedge x \neq 0$.

\begin{lemma}\label{lem:phone} Let $T$ be an inverse $\wedge$-semigroup with zero.
Then $a \leq_{e} b$ if and only if $\mathbf{d}(a) \leq_{e} \mathbf{d}(b)$.
\end{lemma}
\begin{proof} Assume first that  $a \leq_{e} b$.
Then $\mathbf{d}(a) \leq \mathbf{d}(b)$.
Let $0 < e \leq \mathbf{d}(b)$.
Then $0 < eb \leq b$.
Thus $a \wedge eb \neq 0$.
But $a$ and $eb$ are compatible and so by \cite[Lemma~1.4.11]{Lawson1998}.
we have that $\mathbf{d} (a \wedge eb) = \mathbf{d}(a) \wedge e$, which is therefore non-zero.
Conversely, suppose that  $\mathbf{d}(a) \leq_{e} \mathbf{d}(b)$.
Let  $0 < x \leq b$.
Then $0 < \mathbf{d}(x) \leq \mathbf{d}(b)$,
By assumption, $\mathbf{d}(x) \wedge \mathbf{d}(a) \neq 0$.
But $a$ and $x$ are compatible and so $\mathbf{d}(x \wedge a) = \mathbf{d}(x) \wedge \mathbf{d}(a) \neq 0$
again by \cite[Lemma~1.4.11]{Lawson1998}.
Thus $x \wedge a \neq 0$.\end{proof}

\begin{lemma}\label{lem:story} Let $T$ be an arbitrary inverse $\wedge$-monoid with zero.
Then $a \equiv b$ if and only if $a \wedge b \leq_{e} a,b$.
\end{lemma}
\begin{proof} Suppose first that $a \equiv b$ where $a,b \neq 0$.
Then since $a \leq a$ we have that $a \wedge b \neq 0$.
Clearly, $a \wedge b \leq a$.
Let $0 < x \leq a$.
Then, by definition, $b \wedge x \neq 0$.
But $x = x \wedge a$.
It follows that $x \wedge (a \wedge b) \neq 0$.
We have therefore proved that $a \wedge b \leq_{e} a$.
By symmetry, we also have that $a \wedge b \leq_{e} b$, as required.
To prove the converse, suppose that $a \wedge b \leq_{e} a,b$.
Let $0 < x \leq a$.
Then $x \wedge (a \wedge b) \neq 0$.
It follows that $x \wedge b \neq 0$.
By symmetry, we deduce that $a \equiv b$.
\end{proof}

We can now say something specific about the congruence $\equiv$.

\begin{lemma}\label{lem:four} Let $T$ be an arbitrary inverse $\wedge$-monoid with zero.
Let $\rho$ be any congruence such that $a \leq_{e} b$ implies that $a \, \rho \, b$.
Then $\equiv \, \subseteq \, \rho$.
\end{lemma}
\begin{proof} Suppose that $a \equiv b$.
Then by Lemma~\ref{lem:story}, we have that $a \wedge b \leq_{e} a,b$.
By assumption, $(a \wedge b) \, \rho \, a$ and $(a \wedge b) \, \rho \, b$.
It follows that $a \, \rho \, b$.
\end{proof}

We define a congruence $\rho$ to be {\em essential} if  $a \leq_{e} b$ implies that $a \, \rho \, b$.

\begin{lemma}\label{lem:waller} Let $T$ be an arbitrary inverse $\wedge$-monoid with zero.
Let $\rho$ be a $0$-restricted idempotent-pure essential congruence on $T$.
Then $\rho \,=\, \equiv$.
\end{lemma}
\begin{proof} Let $a \rho b$.
Then $a \sim b$ because $\rho$ is idempotent-pure.
It follows that $a \wedge b = ab^{-1}b$;
this can be proved directly or see \cite[Lemma~1.4.12]{Lawson1998}.
It follows that $ab^{-1}b \, \rho \, b$ and so $(a \wedge b)\, \rho \, b$.
Let $0 < x \leq b$.
Then  $(a \wedge b)x^{-1}x\, \rho \, bx^{-1}x = x$.
Since $x$ is non-zero and $\rho$ is $0$-restricted, it follows that $(a \wedge b)x^{-1}x \neq 0$.
Now $x, a \wedge b \leq b$ and so $x \sim (a \wedge b)$.
It follows that $(a \wedge b) \wedge x = (a \wedge b)x^{-1}x$.
We have therefore shown that if $0 < x \leq b$ then $(a \wedge b) \wedge x \neq 0$.
We have therefore proved that $(a \wedge b) \leq_{e} b$.
By symmetry, $(a \wedge b) \leq_{e} a$.
Thus by Lemma~\ref{lem:story} we have that $a \equiv b$.
\end{proof}

We have therefore proved the following theorem which demonstrates the significance of the congruence $\equiv$.

\begin{theorem}\label{them:phoebe} Let $T$ be an arbitrary inverse $\wedge$-monoid with zero
in which $\equiv$ is idempotent-pure.
Then $\equiv$ is the unique congruence on $T$ which is $0$-restricted, idempotent-pure and essential.
\end{theorem}

The following now connects the congruence $\equiv$ with Proposition~\ref{prop:tizer}.

\begin{lemma}\label{lem:villanelle} Let $T$ be an arbitrary inverse $\wedge$-monoid with zero
in which $\equiv$ is idempotent-pure.
Suppose that $T^{e}$ is $E$-unitary.
Then the restriction of the congruence $\equiv$ to $T^{e}$ is the minimum group congruence on $S^{e}$.
\end{lemma}
\begin{proof} Observe that $a \in T^{e}$ if and only if $a^{-1}a \equiv 1$ and $aa^{-1} \equiv 1$.
Suppose that $a, b \in T^{e}$ and $a \equiv b$.
Then, since $\equiv$ is idempotent-pure, we have that $a \, \sim \, b$ and so $a \, \sigma \, b$,
since in an $E$-unitary inverse semigroup $\sigma \, = \, \sim$ \cite[Theorem 2.4.6]{Lawson1998}.
Suppose that $a \, \sigma \, b$ in $T^{e}$.
Then there exists $c \in T^{e}$ such that $c \leq a, b$.
Now $\mathbf{d}(c) \equiv 1$ and $\mathbf{d}(a) \equiv 1$
and so $\mathbf{d}(c) \equiv \mathbf{d}(a)$.
It follows that $\mathbf{d}(c)\mathbf{d}(a) \leq_{e} \mathbf{d}(a)$.
That is, $\mathbf{d}(c) \leq_{e} \mathbf{d}(a)$.
Thus by Lemma~\ref{lem:phone}, we have that $c \leq_{e} a$ and, by symmetry, $c \leq_{e} b$.
It follows by Lemma~\ref{lem:story}, that $a \, \equiv \, b$, as required.
\end{proof}

We have therefore proved the following.

\begin{proposition}\label{prop:delay} Let $T$ be an inverse $\wedge$-monoid with zero
in which $\equiv$ is idempotent-pure and $T^{e}$ is $E$-unitary.
Then the group of units of $T/\equiv$ is isomorphic to $T^{e}/\sigma$.
\end{proposition}

\begin{proposition}\label{prop:salami} Let $S$ be a $k$-monoid.
Then  $\mathsf{R}(S)$ is a distributive inverse monoid.
\end{proposition}
\begin{proof} We proved in Proposition~\ref{prop:truss} that $\mathsf{R}(S)$ is an inverse monoid.
If $XS$ and $YS$ are two finitely generated right ideals then their union is $(X \cup Y)S$
which is also a finitely generated right ideal.
It is now immediate that  $\mathsf{R}(S)$ is a distributive inverse monoid.
\end{proof}

In fact, the inverse monoid $\mathsf{R}(S)$ has all binary meets, a consequence of the following result.

\begin{lemma}\label{lem:gove} Let $S$ be a $k$-monoid.
Then the fixed point set of a bijective morphism $\theta \colon XS \rightarrow YS$ is a finitely generated right ideal. 
\end{lemma}
\begin{proof} Let $a \in XS$ be such that $\theta (a) = a$.
By assumption, $a = xs$ for some $x \in X$ and $s \in S$.
It follows that $\theta (x)s = xs$ and so, by right cancellation, we have that $\theta (x) = x$.
Define $X' = \{x \in X \colon \theta (x) = x\}$.
Then the fixed point set of $\theta$ is precisely the set $X'S$.
\end{proof}

By Lemma~\ref{lem:gove} and \cite[Theorem 1.9]{Leech}, we have the following.

\begin{corollary}\label{lem:borisy} Let $S$ be a $k$-monoid.
Then  $\mathsf{R}(S)$ is a $\wedge$-monoid.
\end{corollary}

\begin{lemma}\label{lem:marina} Let $S$ be an inverse semigroup and let $\rho$ be an idempotent-pure congruence on $S$.
\begin{enumerate}

\item If $S$ is distributive then $S/\rho$ is distributive and the natural map $S \rightarrow S/\rho$ is a morphism.

\item If $S$ is a $\wedge$-semigroup then $S/\rho$ is a $\wedge$-semigroup and the natural map $S \rightarrow S/\rho$ preserves meets.

\end{enumerate}
\end{lemma}
\begin{proof} (1) Denote the $\rho$-class containing $a$ by $[a]$.
We prove first that for idempotents $e$ and $f$ we have that $[e \vee f] = [e] \vee [f]$.
Clearly, $[e] \vee [f] \leq [e \vee f]$.
Suppose that $[a]$ is any idempotent such that $[e], [f] \leq [a]$.Because $\rho$ is idempotent-pure, we know that $[a] = [i]$, where $i$ is an idempotent.
Thus $[e], [f] \leq [i]$.
It follows that $[e] = [ei]$ and $[f] = [fi]$.
Thus $[e \vee f] = [ei \vee fi] = [(e \vee f)i] \leq [i]$.
It follows that $[e \vee f] = [e] \vee [f]$. 
Suppose, now, that $[a] \sim [b]$.
Then, since $\rho$ is idempotent-pure we have that $a \sim b$.
Thus $a \vee b$ is defined.
We claim that $[a \vee b] = [a] \vee [b]$.
Since $a,b \leq a \vee b$ we have that $[a], [b] \leq [a \vee b]$.
By our result on idempotents above, we have that $[\mathbf{d}(a \vee b)] = [\mathbf{d}(a)] \vee [\mathbf{d}(b)]$.
Suppose that $[a], [b] \leq [c]$.
It is easy to check that $[c][\mathbf{d}(a) \vee \mathbf{d}(b)] = [c\mathbf{d}(a) \vee c \mathbf{d}(b)] = [a \vee b]$.
Thus $[a] \vee [b] = [a \vee b]$.

(2) Let $[e] \leq [a]$.
Then $[e] = [ea]$.
Since $\rho$ is idempotent-pure, we have that $ea$ is an idempotent and clearly $ea \leq a$.
It follows that $ea \leq \phi (a)$, where $\phi$ is the operator introduced in Section~2, and so $[e] = [ea] \leq [\phi (a)]$.
We have therefore proved that $\phi ([a]) = [\phi (a)]$.
That is, $S/\rho$ is a $\wedge$-semigroup.

It remains to prove that the natural map preserves meets.
Let $a,b \in S$.
Then $a \wedge b = \phi (ab^{-1})b$.
It follows that 
$$[a \wedge b] = [\phi (ab^{-1})b] = [\phi (ab^{-1})][b] = \phi ([ab^{-1}][b]) = [a] \wedge [b].$$
\end{proof}

In the following result, we shall use the description of elements of $\mathsf{R}(S)$
as joins of basic morphisms Lemma~\ref{lem:trump}.

\begin{lemma}\label{lem:hunt} Let $S$ be a $k$-monoid.
The congruence $\equiv$ is idempotent-pure on the inverse monoid $\mathsf{R}(S)$.
\end{lemma}
\begin{proof} Suppose that 
$$\left( \bigvee_{i=1}^{m} x_{i}y_{i}^{-1} \right) \equiv \left( \bigvee_{j=1}^{n} u_{j}u_{j}^{-1} \right).$$
Let $1 \leq i\leq m$.
Then $x_{i}y_{i}^{-1}$ is less than or equal to the left-hand side and non-zero.
By definition of the relation $\equiv$ it follows that $(x_{i}y_{i}^{-1})^{\downarrow} \cap (u_{j}u_{j}^{-1})^{\downarrow} \neq \varnothing$.
Now, the set of idempotents in an inverse semigroup is an order ideal.
Thus any element less than or equal to $u_{j}u_{j}^{-1}$ is an idempotent.
It follows that there is an idempotent $zz^{-1} \leq x_{i}y_{i}^{-1}$.
Thus by Lemma~\ref{lem:tom}, we have that $z = x_{i}p = y_{i}p$ for some $p \in S$.
By right cancellation, we have that $x_{i} = y_{i}$.
It follows that $x_{i}y_{i}^{-1}$ is an idempotent.
Thus 
$\bigvee_{i=1}^{m} x_{i}y_{i}^{-1}$
is an idempotent, as required.
\end{proof}

We have therefore proved the following.

\begin{proposition} Let $S$ be a $k$-monoid.
Then $\mathsf{R}(S)/\equiv$ is a distributive $\wedge$-monoid.
\end{proposition}

\noindent
{\bf Definition.} Let $S$ be a $k$-monoid.
Define
$$\mathsf{C}(S) = \mathsf{R}(S)/\equiv.$$

The rationale for defining this monoid now follows.
It is a consequence of Proposition~\ref{prop:delay} and what we proved above.

\begin{proposition}\label{prop:rory} Let $S$ be a $k$-monoid.
Then the group of units of $\mathsf{C}(S)$ is isomorphic to the group $\mathscr{G}(S)$.
\end{proposition}

We shall now prove that $\mathsf{C}(S)$ is actually a Boolean inverse monoid.
By Proposition~\ref{prop:cheese} and Proposition~\ref{prop:bojo}, there is a monoid homomorphism
$\chi \colon \mathsf{R}(S) \rightarrow \mathcal{I}^{\scriptsize cl}(S^{\infty})$
to the inverse monoid of partial homeomorphisms between the clopen subsets of $S^{\infty}$.
By \cite[Proposition~4.4.8]{Wehrung}, this is a fundamental inverse semigroup --- the significance of this property will be explained later.
Our goal now is to calculate the kernel of the homomorphism $\chi$.
To do this, we need some preparation.
We use the fact that $k$-tilings have corners of every possible size (Remark~\ref{rem:ribena}).
The first result is preparatory.

\begin{lemma}\label{lem:glue} Let $S$ be a $k$-monoid.
Let $X = \{x_{1}, \ldots, x_{m}\}$ and $Y = \{y_{1}, \ldots, y_{n}\}$
Then in $\mathsf{R}(S)$ we have that
$\left( \bigvee_{i=1}^{m} x_{i}x_{i}^{-1}   \right)
\equiv
\left( \bigvee_{j=1}^{n} y_{j}y_{j}^{-1}   \right)$
if and only if
$XS^{\infty} = YS^{\infty}$.
\end{lemma}
\begin{proof}
Suppose first that 
$\left( \bigvee_{i=1}^{m} x_{i}x_{i}^{-1}   \right)
\equiv
\left( \bigvee_{j=1}^{n} y_{j}y_{j}^{-1}   \right)$.
We prove that $XS^{\infty} = YS^{\infty}$.
Let $x_{i}w \in XS^{\infty}$ where $x_{i} \in X$ and $w \in S^{\infty}$.
Choose a corner $x$ of $w$ (so that $w = xw'$) such that $x_{i}x$ is bigger than any element of $Y$.
Now, $0 < x_{i}x(x_{i}x)^{-1} \leq \left( \bigvee_{i=1}^{m} x_{i}x_{i}^{-1}   \right)$.
Thus for some $j$, the set
$(x_{i}x(x_{i}x)^{-1})^{\downarrow} \cap (y_{j}y_{j}^{-1})^{\downarrow}$
is non-empty.
Let $zz^{-1} \leq x_{i}x(x_{i}x)^{-1}, y_{j}y_{j}^{-1}$.
Then $z = x_{i}wp = y_{j}q$ for some $p,q \in S$.
But by assumption, $d(x_{i}x) > d (y_{j})$.
Thus by Lemma~\ref{lem:levi}, there is $t \in S$ such that $x_{i}x = y_{j}t$.
It follows that $x_{i}w = y_{i}tw'$ and so $XS^{\infty} \subseteq YS^{\infty}$.
The reverse inclusion follows by symmetry.

Suppose that $XS^{\infty} = YS^{\infty}$.
We prove that
$\left( \bigvee_{i=1}^{m} x_{i}x_{i}^{-1}   \right)
\equiv
\left( \bigvee_{j=1}^{n} y_{j}y_{j}^{-1}   \right)$.
Let $uu^{-1} \leq x_{i}x_{i}^{-1}$.
Then $u = x_{i}p$ for some $p \in S$.
By assumption, $x_{i}pS^{\omega} \subseteq YS^{\infty}$.
Let $w \in S^{\infty}$.
Then $x_{i}pw = y_{j}w_{1}$ for some $w_{1} \in S^{\infty}$ and some $y_{j} \in Y$.
By Lemma~\ref{lem:size-matters}, we can choose a corner $x$ of $w$, so that $w = xw'$, such that $x_{i}px$ is bigger than any element of $Y$.
Now recall that $x_{i}p = u$.
Then $uxw' = y_{j}w_{1}$.
By Lemma~\ref{lem:levi}, there is $t \in S$ such that $z = ux  = y_{j}t$.
It follows that $zz^{-1} \leq uu^{-1}, y_{j}y_{j}^{-1}$, as required.
The result now follows by symmetry.
\end{proof}

Our key result is the following.

\begin{proposition}\label{prop:pepsi} Let $S$ be a $k$-monoid with an aperiodic $k$-tiling.
Let $X = \{x_{1}, \ldots, x_{m}\}$, $Y = \{y_{1}, \ldots, y_{m}\}$,
$U = \{u_{1}, \ldots, u_{n}\}$ and $V = \{v_{1}, \ldots, v_{n}\}$.
Put $\alpha = \bigvee_{i=1}^{m} x_{i}y_{i}^{-1}$ and $\beta = \bigvee_{j=1}^{n} u_{j}v_{j}^{-1}$.
Then in $\mathsf{R}(S)$ we have that
$\alpha \equiv \beta$ 
if and only if
$\chi (\alpha) = \chi (\beta)$.
\end{proposition}
\begin{proof} Suppose that $\alpha \equiv \beta$.
By Lemma~\ref{lem:glue}, we have that $YS^{\infty} = VS^{\infty}$ and $XS^{\infty} = US^{\infty}$.
Let $y_{i}w \in YS^{\infty}$.
By Lemma~\ref{lem:size-matters},
we can choose a corner $x$ of $w$ such that $x_{i}x$ is bigger than every element in $U$
and $y_{i}x$ is bigger than every element in $V$.
Let $w = x w'$.
Now $x_{i}x(y_{i}x)^{-1} \leq x_{i}y_{i}^{-1}$.
Thus we can find $ab^{-1} \leq x_{i}x(y_{i}x)^{-1}, u_{j}v_{j}^{-1}$ for some $j$.
It follows that $a = x_{i}xp = u_{j}q$ and $b = y_{i}xp = v_{j}q$ for some $p,q \in S$.
Thus by Lemma~\ref{lem:levi}, we can write $x_{i}x = u_{j}t$ and $y_{i}x = v_{j}s$.
Observe that $tp = q$ and that $sp = q$ and so $s = t$.
It is now easy to show that $\chi (\alpha)(y_{i}w) = \chi(\beta)(y_{i}w)$, as required. 

Now suppose that $\chi (\alpha) = \chi (\beta)$.
We prove that $\alpha \equiv \beta$.
Let $ab^{-1} \leq x_{i}y_{i}^{-1}$.
Then $(a,b) = (x_{i},y_{i})p$ for some $p \in S$.
Let $w \in S^{\infty}$ be arbitrary.
By Lemma~\ref{lem:size-matters}, we can choose a corner $x$ of $w$ such that $y_{i}px$ is bigger than every element of $V$
and $x_{i}px$ is bigger than every element of $U$.
Let $w = x w'$.
Then $\chi(\alpha)(y_{i}pxw') = x_{i}pxw'$.
By assumption, $\chi (\beta)(y_{i}pxw')$ is also defined and equals $x_{i}pxw'$.
We must have that $y_{i}pxw' = v_{j}w_{1}$ for some $j$.
It follows that $\chi(\beta)(y_{i}pxw') = u_{j}w_{1}$.
Thus $x_{i}pxw' = u_{j}w_{1}$.
By Lemma~\ref{lem:levi}, there are elements $s,t \in S$ such that
$y_{i}px = v_{j}s$ and $x_{i}px = u_{j}t$.
Now, we have that $y_{i}px\hat{w} = v_{j}s\hat{w}$ for all $\hat{w} \in S^{\infty}$.
But $\chi (\alpha) = \chi (\beta)$ and so
$x_{i}px\hat{w} = u_{j}s\hat{w}$ for all $\hat{w} \in S^{\infty}$.
By Corollary~\ref{cor:good}, it follows that $x_{i}px = u_{j}s$.
Put $c = ax = u_{j}s$ and $d = bx = v_{j}s$ 
where $a = x_{i}p$ and $b = y_{i}p$.
Then $cd^{-1} \leq ab^{-1}, u_{j}v_{j}^{-1}$.
The result now follows by symmetry.
\end{proof}

The idempotents of $\mathcal{I}^{\scriptsize cl}(S^{\infty})$ are the identity functions on the clopen subsets of $S^{\infty}$.
Each clopen subset is compact and therefore a finite union of clopen subsets of the form $xS^{\infty}$ where $x \in S$.
Thus each clopen subset is of the form $\{x_{1}, \ldots, x_{n}\}S^{\infty}$.
It follows that each idempotent in $\mathcal{I}^{\scriptsize cl}(S^{\infty})$ is the image of an idempotent in $\mathsf{R}(S)$.
It is well-known, and easy to prove, that wide inverse subsemigroups of fundamental inverse semigroups are themselves fundamental.
Let $\{x_{1}, \ldots, x_{m}\}S^{\infty}$ and $\{y_{1}, \ldots, y_{n}\}S^{\infty}$ be any
two non-empty clopen subsets of $S^{\infty}$.
There is a bijective morphism between  $\{x_{1}, \ldots, x_{m}\}S^{\infty}$ and $\{y_{1}x_{1}, \ldots, y_{1}x_{m}\}S^{\infty}$ 
where $x_{i} \omega \mapsto y_{1}x_{i}\omega$.
Clearly, $\{y_{1}x_{1}, \ldots, y_{1}x_{m}\}S^{\infty} \subseteq \{y_{1}, \ldots, y_{n}\}S^{\infty}$.
It follows that $\mathsf{C}(S)$ is $0$-simple.
By Lemma~\ref{lem:borisy}, $\mathsf{C}(S)$ is a $\wedge$-monoid. 
We refer to the unique countable atomless Boolean algebra as the {\em Tarski algebra};
it is the dual of the Cantor space under classical Stone duality.
We have therefore proved the following.

\begin{theorem}\label{them:munn} Let $S$ be a $k$-monoid with an aperiodic $k$-tiling.
Then the inverse monoid $\mathsf{C}(S)$ is isomorphic to
a wide inverse submonoid of $\mathcal{I}^{\scriptsize cl}(S^{\infty})$.
In particular, $\mathsf{C}(S)$ is a countably infinite, $0$-simple fundamental Boolean inverse $\wedge$-monoid
whose semilattice of idempotents is the Tarski algebra.
\end{theorem}

It follows that the Boolean inverse monoid $\mathsf{C}(S)$ is of the type discussed in the paper \cite{Lawson2016}.

\section{The group $\mathscr{G}(S)$ as a topological full group}

In this section, we pull the different strands of the paper together using groupoids.
If $G$ is a groupoid (a small category in which every arrow is invertible) then its set of identities is denoted by $G_{o}$.
The groupoid in question is the one described in \cite[Section 2]{KP} but we shall describe it from a new perspective.
We shall begin with some general results before specializing to the case that interests us.
It is well-known that a group acting on a set gives rise to a groupoid.
There is an obvious generalization of this construction which yields a category when the group
is replaced by an arbitrary monoid.

Let $M$ be  an abelian monoid acting on the left on the set $X$.
We denote the action by $(a,x) \mapsto a \cdot x$.
Define $M \ltimes X$ to be the set of all triples in $X \times M \times X$ of the form $(x, a, a \cdot x)$.
Now define $\mathbf{d}(x_{2}, a, x_{1}) = (x_{1}, 1, x_{1})$ and $\mathbf{r}(x_{2}, a, x_{1}) = (x_{2}, 1, x_{2})$.
If $\mathbf{d}(x_{2},a,x_{1}) = \mathbf{r}(x_{3}, b, x_{4})$ then define the partial product $(x_{2},a,x_{1})(x_{3},b,x_{4}) = (x_{2}, ab,x_{4})$.
It is easy to check that $M \ltimes X$ is a category whose identity space consists of all elements of the form $(x,1,x)$
and so can be identified with $X$.
If $M$ is a cancellative monoid then $M \ltimes X$ is a cancellative category.
Put $\mathcal{C} = M \ltimes X$.

Suppose now that $M \subseteq G$ where $G$ is an abelian group.
We shall construct a groupoid $\mathcal{G}$ from $G$ and $X$ but without assuming that the action of $M$ on $X$ can be lifted to an action of $G$ on $X$.
The elements of $\mathcal{G}$ are those triples from $X \times G \times X$ of the form $(x,c,y)$ where $a \cdot x = b \cdot y$ and $c = ab^{-1}$
for some $a,b \in M$ and $x,y \in X$.
We prove that this is a groupoid we shall call $\mathcal{G}$.
Define $\mathbf{d}(x,c,y) = (y,1,y)$ and $\mathbf{r}(x,c,y) = (x,1,x)$, both well-defined elements.
Suppose that $\mathbf{d}(x,c,y) = \mathbf{r}(u,d,v)$.
Then $y = u$, 
$a_{1} \cdot x = a_{2} \cdot y$  where $c = a_{1}a_{2}^{-1}$,
$b_{1} \cdot u = b_{2} \cdot v$ where $d = b_{1}b_{2}^{-1}$.
We prove that $(x,cd,v)$ has the correct form.
We calculate
$$(a_{1}b_{1}) \cdot x = b_{1} \cdot (a_{1} \cdot x) = b_{1} \cdot (a_{2} \cdot y) = a_{2} \cdot (b_{1} \cdot y) = a_{2} \cdot (b_{1} \cdot u) = (a_{2}b_{2}) \cdot v,$$
where we have used commutativity throughout and, 
in addition, $a_{1}b_{1}(a_{2}b_{2})^{-1} = cd$.
Thus $\mathcal{G}$ is certainly a category.
Finally, if we define $(x,c,y) = (y,c^{-1},x)$ then we get a groupoid.

The space of identitities of $\mathcal{G}$ again consists of all triples of the form $(x,1,x)$.
Thus $\mathcal{C} \subseteq \mathcal{G}$ is a wide subcategory\footnote{This simply means that the two categories have the same set of identities.}
and we have therefore embedded a cancellative category into a groupoid.
In fact, we can say more.
Let $(x,c,y)$ be an arbitrary element of $\mathcal{G}$
where $a \cdot x = b \cdot y$ and $c = ab^{-1}$.
Then $(x, a, a \cdot x), (y, b,b \cdot y) \in \mathcal{C}$.
We have that 
$$(x, a, a \cdot x)(y, b, b \cdot y)^{-1} = (x a, a \cdot x)(b \cdot y, b^{-1}, y) = (x, ab^{-1}, y) = (x,c,y).$$
Thus in fact we have that $\mathcal{G} = \mathcal{C}\mathcal{C}^{-1}$.

There is an action of the monoid $\mathbb{N}^{k}$ on the set $S^{\infty}$ given by
$(\mathbf{m},x) \mapsto \sigma^{\mathbf{m}}(x)$.
That this really is an action is apparent from \cite[Definitions~2.1]{KP}.
The monoid $\mathbb{N}^{k}$ is the positive cone of the lattice-ordered abelian group $\mathbb{Z}^{k}$.
Restricting the construction of our groupoid above to this special case yields a groupoid we shall denote by $\mathcal{G}(S)$ which is exactly the one defined in \cite[Definition~2.7]{KP}.
We describe it explicitly.
The elements of $\mathcal{G}(S)$ are those triples $(w_{2}, \mathbf{n}, w_{1}) \in S^{\infty} \times \mathbb{Z}^{k} \times S^{\infty}$ where 
$\sigma^{\mathbf{l}}(w_{2}) = \sigma^{\mathbf{m}}(w_{1})$ and $\mathbf{n} = \mathbf{l} - \mathbf{m}$.
In addition, 
$\mathbf{d}(w_{2}, \mathbf{n}, w_{1}) = (w_{1}, \mathbf{0}, w_{1})$ 
and 
$\mathbf{r}(w_{2}, \mathbf{n}, w_{1}) = (w_{2}, \mathbf{0}, w_{2})$
and 
$$(w_{2}, \mathbf{n}, w_{1})^{-1} = (w_{1}, -\mathbf{n}, w_{2}).$$

We now endow $M$ with the discrete topology, assume that $X$ has a topology (ultimately Boolean), and that the action of $M$ on $X$ is by local homeomorphisms.
If $U \subseteq X$ is an open set and $a \in M$ then $aU$ is an open set since local homeomorphisms are open maps.
Define $\mathcal{O}(a,U) = \{ (x, a, a \cdot x) \colon x \in U\}$.
Clearly, $\mathcal{O}(a,U) \subseteq \mathcal{C}$.
Using these sets as basis elements, we can endow $\mathcal{C}$ with a topology.

We now return to our action $\mathbb{N}^{k} \times S^{\infty} \rightarrow S^{\infty}$ which is by local homeomorphisms 
(from the top of page 8 of \cite{KP}).
Clearly, $\mathcal{O}(d(u),uS^{\infty})$ is well-defined for any $u \in S$.
Calculating $\mathcal{O}(d(x),xS^{\infty})\mathcal{O}(d(y), yS^{\infty})^{-1} \subseteq \mathcal{G}(S)$,
we find it consists precisely of elements of the form $(xw, d(x) - d(y), yw)$ where $w \in S^{\infty}$.
Define
$$Z(x,y) = \{ (xw, d(x) - d(y), yw) \colon w \in S^{\infty}\}$$
where $x,y \in S$ and 
endow the groupoid $\mathcal{G}(S)$ with the topology having as basis elements precisely subsets of this form. 
We refer the reader to \cite{Resende2006} for the theory of \'etale topological groupoids.
An \'etale groupoid $G$ is said to be {\em Boolean} if its space of identities, $G_{o}$, is a Boolean space.
The following is simply a restatement of \cite[Proposition~2.8]{KP}.

\begin{proposition}\label{prop:groupoid} Let $S$ be a $k$-monoid.
Then  $\mathcal{G}(S)$ is a second-countable, Hausdorff Boolean groupoid.
\end{proposition}

\begin{remark}{\em Our calculations above suggest studying those \'etale groupoids that are obtained as categories of fractions
(in the sense of \cite{GZ}) from topological cancellative categories}.
\end{remark}

At this point, we shall make use of the non-commutative Stone duality developed in \cite{Lawson2010, Lawson2012, LL}.
A subset $A \subseteq G$ of a groupoid is said to be a {\em local bisection} if $g,h \in A$ and $g^{-1}g = h^{-1}h$ (respectively $gg^{-1} = hh^{-1}$) implies that $g = h$.
A local  bisection $A$ is said to be a {\em bisection} if $A^{-1}A = G_{o} = AA^{-1}$. 
Let $G$ be a Boolean groupoid.
Denote by $\mathsf{KB}(G)$ the set of all compact-open local bisections.
This is a Boolean inverse monoid.
Let $A$ be a compact-open local bisection.
Then we may define a partial homeomorphism $\alpha_{A} \colon A^{-1}A \rightarrow AA^{-1}$ by $e \mapsto f$
where $e = a^{-1}a$ and $f = aa^{-1}$ where $a \in A$.
This defines a homomorphism $\alpha \colon \mathsf{KB}(G) \rightarrow \mathcal{I}^{\scriptsize cl}(G_{o})$ to the Boolean inverse monoid
of all partial homeomorphisms between clopen subsets.

Let $G$ be a groupoid.
The {\em isotropy groupoid} of $G$ is the set of all elements $g \in G$ such that $g^{-1}g = gg^{-1}$.
Clearly, the space of identities is contained in the isotropy groupoid.
An \'etale topological groupoid is said to be {\em effective} if the interior of the isotropy groupoid is exactly the space of identities of the groupoid. 
An \'etale topological groupoid is said to be {\em topologically principal} if the set of identities with trivial isotropy is dense.\footnote{The term {\em essentially free} is used in \cite{KP}}
The following is just a version of \cite[Corollary 3.3]{Renault2008}.

\begin{lemma}\label{lem:moth} Let $G$ be a Boolean groupoid.
Then 
$\alpha \colon \mathsf{KB}(G) \rightarrow \mathcal{I}^{\scriptsize cl}(G_{o})$
is injective if and only if $G$ is effective
\end{lemma}
\begin{proof} Suppose that $\alpha$ is injective.
We prove that $G$ is effective.
Let $A$ be an open subset of isotropy groupoid.
Without loss of generality, we may assume that $A$ is a compact-open local bisection.
By assumption, $\alpha (A)$ fixes every identity in $A^{-1}A$ and so $A$ and $A^{-1}A$ have the same image under $\alpha$.
It follows that $A = A^{-1}A$ and so every element of $A$ is an identity.
We have therefore proved that the interior of the isotropy groupoid is the space of identities.
Now suppose that $G$ is effective.
Let $A$ and $B$ be compact-open local bisections such that $\alpha (A) = \beta (B)$.
Then $\alpha (A^{-1}B)$ is the identity where it is defined.
Thus $A^{-1}B$ is a subset of the isotropy groupoid.
But $A^{-1}B$ is open and so must consist entirely of identities.
It now follows that $A = b$, as required.  
\end{proof}

The following is \cite[Proposition~4.5]{KP}.

\begin{lemma}\label{lem:tp} Let $S$ be a $k$-monoid.
Then the groupoid  $\mathcal{G}(S)$ is topologically principal if and only if there is an aperiodic $k$-tiling.
\end{lemma}

\begin{lemma}\label{lem:tpy} Let $S$ be a $k$-monoid.
If $S$ is effective --- 
that is, if for each pair of distinct elements $x$ and $y$ there is an
element $c$ such that $xc$ and $yc$ are incomparable 
---
then $\mathcal{G}(S)$ is effective. 
\end{lemma}
\begin{proof} Suppose that $Z(x,y)$ is a subset of the isotropy groupoid
and that $Z(x,y)$ contains a non-identity element $(xw', d(x) - d(y), yw')$ for some $w' \in S^{\infty}$.
In partiuclar, $d(x) - d(y) \neq \mathbf{0}$ and so $x \neq y$.
But, by assumption, there then exists $c \in S$ such that $xc$ and $yc$ are incomparable.
It follows that for all $w \in S^{\infty}$, we must have that $x(cw) \neq y(cw)$
(because if $xcw = ycw$ for some $w \in S^{\infty}$ then $xc$ and $yc$ would be comparable).
This contradicts the assumption that $Z(x,y)$ is a subset of the isotropy groupoid.
It follows that $Z(x,y)$ can only contain identities and so we have proved that $S$ is effective.
\end{proof}

Recall that a {\em Baire space} is a topological space in which the intersection of every countable set of dense open subsets is dense.
Every locally compact Hausdorff space is a Baire space.
Thus Boolean spaces are Baire spaces.
The \'etale groupoid of a $k$-monoid is a Hausdorff, second countable Boolean groupoid by Proposition~\ref{prop:groupoid}.
The following is therefore a consequence of a result proved by Renault \cite[Proposition~3.6]{Renault2008}.

\begin{proposition}\label{prop:nice-result} Let $S$ be a $k$-monoid.
Then the following are equivalent:
\begin{enumerate}

\item The groupoid $\mathcal{G}(S)$ is effective. 

\item The groupoid $\mathcal{G}(S)$ is topologically principle.

\end{enumerate}
\end{proposition}

\begin{remark}
{\em We shall assume that the groupoid is effective in what follows}.
\end{remark}

We now describe the compact-open local bisections.

\begin{lemma}\label{lem:zooo} Each compact-open local bisection of $\mathscr{G}(S)$
is determined by an element of $\mathsf{R}(S)$.
\end{lemma}
\begin{proof} Observe that the sets $Z(x,y)$ are local bisections 
because if $yw = yw'$ then $w = w'$ (and dually).
We calculate the partial bijection of $\mathcal{G}(S)_{o}$ determined by $Z(x,y)$ under the map $\alpha$.
Its domain is $yS^{\infty}$ and its range is $xS^{\infty}$ and the effect of the partial bijection is $yw \mapsto xw$.
It follows that the partial bijection is $xy^{-1}$.
The fact that the groupoid is effective means if two such sets induce the same partial bijection then they are equal.
It follows that there is a bijection $Z(x,y) \longleftrightarrow xy^{-1}$.
Let $A$ be a compact-open local bisection of $\mathcal{G}(S)$.
Then $A$ is a union of elements of the form $Z(x,y)$ and a finite union since it is compact.
We therefore have that $A = \bigcup_{i=1}^{n} Z(x_{i},y_{i})$.
\end{proof}

Because the groupoid $\mathcal{G}(S)$ is effective, we may identify compact-open local bisections
with certain partial homeomorphisms between clopen subsets of $S^{\infty}$.
The set of these homeomorphisms is precisely the image under the monoid homomorphism 
$\chi \colon \mathsf{R}(S) \rightarrow \mathcal{I}^{\scriptsize cl}(S^{\infty})$.   
It follows by Lemma~\ref{lem:zooo} that $\chi$ maps $\mathsf{R}(S)$ onto  $\mathsf{KB}(\mathcal{G}(S))$. 
Thus $\mathsf{KB}(\mathcal{G}(S))$ is isomorphic to $\mathsf{R}(S)/\equiv$.

Given an \'etale topological groupoid $G$, its {\em topological full group} is the groupoid of
all compact-open bisections.
We have therefore proved the following.

\begin{theorem}\label{them:main} Let $S$ be a $k$-monoid with an aperiodic $k$-tiling such that $S^{\infty}$ is the Cantor space.
\begin{enumerate}

\item The Boolean inverse monoid $\mathsf{C}(S)$ is isomorphic to the Boolean inverse monoid $\mathsf{KB}(\mathsf{G}(S))$.

\item The group $\mathscr{G}(S)$ is therefore the topological full group of the Boolean groupoid $\mathcal{G}(S)$.

\end{enumerate}
\end{theorem}

By Theorem~\ref{them:munn}, the Boolean inverse monoid $\mathsf{C}(S)$ is $0$-simple and fundamental.
It follows by \cite[Theorem~4.16]{Lawson2016}, that the groupoid $\mathcal{G}(S)$ is purely infinite and minimal.
Thus by Theorem~\ref{them:main} and \cite[Theorem~4.16]{Matui2015}, we have proved the following.

\begin{theorem}\label{them:commutator} Let $S$ be a $k$-monoid with an aperiodic $k$-tiling such that $S^{\infty}$ is the Cantor space.
Then the commutator subgroup of $\mathscr{G}(S)$ is simple.
\end{theorem}

\section{Rigid $k$-monoids}

Let $S$ be a $k$-monoid with alphabets $X_{1}, \ldots, X_{k}$ and let $s \in S$.
Then $s = x_{1} \ldots x_{k}$, where $x_{i} \in X_{i}^{\ast}$, uniquely.
We call $x_{i}$ the {\em $i$-component} of $s$.
However, there is another way of obtaining elements of $X_{i}^{\ast}$ from $s$.
If $s = s_{i}s_{i}'$, where $s_{i} \in X_{i}^{\ast}$ and $s_{i}' \in X_{1}^{\ast} \ldots \hat{X_{i}^{\ast}} \ldots X_{k}^{\ast}$
(thus, the $i$th co-ordinate of $d(s_{i}')$ is zero),
then we call $s_{i}$ the {\em $i$-projection} of $s$.
Thus we have a $k$-tuple $(s_{1}, \ldots, s_{k})$ which we call a {\em projection}.
If $s \in X_{i}^{\ast}$ we say that $s$ is {\em homogeneous (of type $i$)}.

Simple examples of $k$-monoids can be constructed from direct products of $k$ free monoids by Example~\ref{ex:bbc},
but direct products of free monoids are rather restricted.
It is useful to have a broader class of $k$-monoids to work with.
We define this class now.
Let $S$ be a $k$-monoid with alphabets $X_{1}, \ldots, X_{k}$.
We say that $S$ is {\em right rigid} if the following condition holds where $i \neq j$:
given $x \in X_{i}$ and $y \in X_{j}$ there are unique elements $x' \in X_{i}$ and $y' \in X_{j}$ such that $xy' = yx'$.
This condition implies, in particular, that letters belonging to different alphabets are always comparable.
This condition was called the {\em unique pullback property} in \cite{DY2016},
whereas in \cite{DY2017} it was called the {\em little pullback property} which is the term used in \cite[Section 20]{Exel}
where it seems to have been first introduced.
In terms of `squares', it says that the bottom lefthand-side uniquely determines the top righthand-side.
We say that $S$ is {\em left rigid} if the following condition holds: 
given $x' \in X_{i}$ and $y' \in X_{j}$ there are unique elements $x \in X_{i}$ and $y \in X_{j}$
such that $xy' = yx'$.
This condition was first stated in \cite{DY2016} where it was called the {\em unique pushout property}.
In terms of `squares', it says that the top righthand-side uniquely determines the bottom lefthand-side.
We shall construct examples below to show that these two conditions are independent.
A {\em rigid} monoid is one that is both left rigid and right rigid.
Free monoids are vacuously rigid.
The direct product of rigid monoids is rigid.
Thus finite direct products of free monoids are rigid.

\begin{example}{\em We now describe all the $2$-monoids
in which there are two types of alphabets: the {\em $e$-alphabet} $\{e_{1},e_{2}\}$ and the {\em $f$-alphabet} $\{f_{1},f_{2}\}$.
According to \cite{Power2007}, we can describe all such $2$-monoids by means of a permutation $\theta$ of the set
$\{1,2\} \times \{1,2\}$. Given such a permutation, we write down all the relations of the form $e_{i}f_{j} = f_{j'}e_{i'}$
where $\theta (i,j) = (i',j')$. 
Observe that whenever $\theta (i,j) = (i,j)$, we get the relation $e_{i}f_{j} = f_{j}e_{i}$.
The monoid that results is denoted by $A_{2}^{\ast} \times_{\theta} A_{2}^{\ast}$.
By \cite[Proposition 3.1]{Power2007}, there are 9 non-isomorphic 2-monoids of this type.
Here is a list of all 9 monoids of this type.
\begin{enumerate}

\item $\theta$ is the identity permutation. The relations are $e_{1}f_{1} = f_{1}e_{1}$, $e_{1}f_{2} = f_{2}e_{1}$, $e_{2}f_{1} = f_{1}e_{2}$, $e_{2}f_{2} = f_{2}e_{2}$.
This is just the presentation of the monoid $A_{2}^{\ast} \times A_{2}^{\ast}$.
Each of the four relations can be written as a square;
observe that we orientate the $e$-alphabet from left-to-right and the $f$-alphabet from bottom-to-top:
$$
A: 
\xymatrix{
 \ar@{-}[r]_{e_{1}} & \\
\ar@{-}[u]_{f_{1}} \ar@{-}[r]^{e_{1}} &   \ar@{-}[u]^{f_{1}} }
\quad
B: 
\xymatrix{
 \ar@{-}[d]^{f_{2}} \ar@{-}[r]_{e_{1}} &  \ar@{-}[d]_{f_{2}} \\
 \ar@{-}[r]^{e_{1}} & }
\quad
C: 
\xymatrix{
 \ar@{-}[d]^{f_{1}} \ar@{-}[r]_{e_{2}} &  \ar@{-}[d]_{f_{1}} \\
 \ar@{-}[r]^{e_{2}} &  }
\quad
D: 
\xymatrix{
 \ar@{-}[d]^{f_{2}} \ar@{-}[r]_{e_{2}} &  \ar@{-}[d]_{f_{2}} \\
 \ar@{-}[r]^{e_{2}} &  }
$$

\item $\theta$ is the transposition $(1,1) \leftrightarrow (1,2)$.
The relations are $e_{1}f_{1} = f_{2}e_{1}$ and $e_{1}f_{2} = f_{1}e_{1}$
and then $e_{2}f_{1} = f_{1}e_{2}$, $e_{2}f_{2} = f_{2}e_{2}$.
We write each of the four relations as a square:
$$
A: 
\xymatrix{
 \ar@{-}[d]^{f_{2}} \ar@{-}[r]_{e_{1}} & \ar@{-}[d]_{f_{1}} \\
 \ar@{-}[r]^{e_{1}} &  }
\quad
B: 
\xymatrix{
 \ar@{-}[d]^{f_{1}} \ar@{-}[r]_{e_{1}} &  \ar@{-}[d]_{f_{2}} \\
 \ar@{-}[r]^{e_{1}} & }
\quad
C: 
\xymatrix{
 \ar@{-}[d]^{f_{1}} \ar@{-}[r]_{e_{2}} &  \ar@{-}[d]_{f_{1}} \\
 \ar@{-}[r]^{e_{2}} &  }
\quad
D: 
\xymatrix{
 \ar@{-}[d]^{f_{2}} \ar@{-}[r]_{e_{2}} &  \ar@{-}[d]_{f_{2}} \\
 \ar@{-}[r]^{e_{2}} & }
$$

\item $\theta$ is the transposition $(1,1) \leftrightarrow (2,2)$.
The relations are $e_{1}f_{1} = f_{2}e_{2}$ and $e_{2}f_{2} = f_{1}e_{1}$ and $e_{1}f_{2} = f_{2}e_{1}$, $e_{2}f_{1} = f_{1}e_{2}$.
We write each of the four relations as a square.
$$
A: 
\xymatrix{
 \ar@{-}[d]^{f_{2}} \ar@{-}[r]_{e_{2}} &  \ar@{-}[d]_{f_{1}} \\
 \ar@{-}[r]^{e_{1}} &  }
\quad
B: 
\xymatrix{
 \ar@{-}[d]^{f_{1}} \ar@{-}[r]_{e_{1}} &  \ar@{-}[d]_{f_{2}} \\
 \ar@{-}[r]^{e_{2}} &  }
\quad
C: 
\xymatrix{
 \ar@{-}[d]^{f_{2}} \ar@{-}[r]_{e_{1}} &  \ar@{-}[d]_{f_{2}} \\
 \ar@{-}[r]^{e_{1}} &  }
\quad
D: 
\xymatrix{
 \ar@{-}[d]^{f_{1}} \ar@{-}[r]_{e_{2}} &  \ar@{-}[d]_{f_{1}} \\
 \ar@{-}[r]^{e_{2}} &  }
$$

\item $\theta$ is the $3$-cycle $(1,1) \rightarrow (1,2) \rightarrow (2,2) \rightarrow (1,1)$.
The relations are  $e_{1}f_{1} = f_{2}e_{1}$, $e_{1}f_{2} = f_{2}e_{2}$, $e_{2}f_{2} = f_{1}e_{1}$ and $e_{2}f_{1} = f_{1}e_{2}$.
We write each of the four relations as a square:
$$
A: 
\xymatrix{
 \ar@{-}[d]^{f_{2}} \ar@{-}[r]_{e_{1}} &  \ar@{-}[d]_{f_{1}} \\
 \ar@{-}[r]^{e_{1}} &  }
\quad
B: 
\xymatrix{
 \ar@{-}[d]^{f_{2}} \ar@{-}[r]_{e_{2}} &  \ar@{-}[d]_{f_{2}} \\
 \ar@{-}[r]^{e_{1}} &  }
\quad
C: 
\xymatrix{
 \ar@{-}[d]^{f_{1}} \ar@{-}[r]_{e_{1}} &  \ar@{-}[d]_{f_{2}} \\
 \ar@{-}[r]^{e_{2}} &  }
\quad
D: 
\xymatrix{
 \ar@{-}[d]^{f_{1}} \ar@{-}[r]_{e_{2}} &  \ar@{-}[d]_{f_{1}} \\
 \ar@{-}[r]^{e_{2}} &  }
$$

\item $\theta$ is the $3$-cycle $(1,1) \rightarrow (2,2) \rightarrow (1,2) \rightarrow (1,1)$.
The relations are $e_{1}f_{1} = f_{2}e_{2}$, $e_{2}f_{2} = f_{2}e_{1}$, $e_{1}f_{2} = f_{1}e_{1}$ and  $e_{2}f_{1} = f_{1}e_{2}$.
$$
A: 
\xymatrix{
 \ar@{-}[d]^{f_{2}} \ar@{-}[r]_{e_{2}} &  \ar@{-}[d]_{f_{1}} \\
 \ar@{-}[r]^{e_{1}} &  }
\quad
B: 
\xymatrix{
 \ar@{-}[d]^{f_{2}} \ar@{-}[r]_{e_{1}} &  \ar@{-}[d]_{f_{2}} \\
 \ar@{-}[r]^{e_{2}} &  }
\quad
C: 
\xymatrix{
 \ar@{-}[d]^{f_{1}} \ar@{-}[r]_{e_{1}} &  \ar@{-}[d]_{f_{2}} \\
 \ar@{-}[r]^{e_{1}} &  }
\quad
D: 
\xymatrix{
 \ar@{-}[d]^{f_{1}} \ar@{-}[r]_{e_{2}} &  \ar@{-}[d]_{f_{1}} \\
 \ar@{-}[r]^{e_{2}} &  }
$$

\item $\theta$ is the following pair of transpositions $(1,1) \leftrightarrow (1,2)$ and $(2,1) \leftrightarrow (2,2)$.
The relations are $e_{1}f_{1} = f_{2}e_{1}$, $e_{1}f_{2} = f_{1}e_{1}$, $e_{2}f_{1} = f_{2}e_{2}$ and $e_{2}f_{2} = f_{1}e_{2}$.
We write each of the four relations as a square:
$$
A: 
\xymatrix{
 \ar@{-}[d]^{f_{2}} \ar@{-}[r]_{e_{1}} &  \ar@{-}[d]_{f_{1}} \\
 \ar@{-}[r]^{e_{1}} &  }
\quad
B: 
\xymatrix{
 \ar@{-}[d]^{f_{1}} \ar@{-}[r]_{e_{1}} &  \ar@{-}[d]_{f_{2}} \\
 \ar@{-}[r]^{e_{1}} &  }
\quad
C: 
\xymatrix{
 \ar@{-}[d]^{f_{2}} \ar@{-}[r]_{e_{2}} & \ar@{-}[d]_{f_{1}} \\
 \ar@{-}[r]^{e_{2}} &  }
\quad
D: 
\xymatrix{
 \ar@{-}[d]^{f_{1}} \ar@{-}[r]_{e_{2}} &  \ar@{-}[d]_{f_{2}} \\
 \ar@{-}[r]^{e_{2}} &  }
$$

\item $\theta$ is the following pair of transpositions $(1,1) \leftrightarrow (2,2)$ and $(1,2) \leftrightarrow (2,1)$.
The relations are $e_{1}f_{1} = f_{2}e_{2}$ and $e_{2}f_{2} = f_{1}e_{1}$, $e_{1}f_{2} = f_{1}e_{2}$ and $e_{2}f_{1} = f_{2}e_{1}$.
We write each of the four relations as a square:
$$
A: 
\xymatrix{
 \ar@{-}[d]^{f_{2}} \ar@{-}[r]_{e_{2}} &  \ar@{-}[d]_{f_{1}} \\
 \ar@{-}[r]^{e_{1}} &  }
\quad
B: 
\xymatrix{
 \ar@{-}[d]^{f_{1}} \ar@{-}[r]_{e_{1}} &  \ar@{-}[d]_{f_{2}} \\
 \ar@{-}[r]^{e_{2}} &  }
\quad
C: 
\xymatrix{
 \ar@{-}[d]^{f_{1}} \ar@{-}[r]_{e_{2}} &  \ar@{-}[d]_{f_{2}} \\
 \ar@{-}[r]^{e_{1}} &  }
\quad
D: 
\xymatrix{
 \ar@{-}[d]^{f_{2}} \ar@{-}[r]_{e_{1}} &  \ar@{-}[d]_{f_{1}} \\
 \ar@{-}[r]^{e_{2}} &  }
$$

\item $\theta$ is the following $4$-cycle $(1,1) \rightarrow (1,2) \rightarrow (2,2) \rightarrow (2,1) \rightarrow (1,1)$.
The relations are $e_{1}f_{1} = f_{2}e_{1}$, $e_{1}f_{2} = f_{2}e_{2}$, $e_{2}f_{2} = f_{1}e_{2}$, $e_{2}f_{1} = f_{1}e_{1}$.
We write each of the four relations as a square:
$$
A: 
\xymatrix{
 \ar@{-}[d]^{f_{2}} \ar@{-}[r]_{e_{1}} &  \ar@{-}[d]_{f_{1}} \\
 \ar@{-}[r]^{e_{1}} &  }
\quad
B: 
\xymatrix{
 \ar@{-}[d]^{f_{2}} \ar@{-}[r]_{e_{2}} &  \ar@{-}[d]_{f_{2}} \\
 \ar@{-}[r]^{e_{1}} &  }
\quad
C: 
\xymatrix{
 \ar@{-}[d]^{f_{1}} \ar@{-}[r]_{e_{2}} &  \ar@{-}[d]_{f_{2}} \\
 \ar@{-}[r]^{e_{2}} &  }
\quad
D: 
\xymatrix{
 \ar@{-}[d]^{f_{1}} \ar@{-}[r]_{e_{1}} &  \ar@{-}[d]_{f_{1}} \\
 \ar@{-}[r]^{e_{2}} &  }
$$

\item $\theta$ is the following $4$-cycle $(1,1) \rightarrow (1,2) \rightarrow (2,1) \rightarrow (2,2) \rightarrow (1,1)$.
The relations are $e_{1}f_{1} = f_{2}e_{1}$, $e_{1}f_{2} = f_{1}e_{2}$, $e_{2}f_{1} = f_{2}e_{2}$, $e_{2}f_{2} = f_{1}e_{1}$.
We write each of the four relations as a square:
$$
A: 
\xymatrix{
 \ar@{-}[d]^{f_{2}} \ar@{-}[r]_{e_{1}} &  \ar@{-}[d]_{f_{1}} \\
 \ar@{-}[r]^{e_{1}} &  }
\quad
B: 
\xymatrix{
 \ar@{-}[d]^{f_{1}} \ar@{-}[r]_{e_{2}} &  \ar@{-}[d]_{f_{2}} \\
 \ar@{-}[r]^{e_{1}} &  }
\quad
C: 
\xymatrix{
 \ar@{-}[d]^{f_{2}} \ar@{-}[r]_{e_{2}} &  \ar@{-}[d]_{f_{1}} \\
 \ar@{-}[r]^{e_{2}} &  }
\quad
D: 
\xymatrix{
 \ar@{-}[d]^{f_{1}} \ar@{-}[r]_{e_{1}} &  \ar@{-}[d]_{f_{2}} \\
 \ar@{-}[r]^{e_{2}} &  }
$$

\end{enumerate}
In the above nine examples, 
(1), (2), (6), (7) and (9) are rigid,
(3) and (8) are neither left nor right rigid,
(4) is left rigid but not right rigid
and 
(5) is right rigid but not left rigid.}
\end{example}

\begin{lemma}\label{lem:bell} Let $S$ be a rigid $k$-monoid with alphabets $X_{1}, \ldots, X_{k}$.
Let $s_{1} \in X_{1}^{\ast}, \ldots, s_{k} \in X_{k}^{\ast}$.
Then there is a unique element $s \in S$ which has $(s_{1},\ldots,s_{k})$ as its projection.
\end{lemma} 
\begin{proof} This proof is best viewed graphically by filling in the region bounded
by the elements $s_{1},\ldots, s_{k}$ which the rigidity condition guarantees can be done uniquely.
\end{proof}

The following is \cite[Lemma~6.3]{DY2016}.

\begin{lemma}\label{lem:tsar} Let $S$ be a $k$-monoid
and let $x,y \in S$ such that $d(x) \wedge d(y) = 0$.
\begin{enumerate}

\item If $S$ is right rigid then there are unique elements $u$ and $v$ such that $xu = yv$ where $d(u) = d(y)$ and $d(v) = d(x)$.

\item If $S$ is left rigid then there are unique elements $u$ and $v$ such that $ux = vy$ where $d(u) = d(y)$ and $d(v) = d(x)$.

\end{enumerate}
\end{lemma}

We can refine the above lemma a little to obtain a `categorical-looking result'.
We prove the refinement of part (1) above since the proof of the refinement of part (2) follows by symmetry.
Recall that $d(x)_{i}$ is the $i$th co-ordinate of $d(x)$.

\begin{lemma}\label{lem:tsartwo}
Let $S$ be a right rigid $k$-monoid
and let $x,y \in S$ be such that $d(x) \wedge d(y) = \mathbf{0}$.
Let $u$ and $v$ be the unique elements such that $xu = yv$ where $d(u) = d(y)$ and $d(v) = d(x)$.
Let $p$ and $q$ be any elements such that $xp = yq$.
Then there is a unique element $t$ such that $p = ut$ and $q = vt$.
\end{lemma}
\begin{proof} Suppose that $xp = yq$.
Then $d(x) + d(p) = d(y) + d(q)$.
Clearly, $d(x), d(y) \leq d(x) + d(p)$.
Now $d(x) \wedge d(y) = \mathbf{0}$ implies that for each $1 \leq i \leq k$
at least one of $d(x)_{i}$ or $d(y)_{i}$ is zero.
It follows that $d(x) + d(y) \leq d(x) + d(p)$.
Similarly $d(x) + d(y) \leq d(y) + d(q)$.
It follows that $xp = xy's$ where $d(y') = d(y)$ and $d(s) = d(p) - d(y)$
and $yq = yx't$ where $d(x') = d(x)$ and $d(t) = d(q) - d(x)$.
It follows that $s$ and $t$ have the same length.
By the dual of Lemma~\ref{lem:levi}, we have that $s = t$.
It follows that $q = x't$ and $p = y't$.
Thus $xy' = yx'$ where $d(y') = d(y) = d(u)$ and $d(x') = d(x) = d(v)$.
By uniqueness, we have that $y' = u$ and $x' = v$.
It follows that $q = vt$ and $p = ut$.
\end{proof}

The terminology `singly aligned' seems to have been first used in \cite[Proposition~20.3]{Exel} where the following result was proved.

\begin{lemma}\label{lem:sa} Let $S$ be a right rigid $k$-monoid.
Then $S$ is singly aligned; that is,
if $x,y \in S$, where $x \neq y$, are such that 
$xS \cap yS \neq \varnothing$
then there is a unique element $z$ such that $xS \cap yS = zS$.
\end{lemma}

\begin{remark}{\em The property of being `singly aligned' has played an important role in semigroup theory;
see \cite{Lawson1999} and \cite{FK2008}.
The latter paper would refer to left cancellative monoids in which the intersection of any two principal rights ideals
is either empty or a principal right ideal as {\em right LCM} monoid.
In particular, left cancellative monoids which are singly aligned can be used to construct all $0$-bisimple inverse monoids. 
Given such a monoid $S$, we construct the inverse monoid $\mathsf{B}(S)$ that consists of all the isomorphisms $ab^{-1} \colon bS \rightarrow aS$.
That this is even a semigroup follows from the singly aligned property.
It is $0$-bisimple because any two non-zero principal right ideals are connected by an element of $\mathsf{B}(S)$. 
If $S$ is cancellative then the inverse monoid $\mathsf{B}(S)$ is $E^{\ast}$-unitary.
$C^{\ast}$-algebras of right LCM monoids in general are studied in \cite{BLS2017}.}
\end{remark}

The following example motivates the idea behind Proposition~\ref{prop:chocolate}.

\begin{example}{\em Let $A = \{a_{1},a_{2}\}$, $B = \{b_{1},b_{2}\}$, $C = \{c_{1},c_{2}\}$ and $D = \{d_{1},d_{2}\}$.
We work in the rigid $4$-monoid $A^{\ast} \times B^{\ast} \times C^{\ast} \times D^{\ast}$.
Let $x = ac$ and $y = bd$ where $a \in A^{\ast}$, $b \in B^{\ast}$, $c \in C^{\ast}$ and $d \in D^{\ast}$.
Then $xu = yv$ where $u = y$ and $v = x$ since, in fact, $x$ and $y$ commute.
We choose the $1$-projection of $v$ which is the element $a$.
Let $\bar{a}$ be any element of $A^{\ast}$ which is prefix incomparable with $a$ (which makes sense since we are working in a free monoid).
Observe that $x\bar{a} = a\bar{a}c$ and $y\bar{a} = \bar{a}bd$.
The only way for $x\bar{a}$ and $y\bar{a}$ to be comparable is for $a\bar{a}$ and $\bar{a}$ to be prefix comparable
but this is impossible.}
\end{example}

The following was motivated by \cite{KR}.

\begin{proposition}\label{prop:chocolate} 
Let $S$ be a right rigid $k$-monoid in which every alphabet has cardinality at least 2.
Then $S$ is effective;
that is, if $x,y \in S$ are distinct elements, there exists $c \in S$ such that $xc$ and $yc$ are incomparable.
\end{proposition}
\begin{proof} We begin by proving a special case that does almost all the work.
Let $x,y \in S$, $x \neq y$ such that $d(x) \wedge d(y) = \mathbf{0}$.
We shall prove that there is an element $c \in S$ such that $xc$ and $yc$ are not comparable.
By right rigidity and Lemma~\ref{lem:tsar}, there are unique elements $u,v \in S$ such that
$xu = yv$ where $d(u) = d(y)$ and $d(v) = d(x)$.
Observe that $u \neq v$ since otherwise we could cancel them and get $x = y$.
Also, we cannot have both $x$ and $y$ equal to the identity since then $x = y$.
Now suppose that exactly one of them, say $y$, was equal to the identity;  accordingly $x \neq 1$.
In fact, then, for any $x$ we have $d(x) \wedge d(1) = \mathbf{0}$.
We have to find an element $a$ such that $xa \neq a$ but this will be true for any element $a \in S$
since $x \neq 1$.
In what follows, therefore, we can assume that $x \neq 1$ and $y \neq 1$.
Observe that since $d(x) \wedge d(y) = \mathbf{0}$, if the $i$-projection of $x$ is not the identity
then the $i$-projection of $y$ is the identity, and vice versa. 
Let $i$ be any value $1 \leq i \leq k$ such that 
the $i$-projection of $x$ is not trivial.
Now $d(v) = d(x)$.
It follows that $v$ has a non-trivial $i$-projection.
Write $v = bv'$ where $b \in X_{i}^{\ast}$ is the $i$-projection and $v'$ has only a trivial $i$-projection.
Let $c$ be any element of $X_{i}^{\ast}$ which is prefix incomparable with $b$
which we can do this since the cardinality of $X_{i}$ is at least 2 and we are working in a free monoid.
In fact, $b$ and $c$ are therefore incomparable by Lemma~\ref{lem:levi}.
We now consider the elements $xc$ and $yc$.
If they are incomparable we are done.
So, assume that they are comparable and we shall derive a contradiction.
Then $xcp = ycq$, for some $p,q \in S$,
which we regard as $x(cp) = y(cq)$.
By Lemma~\ref{lem:tsartwo}, there is a unique element $t$ such that
$cp = ut$ and $cq = vt$.
We focus on the second equation $cq = vt$.
We can write this $cq = bv't$ which says that $b$ and $c$ are comparable which is a contradiction.
It follows that $xc$ and $yc$ are not comparable.

The general case, where we do not assume that $d(a) \wedge d(b) = \mathbf{0}$, is proved using the same argument as \cite[Remark~3.2]{LS2010}.
Now let $x,y \in S$, where $x \neq y$, be arbitrary elements, where $d(x) \wedge d(y) = \mathbf{m} \neq \mathbf{0}$.
Since $\mathbf{m} \leq d(x), d(y)$ we can write $x = u_{1}x'$ and $y = u_{2}y'$ where $d(u_{1}) = \mathbf{m} = d(u_{2})$.
Suppose first that $u_{1} \neq u_{2}$.
Suppose that $x$ and $y$ were comparable.
Then $xp = yq$ for some $p,q \in S$.
But then $u_{1}x'p = u_{2}y'q$ where $d(u_{1}) = d(u_{2})$.
By Lemma~\ref{lem:levi}, we would then have $u_{1} = u_{2}$ which is a contradiction.
It follows that $x$ and $y$ are already not comparable.
Assume, therefore, that $u_{1} = u_{2}$.
If $x' = y'$ then we would have $x = y$.
It follows that $x' \neq y'$.
Observe that $d(x') \wedge d(y') = \mathbf{0}$.
By the first part of the proof there is therefore an element $c$ such that $x'c$ and $y'c$ are incomparable.
It follows that $xc$ and $yc$ are incomparable.
\end{proof}

We can rephrase the above result in terms of the $0$-bisimple inverse monoid $\mathsf{B}(S)$.
The following is immediate by Lemma~\ref{lem:tom}.

\begin{lemma}\label{lem:nosix} Let $S$ be a right rigid $k$-monoid.
Then the following are equivalent:
\begin{enumerate}
\item For each non-idempotent element $xy^{-1}$ there is an element $x_{1}y_{1}^{-1} \leq xy^{-1}$ such that
$x_{1}y_{1}^{-1}$ has disjoint domain and range.

\item If $x,y \in S$ are distinct elements, then there exists $c \in S$ such that $xc$ and $yc$ are incomparable;
thus, $S$ is effective.
\end{enumerate}
\end{lemma}

The following is now immediate by Proposition~\ref{prop:chocolate} and Lemma~\ref{lem:as}.

\begin{corollary}\label{cor:doggy} 
The \'etale groupoid of a right rigid $k$-monoid in which every alphabet has cardinality at least 2
is effective and so aperiodic.
\end{corollary}

\begin{lemma}\label{lem:cf} Let $S$ be a right rigid $k$-monoid in which each of the associated alphabets has cardinality at least two.
Then the $0$-bisimple inverse monoid $\mathsf{B}(S)$ is congruence free.
\end{lemma}
\begin{proof} We use \cite[Theorem~1.11, Lemma~1.10]{Munn1974} 
and prove that $\mathsf{B}(S)$ is $0$-disjunctive and fundamental;
this is enough because it is $0$-bisimple.

We first prove the following simple result;
we are grateful to the referee for their nice proof.
Let $p \in S$ be a non-trivial element.
Then there is an element $q \in S$ such that $p$ and $q$ are not comparable.
Choose $1 \leq i \leq k$ such that $p = ap'$ where $a \in X_{i}$.
Since $X_{i}$ contains at least 2 elements, there is $b \in X_{i}$ such that $b \neq a$.
Clearly, $a$ and $b$ are incomparable and we may put $q = b$.

We apply this result to the structure of $\mathsf{B}(S)$.
Let $bb^{-1} < aa^{-1}$.
Then $b = ap$ for some $p \in S$.
By the result above, there is a $q \in S$ such that $p$ and $q$ are incomparable.
We claim that the idempotents $bb^{-1}$ and $aq(aq)^{-1}$ are orthogonal.
Suppose not.
Then $b$ and $aq$ are comparable.
Thus $ap$ and  $aq$ are comparable.
Thus $apu = aqv$ for some $u,v \in S$.
Thus $pu = qv$ which implies that $p$ and $q$ are comparable, which is a contradiction.
We have therefore proved that $\mathsf{B}(S)$ is $0$-disjunctive.

We shall prove that $\mathsf{B}(S)$ is fundamental but we prove a more general result first.
Let $a$ be an element in an inverse semigroup that there exists $b \leq a$ such that $bb^{-1} \perp b^{-1}b$.
We prove that it cannot commute with all idempotents.
Suppose it did.
Since $b \leq a$ we have that $b = ab^{-1}b = bb^{-1}a$.
Thus $a^{-1}ab^{-1}b = a^{-1}bb^{-1}a$.
But $b^{-1}b \leq a^{-1}a$ and so
we get $b^{-1}b = a^{-1}bb^{-1}a$.
We now use the fact that $a$ commutes with all idempotents
to deduce $b^{-1}b = a^{-1}abb^{-1}$.
It follows that $b^{-1}b \leq bb^{-1}$ which is impossible because these idempotents are supposed to be orthogonal.
Now let $xy^{-1}$ be a non-idempotent element of  $\mathsf{B}(S)$.
Then $x \neq y$.
By Proposition~\ref{prop:chocolate},
$S$ is effective.
Thus there is a $c \in S$ such that $xc$ and $yc$ are incomparable.
Thus $(xc)(yc)^{-1} \leq xy^{-1}$ wherethe domain and range of $(xc)(yc)^{-1}$ are orthogonal.
It follows by our result above that $xy^{-1}$ cannot commute with all idempotents.
We have therefore proved that $\mathsf{B}(S)$ is fundamental.
Thus  $\mathsf{B}(S)$ is congruence-free.
\end{proof}

\begin{remark}\label{rem:periodicty}
{\em Determining whether a $k$-monoid is aperiodic in general is complex,
see \cite{RS2007, DY2009a, LS2010},
but the right rigidity condition is an easy one to check.}
\end{remark}

The question of whether $k$-monoids can be embedded in groups is discussed in \cite{PQR}.
They construct a counterexample \cite[Example~7.1]{PQR} of a $2$-monoid that cannot be embedded into a group.
Their example is evidently not rigid.

\begin{example}{\em This example was constructed by Benjamin Steinberg and communicated to the second author.
Consider the $2$-monoid $S$ with alphabets $\{e_{1}, e_{2}, e_{3}, e_{4}\}$ and $\{f_{1}, f_{2}, f_{3}, f_{4}\}$
defined by means of the following relations:
$$ 
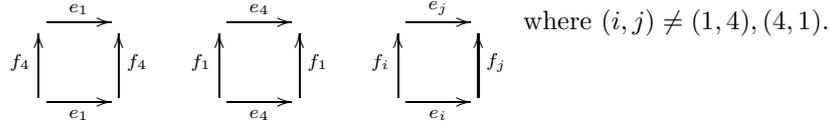
\xymatrix{
 \ar@{->}[r]^{e_{1}} &  \\
   \ar@{->}[u]^{f_{4}}  \ar@{->}[r]_{e_{1}} & \ar@{->}[u]_{f_{4}}  }
\quad 
\xymatrix{
 \ar@{->}[r]^{e_{4}} &  \\
   \ar@{->}[u]^{f_{1}}  \ar@{->}[r]_{e_{4}} & \ar@{->}[u]_{f_{1}}  }
\quad 
\xymatrix{
 \ar@{->}[r]^{e_{j}} &  \\
   \ar@{->}[u]^{f_{i}}  \ar@{->}[r]_{e_{i}} & \ar@{->}[u]_{f_{j}}  }
\text{ where } (i,j) \neq (1,4), (4,1).
$$
This is not a rigid monoid.
Suppose that $S$ could be embedded in a group.
Then for $(i,j) \neq (1,4), (4,1)$, we would have $e_{i}^{-1}f_{i} = f_{j}e_{j}^{-1}$.
Thus $e_{1}^{-1}f_{1} = f_{2}e_{2}^{-1} = e_{3}^{-1}f_{3} = f_{4}e_{4}^{-1}$.
It follows that $e_{1}f_{4} = f_{1}e_{4}$ but we also have $e_{1}f_{4} = f_{4}e_{1}$.
Thus $f_{1}e_{4} = f_{4}e_{1}$ which is an equation in $S$ that contradicts the UFP.}
\end{example}

The following theorem is not used in this paper but is of independent interest.

\begin{theorem}\label{them:embedding} 
Every rigid $k$-monoid can be embedded in a group.
\end{theorem}
\begin{proof} Let $S$ be a rigid $k$-monoid with alphabets $(X_{1}, \ldots, X_{k})$.
The monoid $S$ is defined by relations of the form $xy = y_{1}x_{1}$
where $x, x_{1} \in X_{i}$ and $y, y_{1} \in X_{j}$ and $i \neq j$. 
This relation can be represented by means of the following square
$$ 
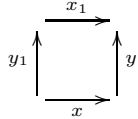
\xymatrix{
 \ar@{->}[r]^{x_{1}} &  \\
   \ar@{->}[u]^{y_{1}}  \ar@{->}[r]_{x} & \ar@{->}[u]_{y}  }
$$
The fact that the monoid is rigid means that each corner of this square uniquely determines the opposite corner uniquely.
For each set $X_{i}$, construct a new set $X_{i}'$ which is in bijective correspondence with it;
the map is $x \mapsto x'$ where $x \in X_{i}$ and $x' \in X_{i}'$.
The free group $G_{i}$ on the set $X_{i}$ is the free monoid on the set $X_{i} \cup X_{i}'$
factored out by the congruence generated by $x'x = 1 = xx'$ where $x \in X_{i}$.
Observe that $X_{i}^{\ast} \subseteq G_{i}$.
We regard $G_{i}$ as the set of reduced words.
We now define another monoid $S'$ as follows.
Its generators will be the set $X_{1} \cup X_{1}' \cup \ldots \cup X_{k} \cup X_{k}'$
subject to the following relations:
\begin{itemize}
\item[{\rm (1)}] $xy = y_{1}x_{1}$, the monoid relations of $S$.
\item[{\rm (2)}] $y'x' = x_{1}'y_{1}'$.
\item[{\rm (3)}] $x'y_{1} = yx_{1}'$.
\item[{\rm (4)}] $y_{1}'x = x_{1}y'$.
\end{itemize}
It follows that $S'$ is a $k$-monoid;
observe that the compatibility conditions hold because they hold in $S$.
Thus each element of $S'$ can be written uniquely as a product of elements in the free monoids 
$(X_{1} \cup X_{1}')^{\ast} \ldots (X_{k} \cup X_{k}')^{\ast}$.
Define $T$ to be the quotient of $S'$ factored out by the congruence generated by all relations of the form:
\begin{itemize}
\item[{\rm (5)}] $xx' = 1 = x'x$ where $x \in X_{1} \cup \ldots \cup X_{k}$.
\end{itemize} 
Then $T$ is a group in which 
each element can be written uniquely as a product of elements in the free groups $G_{1} \ldots G_{k}$.
The monoid $S$ injects as a set into $G_{1} \ldots G_{k}$, since $X_{1}^{\ast} \ldots X_{k}^{\ast} \subseteq G_{1} \ldots G_{k}$,    
and embeds as a monoid, since the relations in $T$ restricted to $S$ are exactly the defining relations of $S$.
\end{proof}

\begin{remark}{\em In the case where $S$ is a rigid $k$-monoid, it follows by Theorem~\ref{them:embedding}
that the inverse monoid $\mathsf{B}(S)$ is not merely $E^{\ast}$-unitary but actually {\em strongly} $E^{\ast}$-unitary;
an inverse semigroup with zero $T$ is {\em strongly $E^{\ast}$-unitary} if there is a function $\theta \colon T \setminus \{0\} \rightarrow G$
to a group $G$ such that $\theta (st) = \theta (s)\theta (t)$ whenever $st \neq 0$ and $\theta^{-1}(1)$ consists of only idempotents in $T$ \cite{Lawson2002}.}
\end{remark}

\section{Maximal generalized prefix codes}

At the conclusion to Section~7, we remarked that the structure of the group $\mathscr{G}(S)$ is bound up with the structure of
the maximal generalized prefix codes in $S$.
The key problem is how to construct such codes in order to construct the groups.
The following lemma shows that the properties of (maximal) generalized prefix codes really do generalize well-known properties of (maximal) prefix codes;
see \cite[Section II.4]{BP}.

\begin{lemma}\label{lem:star} Let $M$ be a cancellative monoid.
\begin{enumerate}

\item Let $u \in M$.
Then $$\{x_{1}, \ldots, x_{m}\}$$ is a generalized prefix code if and only if  $$\{ux_{1}, \ldots, ux_{m}\}$$ is a generalized prefix code.

\item Let $X = \{x_{1}, \ldots, x_{m}\}$ be a (resp. maximal) generalized prefix code and let $Y = \{y_{1}, \ldots, y_{n}\}$ be a (resp. maximal) generalized prefix code.
Then $$\{x_{1}, \ldots, x_{m-1}, x_{m}y_{1}, \ldots, x_{m}y_{n}\}$$ is a (resp. maximal) generalized prefix code.
We call this process of enlarging maximal generalized prefix codes in this way an {\em elementary expansion}.

\item Let $C$ be a maximal generalized prefix code.
Let $xD \subseteq C$ where $D$ is also a maximal generalized prefix code.
Then $C' = C \setminus xD \cup \{x\}$ is a maximal generalized prefix code.
We call this process of reducing maximal generalized prefix codes in this way  an {\em elementary reduction}.

\item If $X$ and $Y$ are maximal generalized prefix codes so too is $XY$.

\item Let $X$ and $Y$ be maximal generalized prefix codes.
Then $XM \subseteq YM$ if and only if $X$ is obtained from $Y$ by a finite sequence of elementary expansions.

\end{enumerate}
\end{lemma}
\begin{proof} (1) Suppose that $\{x_{1}, \ldots, x_{m}\}$ is a generalized prefix code.
If $\{ux_{1}, \ldots, ux_{m}\}$ is not a generalized prefix code then there exist $a,b \in M$ such that $ux_{i}a = ux_{j}b$ for some $i$ and $j$.
But then by cancellation, $x_{i}a = x_{j}b$ which contradicts our assumption.
It follows that $\{ux_{1}, \ldots, ux_{m}\}$ is a generalized prefix code.
Conversely, suppose that  $\{ux_{1}, \ldots, ux_{m}\}$ is a generalized prefix code.
If $\{x_{1}, \ldots, x_{m}\}$ is not a generalized prefix code then there exists $a,b \in M$ such that $x_{i}a = x_{j}b$ for some $i$ and $j$.
But then $ux_{i}a = ux_{j}b$ which contradicts our assumption.
It follows that $\{x_{1}, \ldots, x_{m}\}$ is a generalized prefix code.

(2) By Part (1), we know that $\{x_{m}y_{1}, \ldots, x_{m}y_{n}\}$ is a generalized prefix code.
Let  $1 \leq i \leq m-1$ and $1 \leq j \leq n$
and
suppose that $x_{i}M \cap x_{m}y_{j}M \neq \varnothing$.
But then $x_{i}$ and $x_{m}$ are comparable and so this cannot happen.
It follows that $\{x_{1}, \ldots, x_{m-1}, x_{m}y_{1}, \ldots, x_{m}y_{n}\}$ is a generalized prefix code.

Now suppose that $X$ and $Y$ are maximal.
We show that our code is a maximal generalized prefix code.
Suppose not.
Then there is an element $x \in M$ which is incomparable with all of the elements in $\{x_{1}, \ldots, x_{m-1}, x_{m}y_{1}, \ldots, x_{m}y_{n}\}$.
Now, $X$ is a maximal generalized prefix code.
So, either $x_{i}M \cap xM \neq \varnothing$ for $1 \leq i \leq m-1$ or $x_{m}M \cap xM \neq \varnothing$.
It follows that $x_{i}M \cap xM = \varnothing$ and $x_{m}M \cap xM \neq \varnothing$.
Thus $x_{m}a = xb$ for some $a,b \in M$.
Now $Y$ is a maximal generalized prefix code.
By relabelling if necessary, we can assume that $y_{1}M \cap aM \neq \varnothing$.
Thus $y_{1}c = ad$ for some $c,d \in M$.
It follows that $x_{m}y_{1}c = xbd$ and so $x_{m}y_{1}M \cap xM \neq \varnothing$.
This is a contradiction.

(3) We prove first that $C' = C \setminus xD \cup \{x\}$ is a generalized prefix code.
Suppose that there is $c \in C \setminus xD$ such that $cu = xv$ for some $u,v \in S$ (we shall get a contradiction).
Now $D$ is a maximal generalized prefix code and so $da = vb$ for some $d \in D$ and $a,b \in S$.
It follows that $cub = xvb$ and so $cub = (xd)a$ but $xd \in C$ and $c \in C$ are not comparable.
Thus $C'$ is a generalized prefix code.
It only remains to prove that it is maximal.

Let $s \in S$.
Then $su = cv$ for some $c \in C$ and $u,v \in S$.
Suppose that $c \in xD$.
Then $c = xd$ for some $d \in D$.
Thus $su = xdv$ and so $s$ is comparable with $x$.
Suppose that $c \notin xD$.
Then $c \in C \setminus xD$ and we are done.

(4) We prove first that no two elements of $XY$ are comparable.
Suppose that $xy$ and $x'y'$ are comparable where $x,x' \in X$ and $y,y' \in Y$.
Then $xyu = x'y'v$ for some $u,v \in S$.
This implies that $x$ and $x'$ are comparable and so, by assumption, $x = x'$.
It follows that $yu = y'v$.
This implies that $y$ and $y'$ are comparable and so, by assumption, $y = y'$.
It follows that $xy = x'y'$.
We now prove that $XY$ is maximal.
Let $s \in S$.
Then since $X$ is maximal there are $u,v \in S$ and $x \in X$ such that $su = xv$.
Since $Y$ is maximal there are $a,b \in S$ and $y \in Y$ such that $va = yb$.
Thus $s(ua) = xva = xy(b)$.
It follows that $s$ is comparable to an element of $XY$. 

(5) We look at the easy direction first.
Suppose that $X$ is obtained from $Y$ by one application of expansion.
Then $X = Y \setminus \{y\} \cup yZ$ where $y \in Y$ and both $Y$ and $Z$ are maximal generalized prefix codes.
By part (2), we have that $X$ is a maximal generalized prefix code.
We have that 
$$XM =  (Y \setminus \{y\} \cup yZ)M = (Y \setminus \{y\})M \cup (yZ)M \subseteq YM \cup yM \subseteq YM.$$
The general case follows by induction.
We now prove the converse.
Let $XM \subseteq YM$.
For each $x \in X$ we can find a $y \in Y$ and an $m \in M$ such that $x = ym$.
Let $Y = \{y_{1}, \ldots, y_{n}\}$.
We may therefore write $X  = \bigcup_{i=1}^{m} y_{i}Z_{i}$ 
where, by suitable relabelling if necessary, $m \leq n$.
We show that in fact $m = n$.
Suppose not.
Then $y_{m+1}$ is not included.
However, $X$ is a maximal generalized prefix code.
It follows that for some $i \leq m$ we have that $y_{i}za = y_{m+1}b$ where $z \in Z_{i}$.
But this cannot happen since $Y$ is a maximal generalized prefix code.
It follows that $m = n$.
Next, we prove that each $Z_{i}$ is a maximal generalized prefix code.
First, we prove that each $Z_{i}$ is a generalized prefix code.
Suppose not.
Then there are $z,z' \in Z_{i}$ such that $za = z'b$.
It follows that $y_{i}za = y_{i}z'b$.
But $y_{i}z  = x$ and $y_{i}z' = x'$ where $x,x' \in X$.
By assumption, $x$ and $x'$ are incomparable.
It follows that $Z_{i}$ is an generalized prefix code.
We now prove that it is a maximal generalized prefix code.
Suppose not.
Then there exists $w \in M$ such that $Z_{i}M \cap wM = \varnothing$.
It follows that $Z_{i} \cup \{w\}$ is a generalized prefix code.
Thus $y_{i}Z_{i} \cup \{y_{i}w\}$ is a generalized prefix code.
But $X$ is maximal generalized prefix code.
Thus there is a $y_{j}z \in y_{j}Z_{j}$, where $j \neq i$, such that
$y_{j}za = y_{i}wb$.
But this cannot happen since $i \neq j$ and the set $Y$ is a generalized prefix code.
\end{proof}

We highlight two operations in the above lemma that generate new maximal generalized prefix codes from old:
\begin{enumerate}

\item {\em Elementary expansion.} Let $X$ be a maximal generalized prefix code, let $x \in X$ and let $Y$ be a maximal generalized prefix code.
Then $X\setminus \{x \} \cup xY$ is a maximal generalized prefix code. 

\item {\em Elementary reduction.} Let $C$ be maximal generalized prefix code and let $xD \subseteq C$ where $D$ is a maximal generalized prefix code.
Then $C \setminus xD \cup \{x\}$ is a maximal generalized prefix code.
\end{enumerate}
Observe that these two operations are the inverse of each other.

We need a source of examples of maximal generalized prefix codes.
Let $S$ be a $k$-monoid with $k$ alphabets $(X_{1}, \ldots, X_{k})$.
Each $X_{i}^{\ast}$ is a free monoid by part (2) of Lemma~\ref{lem:relations}.
The following tells us that maximal prefix codes in each of the free monoids $X_{i}^{\ast}$
are automatically maximal generalized prefix codes in the ambient $k$-monoid.
Let $S$ be a $k$-monoid.
It is useful to define $d_{i}(s)$ to be the $i$th component of $d(s)$.
It therefore tells us the number of elements of $X_{i}$ that occur in $s$ counting multiplicities.

\begin{lemma}\label{lem:pinko} Let $S$ be a $k$-monoid and let $C \subseteq X_{i}^{\ast}$ be a maximal prefix code in $X_{i}^{\ast}$.
Then $C$ is a maximal generalized prefix code in $S$.
\end{lemma}
\begin{proof} Let $x,y \in C \subseteq X_{i}^{\ast}$.
Suppose that $z = xu = yv$ for some $u,v \in S$.
Define $\mathbf{m} \in \mathbb{N}^{k}$ by
$\mathbf{m}_{i} = d_{i}(z)$ and all other components are $0$.
Then $d(z) = \mathbf{m} + (d(z) - \mathbf{m})$.
By the UFP, there is a unique element $x' \in X_{i}^{\ast}$, such that $d(x') = \mathbf{m}$,
and a unique element $w \in S$, such that $d(w) = d(z) - \mathbf{m}$, where $z = x'w$.
Clearly, $d(x') \geq d(x), d(y)$ and $x'w = xu = yv$.
It follows that $x' = xx_{1} = yy_{1}$, by Lemma~\ref{lem:levi}, for some $x_{1},y_{1} \in X_{i}^{\ast}$.
But this implies that $x$ and $y$ are comparable in $X_{i}^{\ast}$ which means that $x = y$.
Thus $C$ is a generalized prefix code in $S$.

We now prove that $C$ is maximal.
Let $y \in S$.
Choose $u \in S$ such that $d_{i}(yu)$ is greater than or equal to
$d_{i}(x)$ for any $x \in C$.
We can  write $yu = x_{1}y_{1}$ where $x_{1} \in X_{i}^{\ast}$ and $y_{1}$ has no $i$-component.
Now $x_{1} = ct$ for some $c \in C$ and $t \in S$, since $C$ is a maximal prefix code in $X_{i}^{\ast}$.
Thus $yu = cty_{1}$.
It follows that $y$ is comparable with some element of $C$.
\end{proof}

\noindent
{\bf Definition.} Let $S$ be a $k$-monoid with alphabets $X_{1}, \ldots, X_{k}$.
We define a {\em concrete maximal generalized prefix code} to be one constructed as follows:
\begin{enumerate}

\item $\{1\}$ and $X_{i}$, where $1 \leq i \leq k$, are all concrete maximal generalized prefix codes.

\item If $X$ is a concrete maximal generalized prefix code, $x \in X$, and $Y$ is a concrete maximal generalized prefix code
then $X\setminus \{x \} \cup xY$ is a concrete maximal generalized prefix code. 

\item All concrete maximal generalized prefix codes are obtained in a finite number of steps by repeated applications of (1) and (2).

\end{enumerate}

The following is immediate by part (2) of Lemma~\ref{lem:star} and Lemma~\ref{lem:pinko}.

\begin{corollary}\label{cor:dash} Let $S$ be a $k$-monoid.
Then every concrete maximal generalized prefix code is a maximal generalized prefix code.
\end{corollary}

We shall now begin an analysis of the maximal generalized prefix codes $C_{\mathbf{m}}$.

\begin{lemma}\label{lem:winter} Let $\mathbf{m} \in \mathbb{N}^{k}$ be any non-zero element.
Let $i$ be the smallest suffix such that $m_{i} \neq 0$.
Put $X_{i} = \{a_{1}, \ldots, a_{n}\}$
and $\mathbf{m}' = \mathbf{m} - \mathbf{e}_{i}$.  
Then $C_{\mathbf{m}} = \bigcup_{j=1}^{n} a_{j} C_{\mathbf{m}'} = X_{i}C_{\mathbf{m}'}$.
\end{lemma}
\begin{proof} Let $x \in C_{\mathbf{m}}$.
Then $d(x) = \mathbf{m}$.
By assumption, $m_{i} \neq 0$.
Thus $x = a_{j}y$ for a unique element $a_{j} \in X_{i}$ and 
where $d(y) = \mathbf{m} - \mathbf{e}_{i}$.  
It follows that $C_{\mathbf{m}} \subseteq \bigcup_{j=1}^{n} a_{j} C_{\mathbf{m}'}$.
Now let $a_{j}y \in a_{j} C_{\mathbf{m}'}$.
Then $d(a_{j}y) = \mathbf{m}$ and so $a_{j}y \in C_{\mathbf{m}}$.
It follows that $\bigcup_{j=1}^{n} a_{j} C_{\mathbf{m}'} \subseteq C_{\mathbf{m}}$,
and we have therefore proved that the two sets are equal.
\end{proof}

By Lemma~\ref{lem:pinko}, the set $X_{i}$ is a maximal generalized prefix code.
Thus $C_{\mathbf{m}}$ is obtained from $X_{i}$ by a sequence of elementary expansions using the 
maximal generalized prefix codes $C_{\mathbf{m}'}$ where $\mathbf{m}' < \mathbf{m}$.
This process can be repeated and so we have proved the following.

\begin{proposition}\label{prop:spring} Let $S$ be a $k$-monoid.
The maximal generalized prefix code $C_{\mathbf{m}}$ is either $\{1\}$
or is obtained from the maximal generalized prefix codes $X_{1}, \ldots, X_{k}$ by a sequence of elementary expansions.
It follows that such codes are always concrete.
\end{proposition}

Let $C$ be a maximal generalized prefix code.
The {\em extent} of $C$ is defined to be the join of all the sizes of the elements of $C$. 

\begin{proposition}\label{prop:code} Let $S$ be a $k$-monoid.
Let $C$ be a maximal generalized prefix code of extent $\mathbf{m}'$
and let $\mathbf{m} \geq \mathbf{m}'$.
Then $C_{\mathbf{m}}$ can be obtained from $C$ by a sequence of elementary expansions.
\end{proposition}
\begin{proof} Let $C = \{c_{1}, \ldots, c_{n}\}$ with the extent of $C$ being $\mathbf{m}'$.
We shall show how to replace the element $c_{1}$ by a subset all of whose elements have size $\mathbf{m}$.
A similar procedure can then be applied to each of the other elements of $C$ in turn.
If $d(c_{1}) = \mathbf{m}$ then there is nothing to do, 
so that in what follows we may
assume that $d(c_{1}) < \mathbf{m}$.
Let $\mathbf{m} - d(c_{1}) = (s_{1}, \ldots, s_{k})$.
Replace $c_{1}$ by $c_{1}X_{1}^{s_{1}} \ldots X_{k}^{s_{k}}$.
It is immediate that every element of  $c_{1}X_{1}^{s_{1}} \ldots X_{k}^{s_{k}}$
has size $\mathbf{m}$ and $X_{1}^{s_{1}} \ldots X_{k}^{s_{k}}$ is a maximal generalized prefix code by part (4) of Lemma~\ref{lem:star} and Lemma~\ref{lem:pinko}.
Once this procedure has been applied to each element of $C$ call the resulting maximal generalized prefix code $D$.
By construction, every element of $D$ has size $\mathbf{m}$.
We prove that, in fact, $D = C_{\mathbf{m}}$.
Let $x$ be any element of size $\mathbf{m}$.
Then, since $D$ is a maximal generalized prefix code, there is an element $y \in D$ such that $xu = yv$ for some $u,v \in S$.
But $d(x) = d(y)$.
Thus by Lemma~\ref{lem:levi}, we must have that $x = y$ and so $x \in C_{\mathbf{m}}$.
\end{proof}

The above theorem shows that a sequence of elementary expansions can be applied to a maximal generalized prefix code 
to yield a concrete maximal generalized prefix code.

Observe that if $X$ and $Y$ are maximal generalized prefix codes then $XY = \bigcup_{x \in X} xY$.
Thus $XY$ is obtained from $X$ by a sequence of elementary expansions.
It follows that the maximal generalized prefix codes $X_{1}^{s_{1}} \ldots X_{k}^{s_{k}}$ can be obtained by a sequence of elementary expansions
using only the maximal generalized prefix codes $X_{1}, \ldots, X_{k}$.
If we combine this observation with Proposition~\ref{prop:spring} and Proposition~\ref{prop:code},
we obtain the following theorem.

\begin{theorem}\label{them:goodo} Let $S$ be a $k$-monoid with $k$ alphabets $(X_{1}, \ldots,X_{k})$.
Then every maximal generalized prefix code which is not $\{1\}$ is obtained by means of a sequence of elementary expansions
following by a sequence of elementary reductions using only the maximal generalized prefix codes $X_{1}, \ldots, X_{k}$.
\end{theorem}

The above theorem tells us, in particular, that every maximal generalized prefix code can be obtained from a concrete maximal generalized prefix code by 
means of a sequence of elementary reductions.
The obvious question is whether every maximal generalized prefix code is concrete.
We shall show by means of a counterexample that the answer to this question is `no' in general.

\begin{example}\label{ex:prorogue}{\em Let $S = \{a_{1},a_{2}\}^{\ast} \times \{b_{1},b_{2}\}^{\ast} \times \{c_{1},c_{2}\}^{\ast}$, a $3$-monoid
obtained by taking the direct product of three copies of the free monoid on two generators.
We shall write the elements of $S$ as triples $xyz$.
Let 
$$C = \{a_{1}b_{1}, a_{2}c_{1}, b_{2}c_{2}, a_{1}b_{2}c_{1}, a_{2}b_{1}c_{2}\}.$$ 
It is easy to check that this is a generalized prefix code;
a pair of elements $xyz$ and $uvw$ are incomparable if and only if at least one pair of corresponding components are 
prefix-incomparable --- recall that if $M$ is a free monoid, the elements $a$ and $b$ are {\em prefix incomparable} 
if neither $a = bc$ for some $c$ or $b = ac$ for some $c$ or, equivalently, $aM \cap bM = \varnothing$.
It can be proved directly that $C$ is, in fact, a maximal generalized prefix code but we show this by using our results above.
Replace the element $a_{1}b_{1}$ by the elements $a_{1}b_{1}c_{1}$ and $a_{1}b_{1}c_{2}$;
replace the element $a_{2}c_{1}$ by $a_{2}b_{1}c_{1}$ and $a_{2}b_{2}c_{1}$;
finally, replace $b_{2}c_{2}$ by $a_{1}b_{2}c_{2}$ and $a_{2}b_{2}c_{2}$.
Denote by $C'$ the resulting set of eight elements of $S$.
They all have the same size $(1,1,1)$.
It follows that $C' = C_{(1,1,1)}$ and so is a maximal generalized prefix code by Lemma~\ref{lem:homogeneous}.
Now, $C$ is obtained from $C_{(1,1,1)}$ by a sequence of elementary reductions: 
the pair $a_{1}b_{1}c_{1}$ and $a_{1}b_{1}c_{2}$ is replaced by $a_{1}b_{1}$;
the pair $a_{2}b_{1}c_{1}$ and $a_{2}b_{2}c_{1}$ is replaced by $a_{2}c_{1}$;
the pair $a_{1}b_{2}c_{2}$ and $a_{2}b_{2}c_{2}$ is replaced by $b_{2}c_{2}$.
It follows by Lemma~\ref{lem:star} that $C$ is a maximal generalized prefix code.
However, it is not concrete since it contains the elements $b_{2}c_{2}$, $a_{2}c_{1}$ and $a_{1}b_{1}$,
which omit an $a$-coordinate, a $b$-coordinate and a $c$-coordinate, respectively, which means that part (2)
of the definition of concreteness cannot be applied.}
\end{example}

Our counterexample above was constructed for a rigid $3$-monoid.
We now show that there is no counterexample in the case of right rigid $2$-monoids.

\begin{lemma}\label{lem:spa} Let $S$ be a $k$-monoid.
Let $C$ be a maximal generalized prefix code such that
$$C = \bigcup_{x \in Y} xZ_{x}$$
where $Y \subseteq X_{i}$
for some $1 \leq i \leq k$.
Then $Y = X_{i}$ and each $Z_{x}$ is a maximal generalized prefix code.
\end{lemma}
\begin{proof} Suppose first that $Y \neq X_{i}$.
Let $x' \in X_{i} \setminus Y$.
Then, by assumption, $x'u = xz$ where $z \in Z_{x}$ for some $x \in Y$.
It follows by Lemma~\ref{lem:levi} that $x' = x$, which is a contradiction. 
We have therefore proved that $Y = X_{i}$.
We prove that $Z_{x}$ is a generalized prefix code.
Let $a,b \in Z_{x}$ and suppose that $au = bv$ for some $u,v \in S$.
Then $(xa)u = (xb)v$.
But $xa, xb \in C$ and so $xa = xb$ giving $a = b$.
We now prove that $Z_{x}$ is maximal.
Let $a \in S$ be arbitrary.
Then $xa$ must be comparable with an element $c \in C$.
Thus $xau = cv$ where $u,v \in S$.
Now $c$ must contain at least one element of $X_{i}$ by assumption.
It follows that we can write $c = yc'$ for some $c' \in S$ and $y \in X_{i}$.
Thus $xau = yc'v$.
But, by assumption, $d(x) = d(y)$ and so $x = y$
by Lemma~\ref{lem:levi}.
From $c = xc'$, we get that $c' \in Z_{x}$. 
Also $xau = xc'v$ and so by cancellation, $au = c'v$. 
It follows that $Z_{x}$ is a maximal generalized prefix code.
\end{proof}

The following is immediate by Lemma~\ref{lem:tsar} since if $x \in X_{1}^{\ast}$ and $y \in X_{2}^{\ast}$
then $d(x) \wedge d(y) = \mathbf{0}$.

\begin{lemma}\label{lem:alina1} Let $S$ be a right rigid $2$-monoid with alphabets $X_{1}$ and $X_{2}$.
Let
$x \in X_{1}^{\ast}$ and $y \in X_{2}^{\ast}$ be two non-empty strings.
Then there are elements $u,v \in S$ such that $xu = yv$.
\end{lemma}

We can now prove the following.

\begin{theorem}\label{them:alina2} Let $S$ be a right rigid $2$-monoid with alphabets $X_{1}$ and $X_{2}$.
Then every maximal generalized prefix code $C \neq \{1\}$ in $S$ is obtained from
the maximal prefix codes in $X_{1}^{\ast}$ or $X_{2}^{\ast}$ by a sequence of expansions.
In particular, every maximal generalized prefix code in $S$ is concrete.
\end{theorem}
\begin{proof} In this proof, we shall use the term {\em homogeneous} to mean 
a non-identity element of $X_{1}^{\ast}$ or a non-identity element of $X_{2}^{\ast}$.
Let $D \neq \{1\}$ be any maximal generalized prefix code in $S$.
There are two possibilities:
either $D$ contains a homogeneous element or it doesn't.
Suppose that $D$ contains a homogeneous element $x$.
Then, without loss of generality, we can assume that $x \in X_{1}^{\ast}$.
By Lemma~\ref{lem:alina1}, there can be no elements of $D$ that belong to $X_{2}^{\ast}$.
It follows that every element of $D$ contains an element of $X_{1}$.
Thus by Lemma~\ref{lem:spa}, we can write 
$D = \bigcup_{x \in X_{1}} xW_{x}$
where each $W_{x}$ is a maximal generalized prefix code. 
Now suppose that $D$ contains no homogeneous element.
Thus every element of $D$ contains at least one letter from both of the alphabets $X_{1}$ and $X_{2}$.
It follows again that $D = \bigcup_{x \in X_{i}} xW_{x}$
for some $1 \leq i \leq 2$
where each $W_{x}$ is a maximal generalized prefix code.
Thus, in either case, we can write 
$D = \bigcup_{x \in X_{i}} xW_{x}$
for some $1 \leq i \leq 2$.
But each $W_{x}$ is either equal to $\{1\}$ or is a non-trivial maximal generalized prefix code
but with elements strictly shorter than in $D$.
It follows that an inductive argument can now be applied.
\end{proof}


\section{Constructing examples of our groups}

In this section, we shall prove that the groups we have defined contain as special cases the higher-dimensional
Thompson groups introduced by Matt Brin \cite{Brin} and subsequently widely studied \cite{Brin2010, BC, FMWZ, HM,BB, DM, BL}.
Brin's groups arise from the $k$-monoids which are finite direct products of free monoids.
That his groups arise in this way has also been proved by Birget \cite{Birget2019} working solely with finite direct products of free monoids.
We therefore begin by establishing a dictionary between our paper and that of Birget \cite{Birget2019};
specifically, Section~1 and Section~2 of his paper.
Observe that he is working solely with $k$-monoids of the form $S = nA^{\ast}$ --- that is the direct product of $n$ copies of the free monoid $A^{\ast}$.
Let $S$ be any $n$-monoid.
Define $u \leq_{\mbox{\scriptsize int}} v$ if and only if there exists an $x$ such that $ux = v$.
This is an order relation since $S$ is conical.
It is called the {\em initial factor order}.
in the case where $S$ is a free monoid, it is the usual prefix order
where  $u \leq_{\mbox{\scriptsize pref}} v$ if and only if there exists an $x$ such that $ux = v$.
Given $u$ and $v$ we are also interested in when there is an element $z$ such that
$u \leq_{\mbox{\scriptsize pref}} z$ and $v \leq_{\mbox{\scriptsize pref}} z$.
This is equivalent to requiring that there exist $x$ and $y$ such that $ux = vy = z$;
equivalently, $uS \cap vS \neq \varnothing$.
By \cite[Lemma~2.5]{Birget2019}, we have that $u,v \in nA^{\ast}$ have a join if and only if  $uS \cap vS \neq \varnothing$.
It follows that `joinless' used in \cite{Birget2019} is equivalent to our `incomparable'.
It follows that our definition of a 'maximal generalized prefix code' is equivalent to Birget's definition of a `maximal joinless code'.
With this understanding, Birget's inverse monoid $n\mathcal{RI}_{A}^{\scriptsize fin}$ is our inverse monoid $\mathsf{P}(nA^{\ast})$.
His definition of the congruence $\equiv_{\scriptsize end}$ \cite[Definition 2.23]{Birget2019} looks different from that of the minimum group congruence we use
but this is only apparent.
Recall that $\mathsf{P}(nA^{\ast})$ is $E$-unitary by Proposition~\ref{prop:tizer} and Lemma~\ref{lem:birget}.
By Proposition~\ref{prop:E}, in an $E$-unitary inverse semigroup $\sigma\, = \, \sim$.
We prove (the well-known) result that in an $E$-unitary inverse semigroup $a \sim b$ if and only if $ab^{-1}b = ba^{-1}a$.
Suppose that $ab^{-1}b = ba^{-1}a$.
Then $ba^{-1} = ab^{-1}ba^{-1}$ and so is an idempotent.
It follows that $ab^{-1}$ is an idempotent.
Now observe that $a^{-1}ba^{-1}a = a^{-1}ab^{-1}b$.
Thus $a^{-1}b$ is above an idempotent and so, since the inverse semigroup is $E$-unitary,
it is itself an idempotent.
We have therefore proved that $a \sim b$.
Suppose now that $a \sim b$.
Since $ab^{-1}$ is an idempotent it follows that $ab^{-1}b \leq b$.
It is immediate that $ab^{-1}b \leq a$.
Now let $z \leq a,b$.
Then $z \leq ab^{-1}b, ba^{-1}a$.
It follows that $ab^{-1}b = a \wedge b$.
A similar argument shows that $ba^{-1}a = a \wedge b$.
We have therefore proved that $ab^{-1}b = ba^{-1}a$.
Observe that his \cite[Lemma~2.11]{Birget2019} is a special case of our Lemma~\ref{lem:star}.

The papers \cite{Birget} and \cite{Lawson2007, Lawson2007b} simply prove the following.

\begin{proposition}\label{prop:water} 
$$\mathscr{G}(A_{n}^{\ast}) = G_{n,1}.$$
\end{proposition}

Before we can state our next result, we shall need  to describe maximal generalized prefix codes in more geometrical language.

\begin{remark}\label{rem:guardian}
{\em 
Let $S$ be a $k$-monoid with associated $k$-alphabets $(X_{1}, \ldots, X_{k})$.
Let $|X_{i}| = n_{i} \geq 2$ but finite.
Write $X_{i} = \{x_{i}^{1}, \ldots, x_{i}^{n_{i}}\}$ ---
observe that here we use superscripts as labels not powers.
Let $I^{k}$ be the product $[0,1)^{k}$.
We shall show how maximal generalized prefix codes in $S$ divide up the $k$-cube $I^{k}$.
Each element $s \in S$ can be written uniquely as $s = x_{1} \ldots x_{k}$ where $x_{i} \in X_{i}^{\ast}$.
An elements of $I^{k}$ is a $k$-tuple  and we associate the $i$th-coordinate with the alphabet $X_{i}$.
Let $\mathbf{m} \in \mathbb{N}^{k}$.
We show first how the elements of $C_{\mathbf{m}}$ correspond to a special partition of $I^{k}$.
We shall assume first that no component of $\mathbf{m}$ is zero.
We calculate the cardinality of $C_{\mathbf{m}}$.
This is $n_{1}^{m_{1}} \ldots n_{k}^{m_{k}}$.
Given $n_{i} \geq 2$ divide the interval $[0,1]$ into $n_{i}$ right-open, left-closed intervals each of length $\frac{1}{n_{i}}$.
The interval $[\frac{j-1}{n_{i}}, \frac{j}{n_{i}}]$ is associated with  the element $x_{i}^{j}$ where $1 \leq j \leq n_{i}$.
More generally, the element $x_{i}^{j}$ applied to a left-closed, right-open interval $J$ does the following:
divide $J$ equally into $n_{i}$  left-closed, right-open intervals and pick out the $j-1$th.
An element of $X_{i}^{\ast}$ therefore determines a left-closed right-open interval in $[0,1)$.
If that element is $1$ then this is the whole of $[0,1)$.
If it is an individual letter $x_{i}^{j}$ then it picks out the $j-1$th as above.
If it is a sequence of elements of $X_{i}^{\ast}$ then we apply then one at a time from left-to-right.
If the sequence has length $m_{i}$ then the left-closed, right-open interval will have length $\frac{1}{n_{i}^{m_{i}}}$.
It follows that each element $s = (x_{1}, \ldots, x_{k}) \in C_{\mathbf{m}}$ will correspond to a $k$-parallelepiped inside $I^{k}$
whose volume will be $\frac{1}{n_{1}^{m_{1}} \ldots n_{k}^{m_{k}}}$.
We call such a parallelepiped a {\em brick (of size $\mathbf{m}$)}.
In this way, the set $C_{\mathbf{m}}$ leads to a partition of $I^{k}$ into $n_{1}^{m_{1}} \ldots n_{k}^{m_{k}}$ bricks each with the same volume.
We call the partition of $I^{k}$ obtained in this way a {\em uniform partition}.
We can now explain the reduction process in terms of gluing suitable bricks together. 
Let $c_{1}, \ldots, c_{n_{i}}$ be elements of a generalized maximal prefix code.
We say that they are {\em $i$-contiguous} if there exists $c \in S$ such that $c_{j} = c x_{i}^{j}$ where $1 \leq j \leq n_{i}$.
This may be difficult to determine but is easy if the alphabets mutually commute.
We now remove $c_{1}, \ldots, c_{n_{i}}$ from the code but replace them with $c$.
Geometrically, we say that $c$ is obtained from  $c_{1}, \ldots, c_{n_{i}}$ by {\em gluing} them  along the $i$th-component.
A {\em shape} is obtained from bricks by applying the gluing process a finite number of times.
A {\em pattern} is any partition of $I^{k}$ into shapes obtained from a uniform partition by a finite number of gluings.}
\end{remark}

Now consider the $n$-fold direct product $A_{2}^{\ast} \times \ldots \times A_{2}^{\ast}$.
By Proposition~\ref{prop:doodah},
Remark~\ref{rem:guardian},
Section~7,
and \cite{BL}, we have the following.

\begin{proposition}\label{prop:air} 
$$\mathscr{G}((A_{2}^{\ast})^{n}) \cong nV.$$
\end{proposition}

\begin{remark}{\em Observe that our Example~\ref{ex:prorogue} contradicts the statement in the original version of  \cite[Lemma~2.10]{Birget2019}.
This has no bearing on the first two sections of Birget's paper which we refer to here.}
\end{remark}

\section{Concluding remarks}

This section contains sundry results that are interesting but do not fit with the main thrust of the paper.

Maximal generalized prefix codes occupy a special position in the theory as we now show.
Let $S$ be a distributive inverse semigroup.
A congruence $\rho$ on $S$ is said to be {\em additive} if $a \, \rho \, b$ and $c \, \rho \, d$ and $a \, \sim \, c$ and $b \, \sim \, d$ imply that
$(a \vee c) \, \rho \, (b \vee d)$.

\begin{lemma}\label{lem:cola} Let $S$ be a $k$-monoid.
If $X = \{x_{1}, \ldots, x_{m}\}$ is a maximal generalized prefix code then 
$\bigvee_{x \in X} xx^{-1} \leq_{e} 1$.
\end{lemma}
\begin{proof} It is enough to work with idempotents of the form $uu^{-1}$.
By assumption, $uS \cap x_{i}S \neq \varnothing$ for some $i$.
Thus $z = up = x_{i}q$ for some $p,q \in S$.
Then $zz^{-1} \leq uu^{-1}, x_{i}x_{i}^{-1}$.
It follows that $uu^{-1} \wedge x_{i}x_{i}^{-1} \neq 0$.
\end{proof}

The above lemma provides the context for the next result. 

\begin{proposition}\label{prop:essential} Let $S$ be a $k$-monoid.
Let $\rho$ be an additive congruence on $\mathsf{R}(S)$ such that 
$\left( \bigvee_{x \in X} xx^{-1}\right) \, \rho \, 1$ whenever $X$ is a maximal generalized prefix code.
Then $\rho$ is an essential congruence.
\end{proposition}
\begin{proof} Let $\bigvee_{i=1}^{m} x_{i}x_{i}^{-1} \leq_{e} 1$.
Put $\mathbf{n} = \bigvee_{i=1}^{m} d(x_{i})$.
Let $u \in C_{\mathbf{n}}$.
Then $uu^{-1}$ is an idempotent and so, by assumption, there is an $i$ such that
$zz^{-1} \leq uu^{-1}, x_{i}x_{i}^{-1}$.
It follows that $z = up = x_{i}q$ for some $p,q \in S$.
But $d(u) \geq d(x_{i})$.
By  Lemma~\ref{lem:levi}, we have that $u = x_{i}t$ for some $t \in S$.
Thus $uu^{-1} \leq x_{i}x_{i}^{-1}$.
We have therefore proved that
$$\bigvee_{u \in C_{\mathbf{n}}} uu^{-1} \leq \bigvee_{i=1}^{m} x_{i}x_{i}^{-1}.$$
By assumption, $\bigvee_{u \in C_{\mathbf{n}}} uu^{-1} \rho 1$ and so, using the fact that $\rho$ is an additive congruence,
we get that $\bigvee_{i=1}^{m} x_{i}x_{i}^{-1} \rho 1$.
Let $\bigvee_{i=1}^{m} x_{i}x_{i}^{-1} \leq uu^{-1}$.
Then we can write $x_{i} = up_{i}$ for each $1 \leq i \leq m$.
It is routine to check that
$$\bigvee_{i=1}^{m} x_{i}x_{i}^{-1} \leq_{e} uu^{-1} \Leftrightarrow \bigvee_{i=1}^{m} p_{i}p_{i}^{-1} \leq_{e} 1.$$
It is also routine to check that 
$$\bigvee_{i=1}^{m} e_{i} \leq_{e} 1 \Rightarrow a\left( \bigvee_{i=1}^{m} e_{i} \right)a^{-1} \leq_{e} aa^{-1}.$$
We have therefore proved that if
$\bigvee_{i=1}^{m} x_{i}x_{i}^{-1} \leq_{e} uu^{-1}$
then $\bigvee_{i=1}^{m} x_{i}x_{i}^{-1} \rho uu^{-1}$.
Next, one can easily check that
$$\bigvee_{i=1}^{m} e_{i} \leq_{e} \bigvee_{j=1}^{n} f_{j} \Leftrightarrow \bigvee_{i=1}^{m} f_{j} \wedge e_{i} \leq_{e} f_{j}$$
for $1 \leq j \leq n$.
Suppose that 
$\bigvee_{i=1}^{m} e_{i} \leq_{e} \bigvee_{j=1}^{n} f_{j}$.
Then
$\bigvee_{i=1}^{m} f_{j} \wedge e_{i} \rho f_{j}$
for each $1 \leq j \leq n$.
We now use the fact that $\rho$ is an additive congruence to deduce that
$\bigvee_{j=1}^{n} \left( \bigvee_{i=1}^{m} f_{j} \wedge e_{i} \right) \rho \bigvee_{j=1}^{n} f_{j}$.
This simplifies to
$\bigvee_{i=1}^{m} e_{i} \rho \bigvee_{j=1}^{n} f_{j}$.
Suppose that $a \leq_{e} b$ in $\mathsf{R}(S)$.
By Lemma~\ref{lem:phone}, we have that $\mathbf{d}(a) \leq_{e} \mathbf{d}(b)$.
But we have proved that $\mathbf{d}(a) \rho \mathbf{d}(b)$.
Thus $b\mathbf{d}(a) \rho b$.
But $a = b\mathbf{d}(a)$.
We have therefore proved that $a \rho b$.
\end{proof}

For the time being, let $S$ be any cancellative monoid.
For each $a \in S$, define a function $\lambda_{a} \colon S \rightarrow S$ by $\lambda_{a}(x) = ax$.
By left cancellation in $S$, the function $\lambda_{a}$ is injective.
Regarded as a function $\lambda_{a} \colon S \rightarrow aS$, it is a bijection or a partial bijection when viewed as an element of the symmetric inverse monoid $I(S)$.
We denote its inverse (as a partial bijection) by $\lambda_{a}^{-1}$.
Define $\Sigma (S)$, the {\em inverse hull} of $S$, to be the inverse submonoid of $I(S)$ generated by all the elements $\lambda_{a}$.
Thus $\Sigma (S)$ is the monoid generated by all the elements of the form $\lambda_{a}$ and $\lambda_{a}^{-1}$.
We can massage the elements of $\Sigma (S)$ into a certain shape.
Observe that $\lambda_{a}\lambda_{b} = \lambda_{ab}$ and that $\lambda_{a}^{-1}\lambda_{b}^{-1} = \lambda_{ba}^{-1}$.
In addition, $\lambda_{a}^{-1}\lambda_{a}$ is the identity map on $S$.
We can therefore write each element of $\Sigma (S)$ in the form
$$\lambda_{a_{1}}^{-1}\lambda_{a_{2}} \ldots \lambda_{a_{2s-1}}^{-1}\lambda_{a_{2s}}.$$
In general, we cannot say more about the structure of the inverse monoid $\Sigma (S)$.
However, if $S$ is strongly finitely aligned we can say more.
We analyse maps of the form $\lambda_{a}^{-1}\lambda_{b}$.
If $bS \cap aS = \varnothing$ then this map is the empty map.
We therefore assume that $bS \cap aS \neq \varnothing$.
Let $bS \cap aS = \bigcup_{i=1}^{m} c_{i}S$ and  let $c_{i} = ax_{i} = by_{i}$.
Thus  $\lambda_{a}^{-1}\lambda_{b} = \bigsqcup_{i=1}^{m} \lambda_{x_{i}}\lambda_{y_{i}}^{-1}$
which is a disjoint union.
This means that the domains and ranges of the elements of the inverse hull
are finitely generated right ideals generated by incomparable elements.
If we repeatedly apply this result, 
we deduce that every element of the inverse hull can be written in the form
$\bigoplus_{i=1}^{m} \lambda_{x_{i}}\lambda_{y_{i}}^{-1}$.

\begin{remark}
{\em Spielberg's monoid of zig-zags \cite{JS2014} is an inverse hull}.
\end{remark}
 
There is a natural embedding of $\iota \colon \Sigma (S) \rightarrow \mathsf{R}(S)$.
The proof of the following is straightforward and uses Lemma~\ref{lem:line}.

\begin{proposition}\label{prop:butter} Let $S$ be a $k$-monoid.
Let $\alpha \colon \Sigma (S) \rightarrow T$ be any homomorphism to a distributive inverse semigroup $T$.
Then there is a unique morphism of distributive inverse semigroups $\beta \colon \mathsf{R}(S) \rightarrow T$
such that $\iota \beta = \alpha$.
\end{proposition}

The above results means that $\mathsf{R}(S)$ is the {\em distributive completion} of $\Sigma (S)$.
By Proposition~\ref{prop:essential} and Proposition~\ref{prop:butter}, 
it follows that the Boolean inverse monoid $\mathsf{C}(S)$ is what we termed in \cite{LV2019}  the {\em tight completion}
of $\Sigma (S)$ and the groupoid  $\mathcal{G}(S)$ is therefore the {\em tight groupoid} of $\Sigma (S)$.

\begin{remark}
{\em There are parallels and similarities between our paper and \cite{KMN, MN}.
In particular, let $S = A_{n_{1}}^{\ast} \times \ldots \times A_{n_{s}}^{\ast}$, where $n_{i} \geq 2$.
This is an $s$-monoid.
It is convenient to denote the associated alphabets by $X_{1}, \ldots, X_{s}$.
We say that we have {\em $s$ colours}.
Assume that $X_{i}$ contains $n_{i}$ elements.
We write $X_{i} = \{x_{i}^{1}, \ldots, x_{i}^{n_{i}}\}$.
We say that $X_{i}$ has {\em arity} $n_{i}$.
The {\em data} consisting of the number $s$, the list of arities $(n_{1}, \ldots, n_{s})$
and the fact that the elements of $X_{i}$ commute with the elements of $X_{j}$, when $i \neq j$, 
is precisely what is needed to define the $\Omega$-algebras described in \cite{KMN}.
There are also parallels between our maximal generalized prefix codes and the admissible sets used there.
But we do not know if the general class of groups defined in \cite{KMN} is the same as ours.
This is similar to an issue described in \cite{FL}.
It is worth noting that our groups are closely associated with \'etale groupoids, as we have shown.
This is not, a priori, the case with the groups defined in \cite{KMN, MN}.}
\end{remark}



\end{document}